\documentclass[10pt]{amsart}
\usepackage{amsmath,amssymb,amsthm}
\usepackage{enumerate}                
\usepackage{cite}
\usepackage{xspace}                   
\usepackage[margin=1in]{geometry}

\newcommand{\birkhoff}{Birkhoff\xspace}
\newcommand{\lipschitz}{Lipschitz\xspace}
\newcommand{\tychonoff}{Tychonoff\xspace}

\newtheorem{theorem}{Theorem}[section]
\newtheorem*{theorem-nn}{Theorem}

\newtheorem{corollary}[theorem]{Corollary}
\newtheorem{lemma}[theorem]{Lemma}

\newtheorem{proposition}[theorem]{Proposition}
\newtheorem{question}[theorem]{Question}
\newtheorem*{question-nn}{Question}

\theoremstyle{definition}
\newtheorem{definition}[theorem]{Definition}
\newtheorem*{definition-nn}{Definition}
\newtheorem{example}[theorem]{Example}
\newtheorem*{example-nn}{Example}
\theoremstyle{remark}
\newtheorem{remark}[theorem]{Remark}
\newtheorem{case}{Case}

\newcommand{\amalgam}{{\coprod}_{A} G_{\lambda}}
\newcommand{\amalgamH}{{\coprod}_{A} H_{\lambda}}

\newcommand{\concat}{{}^{\frown}}
\newcommand{\env}[1]{\overline{#1}}
\newcommand{\diam}[1]{\mathrm{diam}(#1)}
\newcommand{\dist}{\underline{d}}
\newcommand{\id}{id}
\newcommand{\norm}{\mathrm{N}}
\newcommand{\nsbg}{\mathfrak{N}}
\newcommand{\seg}[2]{[#1,#2]}
\newcommand{\symmet}[4]{\mathrm{Sym}(#1,#2;#3,#4)}
\newcommand{\transferr}[4]{\mathrm{RTran}(#1,#2;#3,#4)}
\newcommand{\transferl}[4]{\mathrm{LTran}(#1,#2;#3,#4)}

\newcommand{\tsi}{tsi }
\newcommand{\word}[1]{\mathrm{Words}(#1)}

\newcommand{\words}{\mathrm{Words}(G)}
\renewcommand{\cong}{\sim}
\newcommand{\sign}[1]{\mathop{\mathrm{sign}}\nolimits(#1)}
\newcommand{\multipliable}{multipliable\xspace}

\begin{document}
\title[Graev metrics on free products]{Graev metrics on free products and HNN extensions} 

\author{Konstantin Slutsky} 
\thanks{Research supported in part by grant no. 10-082689/FNU from Denmark's Natural Sciences Research Council.}
\address{
  Institut for Matematiske Fag\\
  K\o benhavns Universitet\\
  Universitetsparken 5\\
  2100 K\o benhavn \O,\ Denmark \\}
\email{kslutsky@gmail.com}

\begin{abstract}
  We give a construction of two-sided invariant metrics on free products (possibly with
  amalgamation) of groups with two-sided invariant metrics and, under certain conditions, on HNN
  extensions of such groups.  Our approach is similar to the Graev's construction of metrics on free
  groups over pointed metric spaces.
\end{abstract}

\maketitle

\section{Introduction}
\label{sec:introduction}

\subsection{History}
\label{sec:history}

Back in the 40's in his seminal papers \cite{MR0004634,MR0012301} A. Markov came up with a notion of
the free topological group over a completely regular (\tychonoff) space.  This notion gave birth to
a deep and important area in the general theory of topological groups.  We highly recommend an
excellent overview of free topological groups by O. Sipacheva \cite{MR2056625transl}.  Later
M. Graev \cite{MR0038357} gave another proof of the existence of free topological groups over
completely regular spaces.  In his approach Graev starts with a pointed metric space
\( (X,x_{0},d) \) and defines in a canonical way a two-sided invariant metric on
\( F \big( X\setminus \{x_{0}\} \big) \) --- the free group with bases \( X \setminus \{x_{0}\} \).
Moreover, this metric extends the metric \( d \) on \( X \setminus \{x_{0}\} \).  In modern terms,
Graev constructed a functor from the category of pointed metric spaces with \lipschitz maps to the
category of groups with two-sided invariant metrics and \lipschitz homomorphisms.

The topology given by the Graev metric on the free group \( F(X \setminus \{x_{0}\}) \) is, in
general, much weaker than the free topology on \( F \big( X \setminus \{x_{0}\} \big) \).  Since the
early 40's a lot of work was done to understand the free topology on free groups, and some of this
work shed light onto properties of the Graev metrics.

Graev metrics were used to construct exotic examples of Polish groups (see \cite{MR1288299,
  MR2332614,MR2541347}).  For example, the group completion of the free group
\( F(\mathbb{N}^{\mathbb{N}}) \) over the Baire space with the topology given by the Graev metric is
an example of a surjectively universal group in the class of Polish groups that admit compatible
two-sided invariant metrics (see \cite{MR1288299} for the proof).

Once the notion of a free topological group is available, the next step is to construct free
products.  It was made by Graev himself in \cite{MR0036768}, where he proves the existence of free
products in the category of topological groups.  For this he uses, in a clever and unexpected way,
Graev metrics on free groups.  But this time his approach does not produce a canonical metric on the
free product out of metrics on factors.

In this paper we would like to try to push Graev's method from free groups to free products of
groups with and without amalgamation.  As will be evident from the construction, the natural realm
for this approach is the category of groups with two-sided invariant metrics.  To be precise, a
basic object for us will be an abstract group \( G \) with a two-sided invariant metric \( d \) on
it. We recall that \( G \) will then automatically be a topological group in the topology given by
\( d \).  Topological groups that admit a compatible two-sided invariant metric form a very
restrictive subclass of the class of all the metrizable topological groups, but it
includes compact metrizable and abelian metrizable groups.

\subsection{Main results}
\label{sec:main-results}

The paper roughly consists of two parts.  In the first part we show the existence of free products
of groups with two-sided invariant metrics.  Here is a somewhat simplified version of the main
theorem.

\begin{theorem-nn}[Theorem \ref{sec:metrics-amalgams-MAIN-Graev-metric-on-products}]
  Let \( (G_{1},d_{1}) \) and \( (G_{2},d_{2}) \) be groups with two-sided invariant metrics.  If
  \( A < G_{i} \) is a common closed subgroup and \( d_{1}|_{A} = d_{2}|_{A} \), then there is a
  two-sided invariant metric \( \dist \) on the free product with amalgamation
  \( G_{1} *_{A} G_{2} \) such that \( \dist |_{G_{i}} = d_{i} \).  Moreover, if \( G_{1} \) and
  \( G_{2} \) are separable, then so is \( G_{1} *_{A} G_{2} \).
\end{theorem-nn}

Next we address the question of when a two-sided invariant metric can be extended to an HNN
extension.  We obtain the following results.

\begin{theorem-nn}[Theorem \ref{sec:hnn-extens-class-existence-of-hnn-extension}]
  Let \( (G,d) \) be a tsi group, \( \phi : A \to B \) be a \( d \)-isometric isomorphism between
  the closed subgroups \( A, B \).  Let \( H \) be the HNN extension of \( (G,\phi) \) in the
  abstract sense, and let \( t \) be the stable letter of the HNN extension.  If
  \( \diam{A} \le K \), then there is a tsi metric \( \dist \) on \( H \) such that
  \( \dist|_{G} = d \) and \( \dist(t,e) = K \).
\end{theorem-nn}

\begin{theorem-nn}[Theorem \ref{sec:induc-conj-hnn-general-theorem}]
  Let \( G \) be a SIN metrizable group.  Let \( \phi : A \to B \) be a topological isomorphism
  between two closed subgroups.  There exist a SIN metrizable group \( H \) and an element
  \( t \in H \) such that \( G < H \) is a topological subgroup and \( tat^{-1} = \phi(a) \) for all
  \( a \in A \) if and only if there is a compatible tsi metric \( d \) on \( G \) such that
  \( \phi \) becomes a \( d \)-isometric isomorphisms.
\end{theorem-nn}

\subsection{Notations}
\label{sec:notations}

We use the following conventions.  By an interval we always mean an interval of natural numbers, there will be no
intervals of reals in this paper.  An interval \( \{m, m+1, \ldots, n\} \) is denoted by \( [m,n] \).  For a finite set
\( F \) of natural numbers \( m(F) \) and \( M(F) \) denote its minimal and maximal elements respectively.  For two sets
\( F_{1} \) and \( F_{2} \) if \( M(F_{1})<m(F_{2}) \), then we say that \( F_{1} \) is less than \( F_{2} \) and denote this
by \( F_{1} < F_{2} \).

A finite set \( F \) of natural numbers can be represented uniquely as a union of its maximal sub-intervals, i.e.,
there are intervals \( \{I_{k}\}_{k=1}^{n} \) such that
\begin{enumerate}[(i)]
\item \( F = \bigcup_{k} I_{k} \);
\item \( M(I_{k}) + 1 < m(I_{k+1}) \) for all \( k \in \seg{1}{n-1} \).
\end{enumerate}
We refer to such a decomposition of \( F \) as to the \emph{family of maximal sub-intervals}.

By a \emph{tree} we mean a directed graph connected as an undirected graph without undirected cycles and with a
distinguished vertex, which is called the \emph{root} of the tree.  For any tree \( T \) its root
will be denoted by \( \emptyset \).  The \emph{height} on a tree \( T \) is a function \( H_{T} \)
that assigns to a vertex of the tree its graph-theoretic distance to the root.  For example
\( H_{T}(\emptyset) = 0 \) and \( H_{T}(t) = 1 \) for all \( t \in T \setminus \{\emptyset \} \)
such that \( (t,\emptyset) \in E(T) \), where \( E(T) \) is the set of directed edges of \( T \).
We use the word \emph{node} as a synonym for the phrase \emph{vertex of a tree}.
We say that a node \( s \in T \) is a \emph{predecessor} of \( t \in T \), and denote this by \( s \prec t \), if there
are nodes \( s_{0}, \ldots, s_{m} \in T \) such that \( s_{0} = s, s_{m} = t \) and \( (s_{i}, s_{i+1}) \in E(T) \).

For a metric space \( X \) its density character, i.e., the smallest cardinality of a dense subset, is denoted by
\( \chi(X) \).

\subsection{Acknowledgment}
\label{sec:acknowledgment}

The author wants to thank Christian Rosendal for his tireless support and numerous helpful and very inspiring
conversations.  Part of this work was done during the author's participation in the program on ``Asymptotic geometric
analysis'' at the Fields Institute in the Fall, 2010 and during the trimester on ``Von Neumann algebras and ergodic
theory of group actions'' at the Institute of Henri Poincare in Spring, 2011.  The author thanks sincerely the
organizers of these programs.

The author also thanks the anonymous referee for the valuable help in improving paper's writing.

\section{Trivial words in amalgams}
\label{sec:triv-words-amalg}

Let a family \( \{G_{\lambda}\}_{\lambda \in \Lambda} \) of groups be given, where \( \Lambda \) is
an index set.  Suppose all of the groups contain a subgroup \( A \subseteq G_{\lambda} \), and
assume that \( G_{\lambda_{1}} \cap G_{\lambda_{2}} = A \) for all \( \lambda_{1} \ne \lambda_{2}
\).  Let \( G = \bigcup_{\lambda \in \Lambda} G_{\lambda} \) denote the union of the
groups \( G_{\lambda} \). The identity element in any group is denoted by \( e \), the ambient group
will be evident from the context.  Let \( 0 \) be a symbol not in \( \Lambda \).  For
\( g_{1}, g_{2} \in G \) we set \( g_{1} \cong g_{2} \) to denote the existence of
\( \lambda \in \Lambda \) such that \( g_{1}, g_{2} \in G_{\lambda} \).  If \( g_{1} \cong g_{2} \),
we say that \( g_{1} \) and \( g_{2} \) are \emph{\multipliable.}  We also define a relation
on \( \Lambda \cup \{0\} \) by declaring that \( x, y \in \Lambda \cup \{0\} \) are in relation if and
only if either \( x = y \) or at least one of \( x, y \) is \( 0 \).  This relation on
\( \Lambda \cup \{0\} \) is also denoted by \( \cong \).

The free product of the groups \( G_{\lambda} \) with amalgamation over the subgroup \( A \) is
denoted by \( \amalgam \).  We carefully distinguish words over the alphabet \( G \) from elements
of the amalgam \( \amalgam \).  For that we introduce the following notation.  \( \words \) denotes
the set of finite nonempty words over the alphabet \( G \).  The length of a word
\( \alpha \in \words \) is denoted by \( |\alpha| \), the concatenation of two words \( \alpha \)
and \( \beta \) is denoted by \( \alpha \concat \beta \), and the \( i^{th} \) letter of
\( \alpha \) is denoted by \( \alpha(i) \); in particular, for any \( \alpha \) \( \in \words \)
\[ \alpha = \alpha(1) \concat \alpha(2) \concat \cdots \concat \alpha(|\alpha|).  \]
Two words \( \alpha, \beta \in \words \) are said to be \emph{\multipliable} if \( |\alpha| = |\beta| \)
and \( \alpha(i) \cong \beta(i) \) for all \( i \in \seg{1}{|\alpha|} \).  For technical reasons (to
be concrete, for the induction argument in Proposition
\ref{sec:triv-words-amalg-structure-of-the-trivial-word}) we need the following notion of a labeled
word. A \emph{labeled word} is a pair \( (\alpha,l_{\alpha}) \), where \( \alpha \) is a word of
length \( n \), and \( l_{\alpha} : \seg{1}{n} \to \Lambda \cup \{0\} \) is a function, called the
label of \( \alpha \), such that
\[ \alpha(i) \in G_{\lambda} \setminus A \implies l_{\alpha}(i) = \lambda  \]
for all \( i \in \seg{1}{n} \).
\begin{example}
\label{exm:canonical-labeling}
  Let \( \alpha \in \words \) be any word.  There is a canonical label for \( \alpha \) given by
  \begin{displaymath}
    l_{\alpha}(i) =
    \begin{cases}
      0& \textrm{if \( \alpha(i) \in A \)}; \\
      \lambda & \textrm{if \( \alpha(i) \in G_{\lambda} \setminus A \)}.
    \end{cases}
  \end{displaymath}
In fact, everywhere, except for the proof of Proposition \ref{sec:triv-words-amalg-structure-of-the-trivial-word}, we use
this canonical labeling only.  
\end{example}

Let \( \alpha \) be a word of length \( n \). For a subset \( F \subseteq \seg{1}{n} \), with \( F =
\{i_{k}\}_{k=1}^{m} \), where \( i_{1} < i_{2} < \ldots < i_{m} \), set
\[ \alpha[F] = \alpha(i_{1}) \concat \alpha(i_{2}) \concat \cdots \concat \alpha(i_{m}).  \]
We say that a subset \( F \subseteq \seg{1}{n} \) is \emph{\( \alpha \)-\multipliable} if \( \alpha(i) \cong \alpha(j) \) for all \( i, j \in F \).

There is a natural evaluation map from the set of words \( \words \) over the alphabet \( G \) to the amalgam
\( \amalgam \) given by the multiplication of letters in the group \( \amalgam \):
\[ \alpha \mapsto \alpha(1) \cdot \alpha(2) \cdots \alpha(|\alpha|).  \]
This map is denoted by a hat
\[ \widehat{}\ : \words \to \amalgam.  \]
Note that this is map is obviously surjective.  For a word \( \alpha \in \words \) and a subset
\( F \subseteq \seg{1}{|\alpha|} \) we write \( \hat{\alpha}[F] \) instead of
\( \widehat{\alpha[F]} \).  We hope this will not confuse the reader too much.  A word \( \alpha \)
is said to be \emph{trivial} if \( \hat{\alpha} = e \).

\subsection{Structure of trivial words}
\label{sec:struct-triv-words}

Elements of the group \( A \) will be special for us.  Let \( \alpha \in \words\) be a word of
length \( n \).  We say that its \( i^{th} \) letter is \emph{outside of \( A \)} if, as the name
suggests, \( \alpha(i) \not \in A \).  The \emph{list of external letters} of \( \alpha \) is a,
possibly empty, sequence \( \{i_{k}\}_{k=1}^{m} \) such that
\begin{enumerate}[(i)]
\item \( i_{k} < i_{k+1} \) for all \( k \in \seg{1}{m-1} \);
\item \( \alpha(i_{k}) \not \in A \) for all \( k \in \seg{1}{m} \);
\item \( \alpha(i) \not \in A \) implies \( i = i_{k} \) for some \( k \in \seg{1}{m} \).
\end{enumerate}
In other words, this is just the increasing list of all the letters in \( \alpha \) that are outside
of \( A \).

\begin{definition}
  \label{sec:triv-words-amalg-alternating-word}
  Let \( \alpha \in \words \) be a word with the list of external letters \( \{i_{k}\}_{k=1}^{m} \).
  The word \( \alpha \) is called \emph{alternating} if
  \( \alpha(i_{k}) \not \cong \alpha(i_{k+1}) \) for all \( k \in \seg{1}{m-1} \).  Note that a word
  is always alternating if \( m \le 1 \).  The word \( \alpha \) is said to be \emph{reduced} if
  \( \alpha(i) \not \cong \alpha(i+1) \) for all \( i \in \seg{1}{|\alpha|-1}\), and it is called a
  \emph{reduced form of \( f \in \amalgam \)} if additionally \( \hat{\alpha} = f \).
\end{definition}
The following is a basic fact about free products with amalgamation.

\begin{lemma}
  \label{sec:triv-words-amalg-non-triviality-of-reduced-words}
  Let \( \alpha \in \words \) be a reduced word.  If \( \alpha \ne e \), then
  \( \hat{\alpha} \ne e \).
\end{lemma}

It is worth mentioning that if \( A \ne \{e\} \), then an element \( f \in \amalgam \) has many
different reduced forms (unless \( f \in G \), then it has only one).  But all these reduced forms
have the same length, therefore it is legitimate to talk about the length of an element \( f \)
itself.

\begin{lemma}
  \label{sec:triv-words-amalg-reduced-forms}
  Any element \( f \in \amalgam \) has a reduced form \( \alpha \in \words \).  Moreover, if
  \( \beta \in \words \) is another reduced form of \( f \), then \( |\alpha| = |\beta| \) and
  \( A\alpha(i)A = A\beta(i)A \) for all \( i \in \seg{1}{|\alpha|} \).
\end{lemma}

\begin{proof}
  The existence of a reduced form of \( f \in \amalgam \) is obvious.  Suppose \( \alpha \) and
  \( \beta \) are both reduced forms of \( f \).  Set
  \[ \zeta = \alpha(|\alpha|)^{-1} \concat \cdots \concat \alpha(1)^{-1} \concat \beta(1) \concat
  \cdots \concat \beta(|\beta|).  \]
  Since \( \hat{\zeta} = e \) and \( \zeta \ne e \), by Lemma
  \ref{sec:triv-words-amalg-non-triviality-of-reduced-words} \( \zeta \) is not reduced.  By
  assumption, \( \alpha \) and \( \beta \) were reduced, therefore \( \alpha(1) \cong \beta(1) \).
  We claim that \( \alpha(1)^{-1} \beta(1) \in A \).  Indeed, if
  \( \alpha(1)^{-1}\beta(1) \not \in A \), then the word
  \[ \xi = \alpha(|\alpha|)^{-1} \concat \cdots \concat \alpha(1)^{-1} \cdot \beta(1) \concat \cdots
  \concat \beta(|\beta|) \]
  is reduced, \( \hat{\xi} = e \), and \( \xi \ne e \), contradicting Lemma
  \ref{sec:triv-words-amalg-non-triviality-of-reduced-words}.  So \( \alpha(1)^{-1}\beta(1) \in A
  \), and therefore \( \beta(1) = \alpha(1)a_{1} \) for some \( a_{1} \in A \) and
  \( A\alpha(1)A = A\beta(1)A \).  Now set
  \[ \alpha_{1} = \alpha(2) \concat \cdots \concat \alpha(|\alpha|), \quad \beta_{1} = a_{1} \cdot
  \beta(2) \concat \cdots \concat \beta(|\beta|).  \]
  Since \( \hat{\alpha}_{1} = \hat{\beta}_{1} \) and \( \alpha_{1}, \beta_{1} \) are reduced, we can
  apply the same argument to get \( \alpha_{1}(1) = \beta_{1}(1)a_{2} \) for some \( a_{2} \in A \),
  whence
  \[ A\alpha(2)A = A\alpha_{1}(1)A = A\beta_{1}(1)A = A\beta(2)A.  \]
  And we proceed by induction on \( |\alpha| + |\beta| \).
\end{proof}

\begin{lemma}
  \label{sec:triv-words-amalg-reduced-form-length}
  Let \( f \in \amalgam \) and \( \alpha, \beta \in \words \) be given.  If \( \alpha \) is a
  reduced form of \( f \), \( |\alpha| = |\beta| \) and \( \hat{\alpha} = \hat{\beta} \), then
  \( \beta \) is a reduced form of \( f \).
\end{lemma}

\begin{proof}
  If \( \beta \) is not a reduced form of \( f \), we perform cancellations in \( \beta \) and get a
  reduced word \( \beta_{1} \) such that \( \hat{\beta}_{1} = f \) and \( |\beta_{1}| < |\beta| \).
  By Lemma \ref{sec:triv-words-amalg-reduced-forms} we have \( |\beta_{1}| = |\alpha| \),
  contradicting \( |\beta| = |\alpha| \).  Hence \( \beta \) is reduced.
\end{proof}

\begin{lemma}
  \label{sec:triv-words-amalg-alternatin-word-non-trivial}
  If \( \alpha \) is an alternating word with a nonempty list of external letters, then
  \( \hat{\alpha} \ne e \).
\end{lemma}
\begin{proof}
  Let \( \{i_{k}\}_{k=1}^{m} \) be the list of external letters of \( \alpha \).  For
  \( k \in \seg{2}{m-1} \) set
  \[ \xi_{1} = \alpha(1) \cdots \alpha(i_{2}-1), \]
  \[ \xi_{k} = \alpha(i_{k}) \cdot \alpha(i_{k}+1) \cdots \alpha(i_{k+1}-1), \]
  \[ \xi_{m} = \alpha(i_{m}) \cdot \alpha(i_{m}+1) \cdots \alpha(n), \] and put
  \[ \xi = \xi_{1} \concat \cdots \concat \xi_{m}.  \]
  Then \( \hat{\xi} = \hat{\alpha} \), \( \xi \ne e \) (since \( \xi_{i} \ne e \) for all
  \( i \in \seg{1}{m} \)), and, as one easily checks, \( \xi \) is reduced.  An application of Lemma
  \ref{sec:triv-words-amalg-non-triviality-of-reduced-words} finishes the proof.
\end{proof}

\begin{lemma}
  \label{sec:triv-words-amalg-subword-of-trivial-word}
  If \( \zeta \) is a trivial word of length \( n \) with a nonempty list of external letters, then
  there is an interval \( I \subseteq \seg{1}{n} \) such that
  \begin{enumerate}[(i)]
  \item\label{lem:subword-trivial-word-item:evaluation} \( \hat{\zeta}[I] \in A \);
  \item\label{lem:subword-trivial-word-item:congruence} \( I \) is \( \zeta \)-\multipliable;
  \item\label{lem:subword-trivial-word-item:endpoints}
    \( \zeta \big( m(I) \big), \zeta \big( M(I) \big) \not \in A \).
  \end{enumerate}
\end{lemma}

\begin{proof}
  Let \( \{i_{k}\}_{k=1}^{m} \) be the list of external letters.  For all \( k \in \seg{1}{m} \) define \( m_{k} \) and
  \( M_{k} \) by
  \[ m_{k} = \min\{j \in \seg{1}{k}: \seg{i_{j}}{i_{k}}\ \textrm{is \( \zeta \)-\multipliable}\}, \]
  \[ M_{k} = \max\{j \in \seg{k}{m}: \seg{i_{k}}{i_{j}}\ \textrm{is \( \zeta \)-\multipliable}\}.  \]
  Set \( I_{k} = \seg{m_{k}}{M_{k}} \), and note that for \( k, l \in \seg{1}{m} \)
  \[ I_{k} \cap I_{l} \ne \emptyset \implies I_{l} = I_{k}.  \]
  
  Let \( I_{k_{1}}, \ldots, I_{k_{p}} \) be a list of all the distinct intervals \( I_{k_{i}} \).  Then
  \( \{I_{k_{i}}\}_{i=1}^{p} \) are pairwise disjoint.  Note that each of \( I_{k_{i}} \) satisfies items
  \eqref{lem:subword-trivial-word-item:congruence} and \eqref{lem:subword-trivial-word-item:endpoints}.  To prove the
  lemma it is enough to show that for some \( i \in \seg{1}{p} \) the corresponding \( I_{k_{i}} \) satisfies also item
  \eqref{lem:subword-trivial-word-item:evaluation}.  Suppose this is false and \( \hat{\zeta}[I_{k_{i}}] \not \in A \)
  for all \( i \in \seg{1}{p}\).  Set \( \xi_{i} = \hat{\zeta}[I_{k_{i}}] \) and
  \begin{multline*}
    \xi = \zeta(1) \concat \cdots \concat \zeta(m(I_{k_{1}})-1) \concat \xi_{1} \concat \zeta(M(I_{k_{1}})+1) \concat \cdots \\
    \cdots \concat \zeta(m(I_{k_{2}}) - 1) \concat \xi_{2} \concat \zeta(M(I_{k_{2}})+1) \concat \cdots\\
    \cdots \concat \zeta(m(I_{k_{p}})-1) \concat \xi_{p} \concat \zeta(M(I_{k_{p}})+1) \concat \cdots \concat \zeta(n).
  \end{multline*}
  Then, of course, \( \hat{\xi} = \hat{\zeta} = e \) and \( \xi \) is alternating by the choice of \( \{I_{k_{i}}\} \).
  By Lemma \ref{sec:triv-words-amalg-alternatin-word-non-trivial} the word \( \xi \) is non-trivial, which is a
  contradiction.
\end{proof}

\begin{lemma}
  \label{sec:triv-words-amalg-congruent-interval-in-trivial-word}
  If \( (\zeta,l_{\zeta}) \) is a trivial labeled word of length \( n \) with a nonempty list of external letters, then
  there is an interval \( I \subseteq \seg{1}{n} \) such that
  \begin{enumerate}[(i)]
  \item\label{lem:congruent-interval-item:evaluation} \( \hat{\zeta}[I] \in A \);
  \item\label{lem:congruent-interval-item:congruence} \( I \) is \( \zeta \)-\multipliable;
  \item\label{lem:congruent-interval-item:non-triviality} \( \zeta(i) \not \in A \) for some
    \( i \in I \);
  \item\label{lem:congruent-interval-item:weak-maximality} if \( m(I)>1 \), then
    \( l_{\zeta}(m(I)-1) \ne 0 \); if \( M(I)<n \), then \( l_{\zeta}(M(I)+1) \ne 0 \);
  \item\label{lem:congruent-interval-item:endpoints} if \( \zeta(m(I)) \in A \), then
    \( l_{\zeta}(m(I)) = 0 \); if \( \zeta(M(I)) \in A \), then \( l_{\zeta}(M(I)) = 0 \).
  \end{enumerate}
\end{lemma}

\begin{proof}
  We start by applying Lemma \ref{sec:triv-words-amalg-subword-of-trivial-word} to the word
  \( \zeta \).  This Lemma gives as an output an interval \( J \subseteq \seg{1}{n} \).  We will now
  enlarge this interval as follows.  If \( l_{\zeta}(i) = 0 \) for all \( i \in \seg{1}{m(J)-1} \),
  then set \( j_{l} = 1 \).  If there is some \( i < m(J) \) such that \( l_{\zeta}(i) \ne 0 \),
  then let \( j \in \seg{1}{m(J)-1} \) be maximal such that \( l_{\zeta}(j) \ne 0 \) and set
  \( j_{l} = j+1 \).  Similarly, if \( l_{\zeta}(i) = 0 \) for all \( i \in \seg{M(J)+1}{n} \), then
  set \( j_{r} = n \).  If there is some \( i > M(J) \) such that \( l_{\zeta}(i) \ne 0 \), then let
  \( j \in \seg{M(J)+1}{n} \) be minimal such that \( l_{\zeta}(j) \ne 0 \) and set
  \( j_{r} = j-1 \).  Define
  \[ I = J \cup \seg{j_{l}}{m(J)} \cup \seg{M(J)}{j_{r}} = \seg{j_{l}}{j_{r}}. \]
  
  We claim that \( I \) satisfies the assumptions.  Note that \( J \subseteq I \) and \( \zeta(i) \in A \) for all
  \( i \in I \setminus J \), so \eqref{lem:congruent-interval-item:evaluation},
  \eqref{lem:congruent-interval-item:congruence} and
  \eqref{lem:congruent-interval-item:non-triviality} follow from items
  \eqref{lem:subword-trivial-word-item:evaluation}, \eqref{lem:subword-trivial-word-item:congruence}
  and \eqref{lem:subword-trivial-word-item:endpoints} of Lemma
  \ref{sec:triv-words-amalg-subword-of-trivial-word}.  Items
  \eqref{lem:congruent-interval-item:weak-maximality} and
  \eqref{lem:congruent-interval-item:endpoints} follow from the choice of \( j_{l} \) and
  \( j_{r} \) and from item \eqref{lem:subword-trivial-word-item:endpoints} of Lemma
  \ref{sec:triv-words-amalg-subword-of-trivial-word}.
\end{proof}

\begin{definition}
  \label{sec:struct-triv-words-def-of-evaluation-tree}
  Let \( (\zeta, l_{\zeta}) \) be a trivial labeled word of length \( n \), and let \( T \) be a
  tree.  Suppose that to each node \( t \in T \) an interval \( I_{t} \subseteq \seg{1}{n} \) is
  assigned.  Set \( R_{t} = I_{t} \setminus \bigcup_{t' \prec t} I_{t'} \).  The tree \( T \)
  together with the assignment \( t \mapsto I_{t} \) is called \emph{an evaluation tree for
    \( (\zeta,l_{\zeta}) \)} if for all \( s, t \in T \) the following holds:
  \begin{enumerate}[(i)]
  \item\label{lem:structure-item:root} \( I_{\emptyset} = \seg{1}{n} \);
  \item\label{lem:structure-item:evaluation} \( \hat{\zeta}[I_{t}] \in A \);
  \item\label{lem:structure-item:endpoints} if \( t \ne \emptyset \) and \( \zeta(m(I_{t})) \in A \), then
    \( l_{\zeta}(m(I_{t})) = 0 \); if \( t \ne \emptyset \) and \( \zeta(M(I_{t})) \in A \), then
    \( l_{\zeta}(M(I_{t})) = 0 \);
  \item\label{lem:structure-item:intervals-order} if \( H(t) \le H(s) \) and
    \( I_{s} \cap I_{t} \ne \emptyset \), then \( s \prec t \) or \( s=t \);
  \item\label{lem:structure-item:strict-inclusion} if \( s \prec t \) and \( t \ne \emptyset \),
    then
    \[ m(I_{t}) < m(I_{s}) \le M(I_{s}) < M(I_{t}); \]
  \item\label{lem:structure-item:congruence} \( \zeta(i) \cong \zeta(j) \) for all
    \( i, j \in R_{t} \);
  \end{enumerate}
  An evaluation tree \( T \) is called \emph{balanced} if additionally the following two conditions
  hold:
  \begin{enumerate}[(i)]
    \setcounter{enumi}{6}
  \item\label{lem:structure-item:non-trivial-interior} if \( T \ne \{\emptyset\} \), then
    for any \( t \in T \) if \( R_{t} \) is written as a disjoint union of maximal
    sub-intervals \(\{ \mathcal{I}_{j} \}_{j=1}^{k}\), then for any \( j \) there is
    \( i \in \mathcal{I}_{j} \) such that \( l_{\zeta}(i) \ne 0 \);
  \item\label{lem:structure-item:non-trivial-boundary} if \( s \prec t \), then
    \[ m(I_{s}) - 1 \in R_{t} \implies l_{\zeta}(m(I_{s}) - 1) \ne 0; \]
    \[ M(I_{s}) + 1 \in R_{t} \implies l_{\zeta}(M(I_{s}) + 1) \ne 0. \]
  \end{enumerate}
\end{definition}

\begin{remark}
  \label{sec:struct-triv-words-standard-label-vacuous-condition}
  Note that if \( \zeta \in \words \) is a trivial word with the canonical label as in Example
  \ref{exm:canonical-labeling}, then item \eqref{lem:structure-item:endpoints} in the definition of
  an evaluation tree is vacuous.
\end{remark}

\begin{proposition}
  \label{sec:triv-words-amalg-structure-of-the-trivial-word}
  Any trivial labeled word \( (\zeta,l_{\zeta}) \) has a balanced evaluation tree.
\end{proposition}

\begin{proof}
  We prove the proposition by induction on the cardinality of the list of external letters of
  \( \zeta \).  Suppose first that the list is empty, and \( \zeta(i) \in A \) for all
  \( i \in \seg{1}{n} \).  Set \( T_{\zeta} = \{\emptyset\} \) and \( I_{\emptyset} = \seg{1}{n} \).
  It is easy to check that all the conditions are satisfied, and \( T_{\zeta} \) is a balanced
  evaluation tree for \( (\zeta, l_{\zeta}) \).

  From now on we assume there is \( i \in \seg{1}{n} \) such that \( \zeta(i) \not \in A \).  Apply
  Lemma \ref{sec:triv-words-amalg-congruent-interval-in-trivial-word} to \( (\zeta, l_{\zeta}) \)
  and let \( I \) be the interval granted by this lemma.  Set \( \lambda_{0} = l_{\zeta}(i) \) for
  some (equivalently, any) \( i \in I \) such that \( \zeta(i) \not \in A \).  Note that
  \( \lambda_{0} \ne 0 \).  Let \( p = |I| \) be the length of \( I \).  If \( p = n \), then we set
  \( T_{\zeta} = \{\emptyset\} \) and \( I_{\emptyset} = \seg{1}{n} \).  Similarly to the base of
  induction this tree is a balanced evaluation tree for \( (\zeta, l_{\zeta}) \).  From now on we
  assume that \( p < n \).  We define the word \( \xi \) of length \( n-p+1 \) as follows.  Set
  \begin{displaymath}
    \xi(i) = 
    \begin{cases}
      \zeta(i) & \textrm{if \( i < m(I) \)}\\
      \hat{\zeta}[I] & \textrm{if \( i = m(I) \)}\\
      \zeta(i + p-1) & \textrm{if \( i > m(I) \)}.
    \end{cases}
  \end{displaymath}
  Define the label for \( \xi \) to be
  \begin{displaymath}
    l_{\xi}(i) = 
    \begin{cases}
      l_{\zeta}(i) & \textrm{if \( i < m(I) \)}\\
      \lambda_{0}     & \textrm{if \( i = m(I) \)}\\
      l_{\zeta}(i + p-1) & \textrm{if \( i > m(I) \)}.
    \end{cases}
  \end{displaymath}
  We claim that
  \[ \big| \{i \in \seg{1}{|\xi|} : \xi(i) \not \in A\} \big| < \big| \{i \in \seg{1}{n} : \zeta(i)
  \not \in A\} \big|.  \]
  Indeed, by the construction \( \zeta[I] \) has at least one letter (in fact, at least two letters)
  not from \( A \).

  By inductive assumption applied to the labeled word \( (\xi,l_{\xi}) \), there is a balanced
  evaluation tree \( T_{\xi} \) with intervals \( J_{t} \subseteq \seg{1}{|\xi|} \) for
  \( t \in T_{\xi} \). Since \( J_{\emptyset} = \seg{1}{|\xi|} \), there is at least one
  \( t \in T_{\xi}\) (namely \( t = \emptyset \)) such that the interval \( J_{t} \) contains
  \( m(I) \). By item \eqref{lem:structure-item:intervals-order} there is the smallest node
  \( t_{0} \in T_{\xi} \) such that \( m(I) \in J_{t_{0}} \).

  We define \( T_{\zeta} \) to be \( T_{\xi} \cup \{s_{0}\} \), where \( s_{0} \) is a new
  predecessor of \( t_{0} \), i. e. , \( s_{0} \prec t_{0} \). For \( t \in T_{\xi} \) set
  \begin{displaymath}
    I_{t} = 
    \begin{cases}
      \seg{m(J_{t})}{M(J_{t})} & \textrm{if \( M(J_{t}) < m(I) \)};\\
      \seg{m(J_{t})}{M(J_{t})+p-1} & \textrm{if \( m(J_{t}) \le m(I) \le M(J_{t}) \)};\\
      \seg{m(J_{t})+p-1}{M(J_{t})+p-1} & \textrm{if \( m(I)< m(J_{t}) \)};\\
    \end{cases}
  \end{displaymath}
  and
  \[ I_{s_{0}} = \seg{m(I)}{M(I)}.  \]
  We claim that such a tree \( T_{\zeta} \) with such an assignment of intervals \( I_{t} \) is a
  balanced evaluation tree for \( (\zeta,l_{\zeta}) \).

  \eqref{lem:structure-item:root} Since \( J_{\emptyset} = \seg{1}{|\xi|} \), it follows that
  \( I_{\emptyset} = \seg{1}{n} \).

  \eqref{lem:structure-item:evaluation} For any \( t \in T_{\xi} \) one has
  \( \hat{\xi}[J_{t}] = \hat{\zeta}[I_{t}] \).  Also, \( \hat{\zeta}[I_{s_{0}}] \in A \) by item
  \eqref{lem:congruent-interval-item:evaluation} of Lemma
  \ref{sec:triv-words-amalg-congruent-interval-in-trivial-word}.

  \eqref{lem:structure-item:endpoints} Since \( \xi(m(I)) \in A \) and
  \( l_{\xi}(m(I)) = \lambda_{0} \ne 0 \), by inductive hypothesis \( m(I_{t}) \ne m(I) \) and
  \( M(I_{t}) \ne m(I) \) for all \( t \in T_{\xi} \setminus \{\emptyset\} \).  Therefore
  \( l_{\xi}(m(J_{t})) = l_{\zeta}(m(I_{t})) \), \( l_{\xi}(M(J_{t})) = l_{\zeta}(M(I_{t})) \) for
  all \( t \in T_{\xi} \setminus \{\emptyset\} \).  Thus for \( t \ne s_{0} \) the item follows from
  the inductive hypothesis, and for \( t = s_{0} \) it follows from item
  \eqref{lem:congruent-interval-item:endpoints} of Lemma
  \ref{sec:triv-words-amalg-congruent-interval-in-trivial-word}.

  \eqref{lem:structure-item:intervals-order} Follows from the inductive hypothesis and the
  definition of \( s_{0} \).

  \eqref{lem:structure-item:strict-inclusion} It follows from the inductive hypothesis that this
  item is satisfied for all \( s, t \in T_{\xi} \).  We need to consider the case \( s = s_{0} \),
  \( t = t_{0} \) only. By item \eqref{lem:structure-item:endpoints} of the definition of an
  evaluation tree, and since \( l_{\xi}(m(I)) = \lambda_{0} \ne 0 \), it follows that if
  \( t_{0} \ne \emptyset \), then \( m(I_{t_{0}}) < m(I_{s_{0}}) \) and
  \( M(I_{s_{0}}) < M(I_{t_{0}}) \).

  \eqref{lem:structure-item:congruence} Follows easily from the inductive hypothesis and item
  \eqref{lem:congruent-interval-item:congruence} of Lemma
  \ref{sec:triv-words-amalg-congruent-interval-in-trivial-word}.

  Thus \( T_{\zeta} \) is an evaluation tree for \( (\zeta,l_{\zeta}) \).  It remains to check that
  it is balanced.

  \eqref{lem:structure-item:non-trivial-interior} For \( t \in T_{\xi} \setminus \{t_{0}\} \) the
  maximal sub-intervals of \( J_{t} \setminus \bigcup_{s \prec t} J_{s} \) naturally correspond to the
  maximal sub-intervals of \( I_{t} \setminus \bigcup_{s \prec t} I_{s} \), and hence for such a
  \( t \) the item follows from the inductive hypothesis.  For \( t = s_{0} \) the item follows from
  item \eqref{lem:congruent-interval-item:non-triviality} of Lemma
  \ref{sec:triv-words-amalg-congruent-interval-in-trivial-word}.  The remaining case \( t = t_{0} \)
  follows from item \eqref{lem:congruent-interval-item:weak-maximality} of Lemma
  \ref{sec:triv-words-amalg-congruent-interval-in-trivial-word}.

  \eqref{lem:structure-item:non-trivial-boundary} Again, for \( s \ne s_{0} \) this item follows
  from the inductive hypothesis and for \( s = s_{0} \), \( t =t_{0} \) follows from item
  \eqref{lem:congruent-interval-item:weak-maximality} of Lemma
  \ref{sec:triv-words-amalg-congruent-interval-in-trivial-word}.
\end{proof}

If \( \zeta \) is just a word with no labeling, then we canonically associate a label to it by
declaring \( l_{\zeta}(i) = 0 \) if and only if \( \zeta(i) \in A \) (as in Example
\ref{exm:canonical-labeling}).

\emph{From now on we view all trivial words as labeled words with the canonical labeling.}

\begin{definition}
  \label{sec:triv-words-amalg-slim-trivial-word}
  A trivial word \( \zeta \in \words \) of length \( n \) is called \emph{slim} if there exists an
  evaluation tree \( T_{\zeta} \) such that \( \hat{\zeta}[I_{t}] = e \) for all
  \( t \in T_{\zeta} \); such a tree is then called a \emph{slim} evaluation tree.  We say that
  \( \zeta \) is \emph{simple} if it is slim and \( \zeta(i) \in A \) implies \( \zeta(i) = e \) for
  all \( i \in \seg{1}{n} \).
\end{definition}

\begin{definition}
  \label{sec:metrics-amalgams-f-pair}
  Let \( f \in \amalgam \).  A pair of words \( (\alpha, \zeta) \) is called an \emph{\( f \)-pair} if
  \( |\alpha| = |\zeta| \) and \( \hat{\alpha} = f \), \( \hat{\zeta} = e \).  An \( f \)-pair \( (\alpha,\zeta) \) is
  said to be a \emph{\multipliable} \( f \)-pair if \( \alpha \) and \( \zeta \) are \multipliable.  An \( f \)-pair
  \( (\alpha,\zeta) \) is called \emph{slim} if it is \multipliable and \( \zeta \) is slim.  It is called \emph{simple}
  if it is \multipliable and \( \zeta \) is simple.
\end{definition}

For a \multipliable pair \( (\alpha,\beta) \) of length \( n \) we define the notions of right and left
transfers.  Let \( a \in A \) and \( i \in \seg{1}{n-1} \) be given.  The \emph{right \( (a,i)
  \)-transfer} of \( (\alpha,\beta) \) is the pair
\( \transferr{\alpha}{\beta}{a}{i} = (\gamma,\delta) \) defined as follows:
\begin{displaymath}
  (\gamma(j),\delta(j)) = 
  \begin{cases}
    (\alpha(j),\beta(j)) & \textrm{if \( j \not \in \{i,i+1\} \)};\\
    (\alpha(i)a^{-1},\beta(i)a^{-1}) & \textrm{if \( j = i \)};\\
    (a\alpha(i+1),a\beta(i+1)) & \textrm{if \( j = i+1 \)}.
  \end{cases}
\end{displaymath}
For \( a \in A \) and \( i \in \seg{2}{n} \) the \emph{left \( (a,i) \)-transfer} of
\( (\alpha,\beta) \) is denoted by \( \transferl{\alpha}{\beta}{a}{i} = (\gamma,\delta) \) and is
defined as
\begin{displaymath}
  (\gamma(j),\delta(j)) = 
  \begin{cases}
    (\alpha(j),\beta(j)) & \textrm{if \( j \not \in \{i-1,i\} \)};\\
    (a^{-1}\alpha(i), a^{-1}\beta(i)) & \textrm{if \( j = i \)};\\
    (\alpha(i-1)a,\beta(i-1)a) & \textrm{if \( j = i-1 \)}.
  \end{cases}
\end{displaymath}

We will typically have specific sequences of transfers, so it is convenient to make the following
definition.  Let \( (\alpha,\zeta) \) be a \multipliable pair of words of length \( n \).  In all the
applications \( \zeta \) will be a trivial word.  Let \( \{I_{k}\}_{k=1}^{m} \) be a sequence of
intervals such that:
\begin{enumerate}
\item \( I_{k} \subseteq \seg{1}{n} \);
\item \( I_{k} < I_{k+1} \) for all \( k \in \seg{1}{m-1} \);
\item \( \hat{\zeta}[I_{k}] \in A \) for all \( k \in \seg{1}{m} \);
\item \( M(I_{m}) < n \).
\end{enumerate}
Such a sequence is called \emph{right transfer admissible}.  If together with items \( (1)-(3) \)
the following is satisfied
\begin{enumerate}
\item[\( (4') \)] \( m(I_{1}) > 1, \)
\end{enumerate}
then the sequence \( \{I_{k}\}_{k=1}^{m} \) is called \emph{left transfer admissible}.

Let \( \{I_{k}\}_{k=1}^{m} \) be a right transfer admissible sequence of intervals.  Define
inductively words \( (\beta_{k},\xi_{k}) \) by setting \( (\beta_{0}, \xi_{0}) = (\alpha,\zeta) \)
and
\[ (\beta_{k+1}, \xi_{k+1}) =
\transferr{\beta_{k}}{\xi_{k}}{\hat{\xi}_{k}[I_{k+1}]}{M(I_{k+1})}.  \]
We have to show that the right-hand side is well-defined, i.e., that
\( \hat{\xi}_{k}[I_{k+1}] \in A\).  For the first step of the construction we have
\( \hat{\xi}_{0}[I_{1}] = \hat{\zeta}[I_{1}] \in A \), because the sequence is right transfer
admissible.  Suppose we have proved that \( \hat{\xi}_{k-1}[I_{k}] \in A \).  There are two cases:
either \( M(I_{k}) + 1 = m(I_{k+1}) \), and then
\[ \hat{\xi}_{k}[I_{k+1}] = (\hat{\xi}_{k-1}[I_{k}]) \cdot \hat{\zeta}[I_{k+1}], \]
or \( M(I_{k}) + 1 < m(I_{k+1})\), and then \( \hat{\xi}_{k}[I_{k+1}] = \hat{\zeta}[I_{k+1}] \).  In
both cases we get \( \hat{\xi}_{k}[I_{k+1}] \in A \).

By definition, the right \( \{I_{k}\} \)-transfer of \( (\alpha,\zeta) \) is the pair
\( (\beta_{m},\xi_{m}) \).

The left transfer is defined similarly, but with one extra change: we apply left transfers in the
decreasing order from \( I_{m} \) to \( I_{1} \).  Here is a formal definition. For a left
admissible sequence of intervals \( \{I_{k}\}_{k=1}^{m} \) set inductively
\( (\beta_{0},\xi_{0}) = (\alpha,\zeta) \) and
\[ (\beta_{k+1}, \xi_{k+1}) =
\transferl{\beta_{k}}{\xi_{k}}{\hat{\xi}_{k}[I_{m-k}]}{m(I_{m-k})}.  \]
Similarly to the case of the right transfer one shows that the right-hand side in the above
construction is well-defined.  By definition, the left \( \{I_{k}\} \)-transfer of
\( (\alpha,\zeta) \) is the pair \( (\beta_{m},\xi_{m}) \).

This notion of transfer, though a bit technical, will be crucial in some reductions in the next
section.  The following lemma establishes basic properties of the transfer operation with respect to
the earlier notion of the evaluation tree.

\begin{lemma}
  \label{sec:triv-words-amalg-transfer-preserves}
  Let \( (\alpha,\zeta) \) be a \multipliable \( f \)-pair of length \( n \) and let \( T_{\zeta} \) be
  a [balanced] evaluation tree for \( \zeta \).  Let \( \{I_{k}\}_{k=1}^{m} \) be a right [left]
  transfer admissible sequence of intervals.  Let \( (\beta,\xi) \) be the right [left]
  \( \{I_{k}\} \)-transfer of \( (\alpha,\zeta) \).  Then
  \begin{enumerate}[(i)]
  \item\label{lem:transfer-preserves-item:same-length} \( |\beta| = n = |\xi| \);
  \item\label{lem:transfer-preserves-item:congruent} \( (\beta,\xi) \) is a \multipliable \( f \)-pair;
  \item\label{lem:transfer-preserves-item:same-tree} \( T_{\zeta} \) is a [balanced] evaluation tree
    for \( \xi \).
  \item\label{lem:transfer-preserves-item:indices-under-change} \( \xi(i) = \zeta(i) \) for all
    \( i \not \in \{M(I_{k}), M(I_{k})+1 : k \in \seg{1}{m}\} \) for the right transfer and for all
    \( i \not \in \{m(I_{k}), m(I_{k})-1 : k \in \seg{1}{m}\} \) in the case of the left
    transfer;
  \item\label{lem:transfer-preserves-item:triviality-of-transfer-intervals}
    \( \hat{\xi}[I_{k}] = e \) for all \( k \in \seg{1}{m} \).
  \end{enumerate}
\end{lemma}

\begin{proof}
  Items \eqref{lem:transfer-preserves-item:same-length},
  \eqref{lem:transfer-preserves-item:congruent}, and
  \eqref{lem:transfer-preserves-item:indices-under-change} are trivial; item
  \eqref{lem:transfer-preserves-item:same-tree} follows easily from the observation that
  \( \xi(i) \in A \) if and only if \( \zeta(i) \in A\).  For item
  \eqref{lem:transfer-preserves-item:triviality-of-transfer-intervals} let \( \xi_{k} \) be as in
  the definition of the \( \{I_{k}\} \)-transfer.  Suppose for definiteness that we are in the case
  of the right transfer.  Then \( \hat{\xi}_{k}[I_{k}] = e \) by construction and also
  \( \xi_{k+1}[I_{j}] = \xi_{k}[I_{j}] \) for all \( j \in \seg{1}{k} \).  The lemma follows.
\end{proof}

We will later need another operation on words, we call it symmetrization. Here is the definition.

\begin{definition}
  \label{sec:metrics-amalgams-symmetrization-for-simple}
  Let \( (\alpha,\zeta) \) be a slim \( f \)-pair with a slim evaluation tree \( T_{\zeta} \).  Let
  \( t \in T_{\zeta} \) and \( \{i_{k}\}_{k=1}^{m} \subseteq R_{t}\) be a list such that
  \begin{enumerate}[(i)]
  \item \( i_{k} < i_{k+1} \) for \( k \in \seg{1}{m-1} \);
  \item if \( \zeta(i) \ne e \) for some \( i \in R_{t} \), then \( i = i_{k} \) for some
    \( k \in \seg{1}{m} \);
  \item \( \alpha(i_{k}) \cong \alpha(i_{l}) \) for all \( k, l \in \seg{1}{m} \).
  \end{enumerate}
  Such a list is called \emph{symmetrization admissible}.  For \( j_{0} \in \{i_{k}\}_{k=1}^{m} \)
  let \( k_{0} \) be such that \( j_{0} = i_{k_{0}} \) and define a \emph{symmetrization}
  \( \symmet{\alpha}{\zeta}{j_{0}}{\{i_{k}\}_{k=1}^{m}} \) of \( \zeta \) to be the word \( \xi \)
  such that
  \begin{displaymath}
    \xi(i) =
    \begin{cases}
      \zeta(i)     & \textrm{if \( i \ne i_{p} \) for all \( p \in \seg{1}{m} \)};\\
      \alpha(i)    & \textrm{if \( i \in \{i_{k}\}_{k=1}^{m} \setminus \{j_{0}\} \)};\\
      \alpha(i_{k_{0}-1})^{-1} \ldots \alpha(i_{1})^{-1} \cdot \alpha(i_{m})^{-1} \ldots
      \alpha(i_{k_{0}+1})^{-1} & \textrm{if \( i = j_{0} \).}
    \end{cases}
  \end{displaymath}

  If \( m = 1 \), the above definition does not make sense, so we set that in this case
  \( \symmet{\alpha}{\zeta}{i_{1}}{i_{1}} = \zeta \).
\end{definition}

\begin{lemma}
  \label{sec:metrics-amalgams-symmetrization-properties}
  Let \( (\alpha,\zeta) \) be a slim \( f \)-pair with a slim evaluation tree \( T_{\zeta} \).  Let
  \( t \in T_{\zeta} \), and let \( \{i_{k}\}_{k=1}^{m} \subseteq R_{t} \) be a symmetrization
  admissible list.  Fix some \( j_{0} \in \{i_{k}\}_{k=1}^{m} \).  If \( \xi \) is the
  symmetrization \( \symmet{\alpha}{\zeta}{j_{0}}{\{i_{k}\}_{k=1}^{m}} \) of \( \zeta \), then
  \( (\alpha,\xi) \) is a slim \( f \)-pair and \( T_{\zeta} \) is a slim evaluation tree for
  \( \xi \) with the same assignment of intervals \( s \mapsto I_{s} \).
\end{lemma}

\begin{proof}
  The only non-trivial part in the lemma is to show that \( \hat{\xi}[I_{t}] = e \).  This follows
  from the facts that \( \hat{\zeta}[I_{s}] = e \) for all \( s \prec t \) (because \( T_{\zeta} \)
  is slim) and that \( \zeta(i) = e \) for all \( i \in R_{t} \setminus \{i_{1}, \ldots, i_{m}\} \)
  (by the definition of the symmetrization admissible list).
\end{proof}

\medskip

\section{Groups with two-sided invariant metrics}
\label{sec:tsi-groups}

In this section we would like to recall some facts from the theory of groups with two-sided invariant metrics.  The
reader can consult \cite{MR2455198} for the details.

\begin{definition}
  \label{sec:tsi-groups-tsi-group}
  A metric \( d \) on a group \( G \) is called two-sided invariant if
  \[ d(gf_{1},gf_{2}) = d(f_{1},f_{2}) = d(f_{1}g,f_{2}g) \]
  for all \( g, f_{1}, f_{2} \in G \).  A tsi group is a pair \( (G,d) \), where \( G \) is a group and \( d \) is a
  two-sided invariant metric on \( G \); tsi stands for two-sided invariant.
\end{definition}

\begin{proposition}
  \label{sec:tsi-groups-tsi-topological-group}
  If \( (G,d) \) is a tsi group, then \( G \) is a topological group in the topology of the metric \( d \).
\end{proposition}

\begin{proposition}
  \label{sec:tsi-groups-tsi-criterion}
  Let \( d \) be a left invariant metric on the group \( G \).
  \begin{enumerate}[(i)]
  \item If for all \( g_{1}, g_{2}, f_{1}, f_{2} \in G \)
    \[ d(g_{1}g_{2},f_{1}f_{2}) \le d(g_{1},f_{1}) + d(g_{2},f_{2}), \] then \( d \) is two-sided invariant;
  \item If \( d \) is two-sided invariant, then for all \( g_{1}, \ldots, g_{k}, f_{1}, \ldots ,f_{k} \in G \)
    \[ d(g_{1} \cdots g_{k}, f_{1} \cdots f_{k}) \le \sum_{i=1}^{k}d(g_{i},f_{i}).  \]
  \end{enumerate}
\end{proposition}

Because of Proposition \ref{sec:tsi-groups-tsi-topological-group} we choose to speak not about topological groups that
admit a compatible two-sided invariant metric, but rather about abstract groups with a two-sided invariant metric.  Note
that the class of metrizable groups that admit a compatible two-sided invariant metric is very small, but it includes
two important subclasses: abelian and compact metrizable groups.

The class of tsi groups is closed under taking factors by closed normal subgroups, and, moreover, there is a canonical
metric on the factor.

\begin{proposition}
  \label{sec:tsi-groups-factor-metric}
  If \( (G,d) \) is a tsi group and \( N < G \) is a closed normal subgroup, then the function
  \[ d_{0}(g_{1}N,g_{2}N) = \inf\{ d(g_{1}h_{1}, g_{2}h_{2}) : h_{1}, h_{2} \in N\} \]
  is a two-sided invariant metric on the factor group \( G/N \) and the factor map \( \pi : G \to G/N \) is a \( 1
  \)-\lipschitz surjection from \( (G,d) \) onto \( (G/N,d_{0}) \).
\end{proposition}

The metric \( d_{0} \) is called the \emph{factor metric}.

\begin{proposition}
  \label{sec:tsi-groups-completion-tsi-group}
  Let \( (G,d) \) be a tsi group.  Let \( (\overline{G},d) \) be the completion of \( G \) as a metric space; the
  extension of the metric \( d \) on \( G \) to the completion \( \overline{G} \) is again denoted by \( d \).  There is
  a unique extension of group operation from \( G \) to \( \overline{G} \).  This extension turns \( (\overline{G},d) \)
  into a tsi group.
\end{proposition}
This proposition states that for tsi groups metric and group completions are the same.

\section{Graev metric groups}
\label{sec:graev-metric-groups}

Before going into the details of the construction of Graev metrics on free products we would like to recall the
definition of the Graev metrics on free groups.  The reader may consult \cite{MR0038357}, \cite{MR2332614},
\cite{MR2455198} or \cite{MR1288299} for the details and proofs.

Classically one starts with a pointed metric space \( (X,e,d) \), where \( d \) is a metric and \( e \in X \) is a
distinguished point.  Take another copy of this space, denote it by \( (X^{-1},d) \), and its elements are the formal
inverses of the elements in \( X \) with the agreement \( e^{-1} = e \) and \( X \cap X^{-1} = \{e\}\).  Then
\( X^{-1} \) is also a metric space and we can amalgamate \( (X,d) \) and \( (X^{-1},d) \) over the point \( e \).
Denote the resulting space by \( (\env{X},e, d) \).  Equivalently, \( \env{X} = X \cup X^{-1}, \) and for all
\( x, y \in X \)
\[ d(x^{-1},y^{-1}) = d(x,y), \quad d(x,y^{-1}) = d(x,e) + d(e,y).  \]

With the set \( \env{X} \) we associate two objects: the set of \emph{nonempty} words \( \word{\env{X}} \) over the
alphabet \( \env{X} \) and the free group \( F(X) \) over the basis \( X \).  There is a small issue with the second
object.  We want \( e \) to be the identity element of this group rather than an element of the basis.  In other words,
we formally have to write \( F(X \setminus\{e\}) \), but we adopt the convention that given a pointed metric space
\( (X,e,d) \), in \( F(X) \) the letter \( e \in X \) is interpreted as the identity element.  The inverse operation in
\( F(X) \) naturally extends the inverse operation on \( \env{X} \).  We have a natural map
\[ \widehat{ }\ : \word{\env{X}} \to F(X), \]
for \( u \in \word{\env{X}} \) its image \( \hat{u} \) is just the reduced form of \( u \).  For a word
\( u \in \word{\env{X}} \) its length is denoted by \( |u| \) and its \( i^{th} \) letter is denoted by \( u(i) \).  For
two words \( u, v \in \word{\env{X}} \) of the same length \( n \) we define a function
\[ \rho(u,v) = \sum_{i=1}^{n} d(u(i),v(i)).  \] And finally, we define a metric \( \dist \) by
\[ \dist(f,g) = \inf\{\rho(u,v) : |u| = |v| \textrm{ and } \hat{u} = f, \hat{v} = g\}.  \]
A theorem of Graev \cite{MR0038357} states that \( \dist \) is indeed a two-sided invariant metric on \( F(X) \), and
moreover, it extends the metric \( d \) on the amalgam \( \env{X} \).  It is straightforward to see that \( \dist \) is
a two-sided invariant \emph{pseudo}-metric and the hard part of the Graev's theorem is to show that it assigns a
non-zero distance to distinct elements.  Graev showed this by proving some restrictions on \( u \) and \( v \) in the
infimum in the definition of \( d \).  
The effective formula for the Graev metric was first suggested by O. Sipacheva and V. Uspenskij in \cite{MR913066} and
later, but independently, a similar result was obtained in \cite{MR2332614} by L. Ding and S. Gao.
In our presentation we follow \cite{MR2332614}.

\begin{definition}
  \label{sec:graev-metric-groups-match}
  Let \( I \) be an interval of natural numbers.  A bijection \( \theta : I \to I \) is called a \emph{match} if
  \begin{enumerate}[(i)]
  \item \( \theta \circ \theta = \id \);
  \item there are no \( i, j \in I \) such that \( i < j < \theta(i) < \theta(j) \).
  \end{enumerate}
\end{definition}

\begin{definition}
  \label{sec:graev-metric-groups-match-word}
  Let \( w \in \word{\env{X}} \) be a word of length \( n \), let \( \theta \) be a match on \( \seg{1}{n} \).  A word
  \( w^{\theta} \) has length \( n \) and is defined as
  \begin{displaymath}
    w^{\theta}(i) =
    \begin{cases}
      e & \textrm{if \( \theta(i) = i \)}; \\
      w(i) & \textrm{if \( \theta(i) > i \)};\\
      w \big( \theta(i) \big)^{-1} & \textrm{if \( \theta(i) < i \)}.
    \end{cases}
  \end{displaymath}
\end{definition}

It is not hard to check that for any word \( w \) and any match \( \theta \) on \( \seg{1}{|w|} \) the word
\( w^{\theta} \) is trivial, i.e. \( \widehat{w^{\theta}} = e \).

\begin{theorem}[Sipacheva--Uspenskij, Ding--Gao]
  \label{sec:graev-metric-groups-graev-metric-computation}
  If \( f \in F(X)\) and \( w \in \word{\env{X}} \) is the reduced form of \( f \), then
  \[ \dist(f,e) = \min\big\{\rho\big(w,w^{\theta}\big) : \textrm{\( \theta \) is a match on \( \seg{1}{|w|} \)}
  \big\}.  \]
\end{theorem}

Here are some of the properties of the Graev metrics.  They are easy consequences of the definition of the Graev metric
and Theorem \ref{sec:graev-metric-groups-graev-metric-computation}.

\begin{proposition}
  \label{sec:graev-metric-groups-properties}
  Let \( (X,e,d) \) be a pointed metric space, and let \( \dist \) be the Graev metric on \( F(X) \).
  \begin{enumerate}[(i)]
  \item\label{prop:graev-metric-group-properties-item:extending-lipschitz} If \( (T,d_{T}) \) is a tsi group and
    \( \phi : X \to T \) is a \( K \)-\lipschitz map such that \( \phi(e) = e \), then this map extends uniquely to a
    \( K \)-\lipschitz homomorphism \( \phi : F(X) \to T \).
  \item\label{prop:graev-metric-group-properties-item:induced-metric} If \( Y \subseteq X \), \( e \in Y \) is a pointed
    subspace of \( X \) with the induced metric, then the natural embedding \( i : Y \to X \) extends uniquely to an
    isometric embedding
    \[ i : F(Y) \to F(X). \]
    Moreover, if \( Y \) is closed in \( X \), then \( F(Y) \) is closed in \( F(X) \).
  \item\label{prop:graev-metric-group-properties-item:maximality} If \( \delta \) is any tsi metric \( F(X) \) that
    extends \( d \), i.e., if \( d(x_{1},x_{2}) = \delta(x_{1},x_{2}) \) for all \( x_{1}, x_{2} \in X \), then
    \( \delta(u_{1}, u_{2}) \le \dist(u_{1},u_{2}) \) for all \( u_{1}, u_{2} \in F(X) \).  In other words, \( \dist \)
    is maximal among all the tsi metrics that extend \( d \).
  \item If \( X \ne \{e\} \), then
    \[ \chi(F(X)) = \max \{\aleph_{0}, \chi(X)\}.  \] In particular, if \( X \) is separable, then so is \( F(X) \).
  \end{enumerate}
\end{proposition}

\subsection{Free groups over metric groups}
\label{sec:free-groups-over}
In this subsection we prove a technical result that will be used later in Section \ref{sec:prop-graev-metr}.

Suppose \( X \) is itself a group and \( e \in X \) is the identity element of that group.  Let \( \circ \) denote the
multiplication operation on \( X \), and let \( {x}^{\dagger} \) denote the group inverse of an element \( x \in X \).
Suppose also that \( d \) is a two sided invariant metric on \( X \).  For \( u \in \word{\env{X}} \) define a word
\( u^{\sharp} \) by
\begin{displaymath}
  u^{\sharp}(i) =
  \begin{cases}
    u(i) & \textrm{if \( u(i) \in X \)};\\
    (u(i)^{-1})^{\dagger} & \textrm{if \( u(i) \in X^{-1} \)}.
  \end{cases}
\end{displaymath}
For \( h \in F(X) \) let \( h^{\sharp} = \widehat{w^{\sharp}} \), where \( w \) is the reduced form of \( h \).

\begin{proposition}
  \label{sec:free-groups-over-1}
  Let \( f \in F(X) \), and let \( w \) be the reduced form of \( f \).  If \( w \in \word{X} \), then for any
  \( h \in F(X) \)
  \[ \dist(fh,e) \ge \dist(fh^{\sharp},e).  \]
\end{proposition}

\begin{proof}
  Suppose \( w \in \word{X} \) and fix an \( h \in F(X) \).  Let \( u \in \word{\env{X}} \) be the reduced form of
  \( h \).  It is enough to show that
  \[ \rho \Big( w \concat u, \big( w \concat u \big)^{\theta} \Big) \ge \rho \Big( w \concat u^{\sharp}, \big( w \concat
  u^{\sharp} \big)^{\theta} \Big) \]
  for any match \( \theta \) on \( \seg{1}{|w|+|u|} \).  This follows from the following inequalities:
  \begin{itemize}
  \item if \( x, y \in X^{-1} \), then by the two-sided invariance of the metric \( d \)
    \begin{displaymath}
      \begin{aligned}
        d(x,y) = d(x^{-1},y^{-1}) = d\big((x^{-1})^{\dagger},(y^{-1})^{\dagger}\big);
      \end{aligned}
    \end{displaymath}
  \item if \( x \in X^{-1} \) and \( y \in X \), then by the two-sided invariance of the metric \( d \)
    \begin{displaymath}
      \begin{aligned}
        d(x,y) =&\ d(x,e) + d(e,y) = d(x^{-1},e) + d(e,y) = \\
        &\ d\big((x^{-1})^{\dagger},e \big) + d(e,y) \ge d\big((x^{-1})^{\dagger},y \big).
      \end{aligned}
    \end{displaymath}
  \end{itemize}
  Thus \( \dist(fh,e) \ge \dist(fh^{\sharp},e) \).
\end{proof}
\section{Metrics on amalgams}
\label{sec:metrics-amalgams}

\subsection{Basic set up}
\label{sec:basic-set-up}

Let \( (G_{\lambda},d_{\lambda}) \) be a family of \tsi groups, \( A < G_{\lambda}\) be a common closed subgroup,
\( G_{\lambda_{1}} \cap G_{\lambda_{2}} = A \), and assume additionally that the metrics \( \{d_{\lambda}\} \) agree on
\( A \):
\[ \quad d_{\lambda_{1}}(a_{1},a_{2}) = d_{\lambda_{2}}(a_{1},a_{2}) \quad \textrm{for all \( a_{1}, a_{2} \in A \) and
  all \( \lambda_{1}, \lambda_{2} \in \Lambda \)}.  \]
Our main goal is to define a metric on the free product of \( G_{\lambda} \) with amalgamation over \( A \) that extends
all the metrics \( d_{\lambda} \).  It will be an analog of the Graev metrics on free groups.

First of all, let \( d \) denote the amalgam metric on \( G = \bigcup_{\lambda} G_{\lambda} \) given by
\begin{displaymath}
  d(f_{1},f_{2}) = 
  \begin{cases}
    d_{\lambda}(f_{1},f_{2}) & \textrm{if \( f_{1}, f_{2} \in G_{\lambda} \) for some \( \lambda \in \Lambda \);}\\
    \inf\limits_{a \in A} \big\{ d_{\lambda_{1}}(f_{1},a) + d_{\lambda_{2}}(a,f_{2}) \big\} & \textrm{if
      \( f_{1} \in G_{\lambda_{1}} \), \( f_{2} \in G_{\lambda_{2}} \) for \( \lambda_{1} \ne \lambda_{2} \)}.
  \end{cases}
\end{displaymath}
If \( \alpha_{1} \) and \( \alpha_{2} \) are two words in \( \words \) of the same length \( n \), then the value
\( \rho(\alpha_{1}, \alpha_{2}) \) is defined by
\[ \rho(\alpha_{1},\alpha_{2}) = \sum_{i=1}^{n} d\big(\alpha_{1}(i),\alpha_{2}(i)\big).  \]
Finally, for elements \( f_{1}, f_{2} \in \amalgam \) the Graev metric on the free product with amalgamation
\( \amalgam \) is defined as
\[ \dist(f_{1},f_{2}) = \inf \big\{\rho(\alpha_{1},\alpha_{2}) : |\alpha_{1}| = |\alpha_{2}| \textrm{ and } \hat{\alpha}_{i} =
f_{i}\big\}.  \]

\begin{lemma}
  \label{sec:metrics-amalgams-tsi-pseudo-metric}
  \( \dist \) is a \tsi pseudo-metric.
\end{lemma}

\begin{proof}
  It is obvious that \( \dist \) is non-negative, symmetric and attains value zero on the diagonal.  We show that it is
  two-sided invariant.  Let \( f_{1}, f_{2}, h \in \amalgam \) be given. Let \( \gamma \in \words \) be any word such
  that \( \hat{\gamma} = h \).  For any \( \alpha_{1}, \alpha_{2} \in \words \) that have the same length and are such
  that \( \hat{\alpha}_{i} = f_{i} \) we get
  \[ \rho(\alpha_{1}, \alpha_{2}) = \rho(\gamma \concat \alpha_{1}, \gamma \concat \alpha_{2}), \]
  and therefore \( \dist(hf_{1}, hf_{2}) \le \dist(f_{1},f_{2}) \).  But similarly, if \( \beta_{1}, \beta_{2} \) are of
  the same length and \( \hat{\beta}_{i} = hf_{i} \), then
  \[ \rho(\beta_{1}, \beta_{2}) = \rho(\gamma^{-1} {}\concat \beta_{1}, \gamma^{-1} {}\concat \beta_{2}), \]
  where \( \gamma^{-1} = \gamma(|\gamma|)^{-1} \concat \ldots \concat \gamma(1)^{-1} \).  Hence
  \( \dist(f_{1}, f_{2}) = \dist(hf_{1},hf_{2}) \), i.e., \( \dist \) is left invariant.  Right invariance is shown
  similarly.

  We also need to check the triangle inequality.  By the two-sided invariance triangle inequality is equivalent to
  \[ \dist(f_{1}f_{2},e) \le \dist(f_{1},e) + \dist(f_{2},e) \quad \textrm{for all \( f_{1}, f_{2} \in \amalgam \)}.  \]
  The latter follows immediately from the observation that if \( \hat{\alpha}_{i} = f_{i} \),
  \( |\alpha_{i}| = |\zeta_{i}| \), and \( \hat{\zeta}_{1} = e = \hat{\zeta_{2}} \), then
  \( \widehat{\alpha_{1} {}\concat \alpha_{2}} = f_{1}f_{2} \), \( \widehat{\zeta_{1} \concat \zeta_{2}} = e \), and
  also
  \[ \rho(\alpha_{1} \concat \alpha_{2}, \zeta_{1} \concat \zeta_{2}) = \rho(\alpha_{1}, \zeta_{1}) + \rho(\alpha_{2},
  \zeta_{2}). \qedhere \]
\end{proof}

We will show eventually that, in fact, \( \dist \) is not only a pseudo-metric, but a genuine metric.  This will take us
a while though.

It will be convenient for us to talk about norms rather than about metrics.  For this we set \( \norm(f) = \dist(f,e)
\).  Then \( \norm \) is a \tsi pseudo-norm on \( G \) (again, it will turn out to be a norm).  Note that \( \dist \) is
a metric if and only if \( \norm \) is a norm, i. e., if and only if \( \norm(f) = 0 \) implies \( f = e \).

\subsection{Reductions}
\label{sec:reductions}

We start a series of reductions and will gradually simplify the structure of \( \alpha \) in the definition of the
pseudo-norm \( \norm \).

Using the notion of an \( f \)-pair the definition of \( \norm \) can be rewritten as
\[ \norm(f) = \inf \big\{\rho(\alpha,\zeta) : \textrm{\( (\alpha,\zeta) \) is an \( f \)-pair}\}.  \]
\begin{lemma}
  \label{sec:metrics-amalgams-congruent-reduction}
  For all \( f \in \amalgam \)
  \[ \norm(f) = \inf \big\{\rho(\alpha,\zeta) : (\alpha,\zeta) \textrm{ is a \multipliable \( f \)-pair}\}.  \]
\end{lemma}

\begin{proof}
  Fix an \( f \in \amalgam \).  We need to show that for any \( f \)-pair \( (\alpha,\zeta) \) and for any
  \( \epsilon > 0 \) there is a \multipliable \( f \)-pair \( (\beta, \xi) \) such that
  \[ \rho(\beta,\xi) \le \rho(\alpha,\zeta) + \epsilon. \]
  Take an \( f \)-pair \( (\alpha,\zeta) \) and fix an \( \epsilon > 0 \).  Let \( n \) be the length of \( \alpha \).
  For an \( i \in \seg{1}{n} \) we define a pair of words \( \beta_{i}, \xi_{i} \) as follows: if
  \( \alpha(i) \cong \zeta(i) \), then \( \beta_{i} = \alpha(i) \), \( \xi_{i} = \zeta(i) \); if
  \( \alpha(i) \not \cong \zeta(i) \), then \( \beta_{i} = \alpha(i) \concat e \) and
  \( \xi_{i} = a_{i} \concat a_{i}^{-1}\zeta(i) \), where \( a_{i} \in A \) is any element such that
  \[ d\big(\alpha(i),\zeta(i)\big) + \frac{\epsilon}{n} \ge d\big(\alpha(i),a_{i}\big) + d\big(a_{i},\zeta(i)\big), \]
  which exists by the definition of the amalgam metric \( d \).  Then
  \[ \rho(\beta_{i},\xi_{i}) \le \rho\big(\alpha(i),\zeta(i)\big) + \frac{\epsilon}{n} \quad \textrm{for all \( i
    \)}.  \]
  Set \( \beta = \beta_{1} \concat \ldots \concat \beta_{n} \), \( \xi = \xi_{1} \concat \ldots \concat \xi_{n} \).  It
  is now easy to see that \( (\beta,\xi) \) is a \multipliable \( f \)-pair and that indeed
  \[ \rho(\beta,\xi) \le \rho(\alpha,\zeta) + \epsilon. \qedhere \]
\end{proof}

The next lemma follows immediately from the two-sided invariance of the metrics \( d_{\lambda} \).

\begin{lemma}
  \label{sec:metrics-amalgams-transfer-isometry}
  Let \( (\alpha,\zeta) \) be a \multipliable pair of length \( n \), and let \( \{I_{k}\}_{k=1}^{m} \) be a right [left]
  transfer admissible sequence of intervals. If \( (\beta,\xi) \) is the right [left] \( \{I_{k}\}_{k=1}^{m} \)-transfer
  of the pair \( (\alpha, \zeta) \), then
  \[ \rho(\alpha,\zeta) = \rho(\beta,\xi).  \]
\end{lemma}

\begin{lemma}
  \label{sec:metrics-amalgams-slim-reduction}
  Let \( (\alpha,\zeta) \) be a \multipliable \( f \)-pair, and let \( T_{\zeta} \) be an evaluation tree for \( \zeta \).
  There is a slim \( f \)-pair \( (\beta,\xi) \) such that
  \begin{enumerate}[(i)]
  \item\label{lem:slim-reduction:item-same-length} \( |\alpha| = |\beta| \);
  \item\label{lem:slim-reduction:item-same-rho} \( \rho(\alpha,\zeta) = \rho(\beta,\xi) \);
  \item\label{lem:slim-reduction:item-same-eval-tree} \( T_{\zeta} \) is a slim evaluation tree for \( \xi \);
  \item\label{lem:slim-reduction:item-still-balanced} if \( T_{\zeta} \) is a balanced evaluation tree for \( \zeta \),
    then it is also balanced as an evaluation tree for \( \xi \).
  \end{enumerate}
\end{lemma}

\begin{proof}
  Let \( (\alpha,\zeta) \) be a \multipliable \( f \)-pair, let \( T_{\zeta} \) be an evaluation tree for \( \zeta \), and
  let \( H_{T_{\zeta}} \) denote the height of the tree \( T_{\zeta} \).  We do an inductive construction of words
  \( (\beta_{k},\xi_{k}) \) for \( k = 0, \ldots, H_{T_{\zeta}} \) and claim that
  \( (\beta_{H_{T_{\zeta}}},\xi_{H_{T_{\zeta}}}) \) is as desired.  We start by setting
  \( (\beta_{0},\xi_{0}) = (\alpha,\zeta) \).

  Suppose the pair \( (\beta_{k},\xi_{k}) \) has been constructed.  Let \( t_{1}, \ldots, t_{m} \in T \) be all the
  nodes at the level \( H_{T_{\zeta}}-k \) listed in the increasing order: \( M(I_{t_{i}}) < m(I_{t_{i+1}}) \).  We
  define a relation \( \sim \) on \( \seg{1}{m} \) by setting \( k \sim l \) if for any
  \( i \in \seg{m(I_{t_{k}}\cup I_{t_{l}})}{M(I_{t_{k}} \cup I_{t_{l}})} \) there is \( j \in \seg{1}{m} \) such that
  \( i \in I_{t_{j}} \).  It is straightforward to check that \( \sim \) is an equivalence relation on \( \seg{1}{m} \).
  Note that any \( \sim \)-equivalence class is a sub-interval of \( \seg{1}{m} \).  Let \( J_{1}, \ldots, J_{p} \) be
  the increasing list of all the distinct equivalence classes, \( J_{1} < J_{2} < \ldots < J_{p} \).
  
  \setcounter{case}{0}
  \begin{case}
    \label{sec:reductions-case-1}
    \( p \ge 2 \).  Set \( (\gamma,\omega) \) to be the right \( \{I_{t_{r}}\}_{r=1}^{M(J_{p-1})} \)-transfer of
    \( (\beta_{k}, \xi_{k}) \), and define \( (\beta_{k+1}, \xi_{k+1}) \) to be the left
    \( \{I_{t_{r}}\}_{r=m(J_{p})}^{m} \)-transfer of \( (\gamma,\omega) \).
  \end{case}

  \begin{case}
    \label{sec:reductions-case-2}
    \( p =1 \).  Suppose there is only one equivalence class.  We have a trichotomy:
    \begin{itemize}
    \item if \( M(I_{M(J_{1})}) < n \), then set
      \[ (\beta_{k+1},\xi_{k+1}) = \textrm{the right } \{I_{t_{r}}\}_{r=1}^{m} \textrm{-transfer of }
      (\beta_{k},\xi_{k}); \]
    \item if \( M(I_{M(J_{1})}) = n \), but \( m(I_{m(J_{1})}) > 1 \), then set
      \[ (\beta_{k+1},\xi_{k+1}) = \textrm{the left } \{I_{t_{r}}\}_{r=1}^{m}\textrm{-transfer of }
      (\beta_{k},\xi_{k}); \]
    \item if \(m(I_{m(J_{1})}) = 1 \) and \( M(I_{M(J_{1})}) = n \), then set
      \[ (\beta_{k+1},\xi_{k+1}) = \textrm{the right } \{I_{t_{r}}\}_{r=1}^{m-1} \textrm{-transfer of }
      (\beta_{k},\xi_{k}).  \]
      Notice the difference from the first case: the last element of the transfer sequence is \( r = m-1 \), not \( m
      \).
    \end{itemize}
  \end{case}
 
  Denote \( (\beta_{H_{T_{\zeta}}}, \xi_{H_{T_{\zeta}}}) \) simply by \( (\beta,\xi) \).  We claim that this pair
  satisfies all the requirements.  Since \( (\beta,\xi) \) is obtained by the sequence of transfers, items
  \eqref{lem:slim-reduction:item-same-length} and \eqref{lem:slim-reduction:item-still-balanced} follow from Lemma
  \ref{sec:triv-words-amalg-transfer-preserves}.  Item \eqref{lem:slim-reduction:item-same-rho} is a consequence of
  Lemma \ref{sec:metrics-amalgams-transfer-isometry}.

  It remains to check that \( \hat{\xi}[I_{t}] = e \) for all \( t \in T_{\zeta} \).  By item
  \eqref{lem:transfer-preserves-item:triviality-of-transfer-intervals} of Lemma
  \ref{sec:triv-words-amalg-transfer-preserves} \( \hat{\xi}_{k+1}[I_{t}] = e \) for all \( t \in T_{\zeta} \) such that
  \( H_{T_{\zeta}}(t) = H_{T_{\zeta}} - k \).  Therefore it is enough to show that
  \( \hat{\xi}_{k+1}[I_{t}] = \hat{\xi}_{k}[I_{t}] \) for all \( t \in T_{\zeta} \) such that
  \( H_{T_{\zeta}}(t) > H_{T_{\zeta}} - k \).  This follows from item
  \eqref{lem:transfer-preserves-item:indices-under-change} of Lemma \ref{sec:triv-words-amalg-transfer-preserves} and
  item \eqref{lem:structure-item:strict-inclusion} of the definition of the evaluation tree.
\end{proof}

\begin{lemma}
  \label{sec:metrics-amalgams-simple-reduction}
  Let \( (\alpha,\zeta) \) be a slim \( f \)-pair, and let \( T_{\zeta} \) be a slim balanced evaluation tree for
  \( \zeta \).  There is a simple \( f \)-pair \( (\beta,\xi) \) such that
  \begin{enumerate}[(i)]
  \item\label{lem:simple-reduction:item-same-length} \( |\alpha| = |\beta| \);
  \item\label{lem:simple-reduction:item-same-rho} \( \rho(\alpha,\zeta) = \rho(\beta,\xi) \);
  \item\label{lem:simple-reduction:item-same-eval-tree} \( T_{\zeta} \) is a slim balanced evaluation tree for \( \xi
    \).
  \end{enumerate}
\end{lemma}

\begin{proof}
  Let \( (\alpha,\zeta) \) be a slim \( f \)-pair of length \( n \), and let \( T_{\zeta} \) be a slim evaluation tree
  for \( \zeta \).  Sets \( \{R_{t}\}_{t \in T_{\zeta}} \) form a partition of \( \seg{1}{n} \).  For \( t \in T \) let
  \( J^{t}_{1}, \ldots, J^{t}_{q_{t}} \) be the maximal sub-intervals of \( R_{t} \).  Let \( \{i_{k}\}_{k=1}^{m} \) be
  the list of external letters in \( \zeta \).  Set
  \[ F(J^{t}_{i}) = \{i_{k}\} \cap J^{t}_{i}.  \]
  Assume first that \( F(J_{i}^{t}) \ne \emptyset \) for all \( t \in T_{\zeta} \) and all \( i \in \seg{1}{q_{t}} \).
  Note that by item \eqref{lem:structure-item:non-trivial-interior} of the definition of the balanced evaluation tree
  this is the case once \( T \ne \{\emptyset\} \).  Set
  \[ U = \Big( \bigcup_{t \in T_{\zeta}} \bigcup_{i = 1}^{q_{t}} \seg{m(J^{t}_{i})}{M(F(J^{t}_{i}))} \Big) \setminus
  \{i_{k}\}_{k=1}^{m}, \]
  \[ V = \Big( \bigcup_{t \in T_{\zeta}} \bigcup_{i=1}^{q_{t}} \seg{M(F(J^{t}_{i}))}{M(J^{t}_{i})} \Big) \setminus
  \{i_{k}\}_{k=1}^{m}. \]
  Now write \( U = \{u_{k}\}_{k=1}^{p_{u}} \), \( V = \{v_{k}\}_{k=1}^{p_{v}} \) as increasing sequences.  Set
  \( (\gamma,\omega) \) to be the right \( \{u_{k}\} \)-transfer of the pair \( (\alpha,\zeta) \) and \( (\beta,\xi) \)
  to be the left \( \{v_{k}\} \)-transfer of \( (\gamma,\omega) \) (we view \( u_{k} \)'s and \( v_{k} \)'s as intervals
  that consist of a single point).  We claim that the pair \( (\beta,\xi) \) satisfies all the assumptions of the lemma.

  Item \eqref{lem:simple-reduction:item-same-length} follows from item \eqref{lem:transfer-preserves-item:same-length}
  of Lemma \ref{sec:triv-words-amalg-transfer-preserves}.  The latter lemma also implies that \( T_{\zeta} \) is a
  balanced evaluation tree for \( \xi \).  Item \eqref{lem:simple-reduction:item-same-rho} follows from Lemma
  \ref{sec:metrics-amalgams-transfer-isometry}.

  \eqref{lem:simple-reduction:item-same-eval-tree}.  We show that \( T_{\zeta} \) is a slim evaluation tree for
  \( \xi \).  Let \( t \in T_{\zeta} \).  Since \( T_{\zeta} \) was slim for \( \zeta \), we have
  \( \hat{\zeta}[I_{t}] = e \).  Note that if \( u_{k} \in U \cap R_{t} \), then \( u_{k}+1 \in R_{t}\) (by the
  construction of \( U \)).  Similarly for \( v_{k} \in V \), \( v_{k} \in R_{t} \) implies \( v_{k}-1 \in R_{t}\).  It
  now follows from item \eqref{lem:transfer-preserves-item:indices-under-change} of Lemma
  \ref{sec:triv-words-amalg-transfer-preserves} that \( \hat{\xi}[I_{t}] = \hat{\zeta}[I_{t}] = e \) and therefore
  \( T_{\zeta} \) is slim.

  Finally, the simplicity of \( (\beta,\xi) \) is a consequence of items
  \eqref{lem:transfer-preserves-item:indices-under-change} and
  \eqref{lem:transfer-preserves-item:triviality-of-transfer-intervals} of Lemma
  \ref{sec:triv-words-amalg-transfer-preserves}.

  So have we proved the lemma under the assumption that \( F(J_{i}^{t}) \ne \emptyset \) for all \( t \in T_{\zeta} \)
  and all \( i \in \seg{1}{q_{t}} \).  Suppose this assumption was false.  By item
  \eqref{lem:structure-item:non-trivial-interior} of the definition of the balanced evaluation tree we get
  \( T_{\zeta} = \{\emptyset\} \) and \( F(I_{\emptyset}) = \emptyset \).  Therefore \( \zeta(i) \in A \) for all
  \( i \).  Set \( (\beta,\xi) \) to be the right \( (i)_{i=1}^{n-1} \)-transfer of \( (\alpha,\zeta) \).  Then
  \( \xi = e \concat \ldots \concat e \) and obviously \( (\beta,\xi) \) is a simple \( f \)-pair of the same length and
  \( T_{\zeta} = \{\emptyset\} \) is a simple balanced evaluation tree for \( \xi \).
\end{proof}

\begin{lemma}
  \label{sec:metrics-amalgams-symmet-rho-decreases}
  Let \( (\alpha,\zeta) \) be a slim \( f \)-pair of length \( n \) with a slim evaluation tree \( T_{\zeta} \).  Let
  \( t \in T_{\zeta}\) be given and let \( \{i_{k}\}_{k=1}^{m} \subseteq R_{t} \) be a symmetrization admissible
  list. If \( \xi = \symmet{\alpha}{\zeta}{i'}{\{i_{k}\}} \) for some \(i' \in \{i_{k}\}_{k=1}^{m} \), then
  \[ \rho(\alpha,\zeta) \ge \rho(\alpha,\xi).  \]
\end{lemma}

\begin{proof}
  Since \( \zeta \) is slim, we have
  \[ \zeta(i_{1}) \cdot \zeta(i_{2}) \cdots \zeta(i_{m}) = e, \]
  and by Proposition \ref{sec:tsi-groups-tsi-criterion} we get
  \[ d(\alpha(i_{1}) \cdots \alpha(i_{m}), e) = d(\alpha(i_{1}) \cdots \alpha(i_{m}), \zeta(i_{1}) \cdots \zeta(i_{m}))
  \le \sum_{j=1}^{m} d(\alpha(i_{j}),\zeta(i_{j})).  \] If \( i' = i_{k} \), then
  \begin{eqnarray*}
    \begin{aligned}
      & \rho(\alpha,\zeta) - \rho(\alpha,\xi) = \\
      & \sum_{j=1}^{m} d\big(\alpha(i_{j}),\zeta(i_{j})\big) - d\big(\alpha(i_{k}),
      \alpha(i_{k-1})^{-1} \cdots \alpha(i_{1})^{-1} \cdot \alpha(i_{m})^{-1} \cdots \alpha(i_{k+1})^{-1}\big) = \\
      & \sum_{j=1}^{m} d\big(\alpha(i_{j}),\zeta(i_{j})\big) - d\big(\alpha(i_{1}) \cdots \alpha(i_{m}), e\big) \ge 0.
    \end{aligned}
  \end{eqnarray*}
  This proves the lemma.
\end{proof}

\begin{definition}
  \label{sec:metrics-amalgams-simple reduced}
  A simple \( f \)-pair \( (\alpha,\zeta) \) is called \emph{simple reduced} if \( \alpha \) is a reduced form of
  \( f \).
\end{definition}

\begin{lemma}
  \label{sec:metrics-amalgams-simple reduced-reduction}
  For any \( f \in \amalgam \)
  \[ \norm(f) = \inf\{\rho(\alpha,\zeta) : \textrm{\( (\alpha,\zeta) \) is a simple reduced \( f \)-pair}\}.  \]
\end{lemma}

\begin{proof}
  In view of Lemmas \ref{sec:metrics-amalgams-congruent-reduction}, \ref{sec:metrics-amalgams-slim-reduction}, and
  \ref{sec:metrics-amalgams-simple-reduction}, it is enough to show that for any simple \( f \)-pair
  \( (\alpha,\zeta) \) there is a simple reduced \( f \)-pair \( (\beta, \xi) \) such that
  \( \rho(\alpha,\zeta) \ge \rho(\beta,\xi) \).  Let \( (\alpha,\zeta) \) be a simple \( f \)-pair.  Let
  \( (\gamma, \omega) \) be a simple \( f \)-pair of the smallest length among all simple \( f \)-pairs
  \( (\gamma_{0},\omega_{0}) \) such that
  \[ \rho(\alpha,\zeta) \ge \rho(\gamma_{0},\omega_{0}).  \]
  It is enough to show that \( \gamma \) is a reduced form of \( f \).  If \( |\gamma| = 1 \) this is obvious.  Suppose
  \( |\gamma| = n \ge 2 \).

  \textbf{Claim 1.} There is no \( j \in \seg{1}{n} \) such that \( \gamma(j) \in A \).  Suppose this is false and there
  is such a \( j \in \seg{1}{n} \).
  
  \setcounter{case}{0}
  \begin{case}
    \label{sec:reductions-2}
    \( \omega(j) \in A \).  (In fact, since \( (\gamma,\omega) \) is simple, \( \omega(j) \in A \) implies
    \( \omega(j) = e \), but this is not used here.)  Suppose \( j < n \).  Since \( \gamma(j) \in A \),
    \( \omega(j) \in A \) and \( \gamma(j+1) \cong \omega(j+1) \), we have
    \( \gamma(j) \cdot \gamma(j+1) \cong \omega(j) \cdot \omega(j+1) \).  Define \( (\gamma_{1}, \omega_{1}) \) by
    \begin{displaymath}
      \gamma_{1}(i) = 
      \begin{cases}
        \gamma(i)                   & \textrm{if \( i < j \)};\\
        \gamma(j) \cdot \gamma(j+1) & \textrm{if \( i = j \)};\\
        \gamma(i+1) & \textrm{if \( i > j \)};
      \end{cases}
    \end{displaymath}
    \begin{displaymath}
      \omega_{1}(i) = 
      \begin{cases}
        \omega(i)                   & \textrm{if \( i < j \)};\\
        \omega(j) \cdot \omega(j+1) & \textrm{if \( i = j \)};\\
        \omega(i+1) & \textrm{if \( i > j \)}.
      \end{cases}
    \end{displaymath}
    It is easy to see that \( |\gamma_{1}| = |\gamma| - 1 \) and \( (\gamma_{1}, \omega_{1}) \) is a \multipliable \( f
    \)-pair. Moreover, since by the two-sided invariance
    \[ d(\gamma(j)\gamma(j+1),\omega(j)\omega(j+1)) \le d(\gamma(j),\omega(j)) + d(\gamma(j+1),\omega(j+1)), \]
    we also have \( \rho(\gamma,\omega) \ge \rho(\gamma_{1},\omega_{1}) \).  Since \( \gamma_{1}, \omega_{1} \) is a
    \multipliable \( f \)-pair, by Lemmas \ref{sec:metrics-amalgams-slim-reduction} and
    \ref{sec:metrics-amalgams-simple-reduction} there is a simple \( f \)-pair \( (\gamma_{0}, \omega_{0}) \) such that
    \( |\gamma_{0}| = |\gamma_{1}| = n-1 \) and \( \rho(\gamma_{0},\omega_{0}) = \rho(\gamma_{1},\omega_{1}) \).  This
    contradicts the choice of \( (\gamma,\omega) \).

    If \( j = n \), define
    \begin{displaymath}
      \gamma_{1}(i) = 
      \begin{cases}
        \gamma(i)                   & \textrm{if \( i < j-1 \)};\\
        \gamma(j-1) \cdot \gamma(j) & \textrm{if \( i = j-1 \)};\\
        \gamma(i+1) & \textrm{if \( i > j-1 \)};
      \end{cases}
    \end{displaymath}
    \begin{displaymath}
      \omega_{1}(i) = 
      \begin{cases}
        \omega(i)                   & \textrm{if \( i < j-1 \)};\\
        \omega(j-1) \cdot \omega(j) & \textrm{if \( i = j-1 \)};\\
        \omega(i+1) & \textrm{if \( i > j-1 \)},
      \end{cases}
    \end{displaymath}
    and proceed as before.
  \end{case}

  \begin{case}
    \label{sec:reductions-1}
    \( \omega(j) \not \in A \).  Let \( T_{\omega} \) be a slim evaluation tree for \( \omega \).  Let
    \( t \in T_{\omega} \) be such that \( j \in R_{t} \).  Let \( \{i_{k}\}_{k=1}^{m} \) be the list of external
    letters in \( R_{t} \); this list is symmetrization admissible.  Let \( j_{0} \in \{i_{k}\}_{k=1}^{m} \) be any such
    that \( j_{0} \ne j \), set \( \omega_{2} = \symmet{\gamma}{\omega}{j_{0}}{\{i_{k}\}} \).  By Lemma
    \ref{sec:metrics-amalgams-symmetrization-properties} \( (\gamma, \omega_{2}) \) is a slim \( f \)-pair and
    \( \omega_{2}(j) = \gamma(j) \in A \).  And we can decrease the length of the pair \( (\gamma,\omega_{2}) \) as in
    the previous case.  This proves the case and the claim.
  \end{case}

  \medskip
  
  \textbf{Claim 2.} There is no \( j \in \seg{1}{n-1} \) such that \( \gamma(j) \cong \gamma(j+1) \).  Suppose this is
  false and there is such a \( j \in \seg{1}{n-1} \).  Note that by the previous claim \( \gamma(j) \not \in A \) and
  \( \gamma(j+1) \not \in A \).  Hence there is \( \lambda_{0} \in \Lambda \) such that
  \[ \gamma(j),\ \gamma(j+1),\ \omega(j),\ \omega(j+1) \in G_{\lambda_{0}}. \]
  Therefore \( \gamma(j) \cdot \gamma(j+1) \cong \omega(j) \cdot \omega(j+1) \).  The rest of the proof is similar to
  what we have done in the previous claim.  Define \( (\gamma_{3}, \omega_{3}) \) by
  \begin{displaymath}
    \gamma_{3}(i) = 
    \begin{cases}
      \gamma(i)                   & \textrm{if \( i < j \)}\\
      \gamma(j) \cdot \gamma(j+1) & \textrm{if \( i = j \)}\\
      \gamma(i+1) & \textrm{if \( i > j \)}
    \end{cases}
  \end{displaymath}
  \begin{displaymath}
    \omega_{3}(i) = 
    \begin{cases}
      \omega(i)                   & \textrm{if \( i < j \)}\\
      \omega(j) \cdot \omega(j+1) & \textrm{if \( i = j \)}\\
      \omega(i+1) & \textrm{if \( i > j \)}
    \end{cases}
  \end{displaymath}
  Then \( |\gamma_{3}| = |\gamma| - 1 \), \( (\gamma_{3}, \omega_{3}) \) is a \multipliable \( f \)-pair, and
  \( \rho(\gamma,\omega) \ge \rho(\gamma_{1},\omega_{1}) \).  By Lemmas \ref{sec:metrics-amalgams-slim-reduction} and
  \ref{sec:metrics-amalgams-simple-reduction} there is a simple \( f \)-pair \( (\gamma_{0},\omega_{0}) \) such that
  \( |\gamma_{0}| = |\gamma_{3}| \) and \( \rho(\gamma_{3},\omega_{3}) = \rho(\gamma_{0},\omega_{0}) \), contradicting
  the choice of \( (\gamma,\omega) \).  The claim is proved.

  \medskip

  From the second claim it follows that \( \gamma(j) \not \cong \gamma(j+1) \) for any \( j \in \seg{1}{n-1} \) and
  therefore \( \gamma \) is reduced.
\end{proof}

\begin{proposition}
  \label{sec:metrics-amalgams-norm-lower-bound}
  Let \( f \in \amalgam \) be an element of length \( n \).  If \( \alpha \) is a reduced form of \( f \), then
  \[ \norm(f) \ge \min\{d(\alpha(i),A) : i \in \seg{1}{n}\}.  \]
\end{proposition}

\begin{proof}
  Fix a reduced form \( \alpha \) of \( f \), the word \( \alpha \) has length \( n \).  By Lemma
  \ref{sec:metrics-amalgams-simple reduced-reduction} it remains to show that for any simple reduced \( f \)-pair
  \( (\beta,\xi) \) we have
  \[ \rho(\beta,\xi) \ge \min\{d(\alpha(i),A) : i \in \seg{1}{n}\}.  \]
  Let \( (\beta,\xi) \) be a simple reduced \( f \)-pair.  Note that by Lemma \ref{sec:triv-words-amalg-reduced-forms}
  the length of \( \beta \) is \( n \).  Let \( T_{\xi} \) be a slim evaluation tree for \( \xi \), and let
  \( t \in T_{\xi} \) be a leaf (i.e., a node with no predecessors).  Since \( I_{t} \) is \( \xi \)-\multipliable and
  \( (\beta,\xi) \) is a simple reduced pair, it follows that there is \( i_{0} \in I_{t} \) such that
  \( \xi(i_{0}) = e \) (in fact, either \( \xi(m(I_{t})) = e \) or \( \xi(m(I_{t})+1) = e \)).  By Lemma
  \ref{sec:triv-words-amalg-reduced-forms} there are \( a_{1}, a_{2} \in A \) such that
  \( a_{1} \alpha(i_{0}) a_{2} = \beta(i_{0}) \).  By the two-sided invariance we get
  \[ \rho(\beta,\xi) \ge d(\beta(i_{0}),e) = d(a_{1}\alpha(i_{0})a_{2},e) = d(\alpha(i_{0}),a_{1}^{-1}a_{2}^{-1}) \ge
  d(\alpha(i_{0}),A). \qedhere \]
\end{proof}

We are now ready to prove that the pseudo-metric \( \dist \) is, in fact, a metric.

\begin{theorem}
  \label{sec:metrics-amalgams-MAIN-Graev-metric-on-products}
  If \( \dist \) is (as before) the pseudo-metric on \( \amalgam \) associated with the pseudo-norm \( \norm \),
  \( \dist(f,e) = \norm(f) \), then
  \begin{enumerate}[(i)]
  \item\label{thm:main-item:metric} \( \dist \) is a two-sided invariant metric on \( \amalgam \);
  \item\label{thm:main-item:extension} \( \dist \) extends \( d \).
  \end{enumerate}
\end{theorem}

\begin{proof}
  \eqref{thm:main-item:metric} By Proposition \ref{sec:metrics-amalgams-tsi-pseudo-metric} we know that \( \dist \) is a
  tsi pseudo-metric.  It only remains to show that \( \dist(f,e) = 0 \) implies \( f = e \).  Let \( f \in \amalgam \)
  be such that \( \dist(f,e) = 0\), and let \( \alpha \) be a reduced form of \( f \).  Suppose first that
  \( |\alpha| \ge 2 \) and therefore \( \alpha(i) \not \in A \) for all \( i \) by the definition of the reduced form.
  By Proposition \ref{sec:metrics-amalgams-norm-lower-bound} and since \( A \) is closed in \( G_{\lambda} \) for all
  \( \lambda \), we have
  \[ \dist(f,e) \ge \min \big\{ d(\alpha(i),A) : i \in \seg{1}{|\alpha|} \big\} > 0. \]

  Suppose now \( |\alpha| = 1 \) and therefore \( \alpha = f\), \( f \in G \), and the reduced form of \( f \) is
  unique.  By Lemma \ref{sec:metrics-amalgams-simple reduced-reduction} the distance \( d(f,e) \) is given as the
  infimum over all simple reduced \( f \)-pairs, but there is only one such pair: \( (f,e) \), where \( f \) is viewed
  as a letter in \( G \).  Hence \( d(f,e) = 0 \) implies \( f = e \).

  \eqref{thm:main-item:extension} Fix \( g_{1}, g_{2} \in G \) and suppose first that \( g_{1} \not \cong g_{2} \).  Let
  \( (\alpha,\zeta) \) be a simple reduced \( g_{1}g_{2}^{-1} \)-pair.  We claim that there is \( a \in A \) such that
  \( g_{1}a = \alpha(1) \), and \( a^{-1}g_{2}^{-1} = \alpha(2) \).  Indeed,
  \begin{multline*}
    \alpha(1) \alpha(2) = g_{1} g_{2}^{-1} \implies g_{2}g_{1}^{-1}\alpha(1) \alpha(2) = e \implies
    g_{1}^{-1}\alpha(1) \in A \implies \\
    \exists a \in A \textrm{ such that } \alpha(1) = g_{1}a, \textrm{ and } \alpha(2) = a^{-1}g_{2}^{-1}.
  \end{multline*}

  Moreover, since \( g_{1} \not \cong g_{2} \) and since \( (\alpha,\zeta) \) is \multipliable, we get
  \( \zeta = e \concat e\) and thus
  \begin{displaymath}
    \begin{aligned}
      \dist(g_{1},g_{2}) =&\ \dist(g_{1}g_{2}^{-1},e) = \inf\{\rho(g_{1}a \concat a^{-1}g_{2}^{-1}, e \concat e) : a \in
      A\} = \\
      & \inf \{ d(g_{1},a^{-1}) + d(a^{-1},g_{2}) : a \in A \} = d(g_{1},g_{2}).
    \end{aligned}
  \end{displaymath}

  If \( g_{1} \cong g_{2} \), then there is only one simple reduced \( g_{1}g_{2}^{-1} \)-pair, namely
  \( (g_{1}g^{-1}, e) \) and the item follows.
\end{proof}

\section{Properties of Graev metrics}
\label{sec:prop-graev-metr}

Theorem \ref{sec:metrics-amalgams-MAIN-Graev-metric-on-products} allows us to make the following
definition: the metric \( \dist \) constructed in the previous section is called the \emph{Graev
  metric} on the free product of groups \( (G_{\lambda}, d_{\lambda}) \) with amalgamation over
\( A \).

Theorem \ref{sec:graev-metric-groups-graev-metric-computation} implies that the Graev metric on a
free group is, in some sense, computable, that is if one can compute the metric on the base, then to
find the norm of an element \( f \) in the free group one has to calculate the function \( \rho \)
for only \emph{finitely many} trivial words, moreover those words are constructable from the letters
of \( f \).  For the case of free products without amalgamation, i.e., when \( A = \{e\} \), we have
a similar result (see Corollary \ref{sec:prop-graev-metr-computability-graev-for-free-products}
below).

\begin{definition}
  \label{sec:metrics-amalgams-symmetric-word}
  Let \( (\alpha,\zeta) \) be a slim \( f \)-pair with a slim evaluation tree \( T_{\zeta} \).  The
  pair \( (\alpha,\zeta) \) is called \emph{symmetric with respect to the tree \( T_{\zeta} \)} if
  for each \( t \in T_{\zeta} \) there are a symmetrization admissible list
  \( \{i_{t,k}\}_{k=1}^{m_{t}} \) and \( j_{t} \in \{i_{t,k}\}_{k=1}^{m_{t}} \) such that
  \[ \zeta = \symmet{\alpha}{\zeta}{j_{t}}{\{i_{t,k}\}_{k=1}^{m_{t}}}. \]

  An \( f \)-pair \( (\alpha,\zeta) \) is called \emph{symmetric} if there is a slim evaluation tree
  \( T_{\zeta} \) such that \( (\alpha,\zeta) \) is a symmetric \( f \)-pair with respect to
  \( T_{\zeta} \).
\end{definition}

\begin{remark}
  \label{sec:metrics-amalgams-finitely-many-symmetric}
  Note that for any word \( \alpha \) there are only finitely many words \( \zeta \) such that
  \( (\alpha,\zeta) \) is symmetric.
\end{remark}

\begin{proposition}
  \label{sec:metrics-amalgams-graev-computability}
  If \( f \in \amalgam \), then
  \[ \norm(f) = \inf\{\rho(\alpha,\xi) : (\alpha,\xi) \textrm{ is a symmetric reduced \( f
    \)-pair}\}.  \]
\end{proposition}

\begin{proof}
  By Lemma \ref{sec:metrics-amalgams-simple reduced-reduction} it is enough to show that for any
  simple reduced \( f \)-pair \( (\alpha,\zeta) \) there is a symmetric reduced \( f \)-pair
  \( (\alpha,\xi) \) such that
  \[ \rho(\alpha,\zeta) \ge \rho(\alpha,\xi).  \]
  Let \( (\alpha,\zeta) \) be a simple reduced \( f \)-pair, and let \( T_{\zeta} \) be a slim
  evaluation tree for \( \zeta \).  We construct a new slim evaluation tree \( T^{*}_{\zeta} \) for
  \( \zeta \) with the following property: for any \( t \in T^{*}_{\zeta} \) and any
  \( i \in R^{*}_{t} \) if \( \zeta(i) = e \), then \( t \) is a leaf and, moreover,
  \( R^{*}_{t} = I^{*}_{t} = \{ i \} \).

  Let \( \{j_{k}\}_{k=1}^{m} \) be such that \( \zeta(j_{k}) = e \) for all \( k \) and
  \( \zeta(j) = e \) implies \( j = j_{k} \) for some \( k \in \seg{1}{m} \).  We construct a
  sequence of slim evaluation trees \( T^{(k)}_{\zeta} \) for \( \zeta \) and claim that
  \( T_{\zeta}^{(m)} \) is as desired.  Set \( T^{(0)}_{\zeta} = T_{\zeta} \).  Suppose
  \( T^{(k)}_{\zeta} \) has been constructed.  Let \( t_{0} \in T^{(k)}_{\zeta} \) be such that
  \( j_{k+1} \in R^{(k)}_{t_{0}} \).  If \( |R_{t_{0}}^{(k)}| = 1 \), that is if
  \( R_{t_{0}}^{(k)} = I^{(k)} = \{j_{k+1}\} \), then do nothing: set
  \( T^{(k+1)}_{\zeta} = T^{(k)}_{\zeta} \).
  
  Suppose \( |R_{t_{0}}^{(k)}| > 1 \). Let \( s \) be a symbol for a new node.  For all
  \( t \in T_{\zeta}^{(k)} \setminus \{t_{0}\} \) set
  \[ T^{(k+1)}_{\zeta} = T^{(k)}_{\zeta} \cup \{s\},\ I^{(k+1)}_{t} = I^{(k)}_{t},\ I^{(k+1)}_{s} =
  \seg{j_{k+1}}{j_{k+1}} = \{j_{k+1}\}. \]
  We need to turn the set \( T_{\zeta}^{(k+1)} \) into a tree.  For that let the ordering of the
  nodes in \( T^{(k+1)}_{\zeta} \) extend the ordering of the nodes of \( T^{(k)}_{\zeta} \).  To
  finish the construction it remains to define the place for the node \( s \) inside
  \( T_{\zeta}^{(k+1)} \) and an interval \( I^{(k+1)}_{t_{0}} \).
  \begin{itemize}
  \item If \( j_{k+1} \) is the minimal element of \( R^{(k)}_{t_{0}} \), i.e., if
    \( j_{k+1} = m(R^{(k)}_{t_{0}}) \), then set
    \( I^{(k+1)}_{t_{0}} = \seg{m(I_{t_{0}}^{(k)})+1}{M(I_{t_{0}}^{(k)})} \).  Let
    \( t_{1} \in T^{(k)}_{\zeta}\) be such that \( (t_{0},t_{1}) \in E(T^{(k)}_{\zeta}) \).  Set
    \((s,t_{1}) \in E(T^{(k+1)}_{\zeta}) \), or in other words, \( s \prec t_{1} \) in \( T_{\zeta}^{(k+1)} \).
  \item If \( j_{k+1} \) is the maximal element of \( R^{(k)}_{t_{0}} \), i.e., if
    \( j_{k+1} = M(R^{(k)}_{t_{0}}) \), then set
    \( I^{(k+1)}_{t_{0}} = \seg{m(I_{t_{0}}^{(k)})}{M(I_{t_{0}}^{(k)})-1} \).  Let
    \( t_{1} \in T^{(k)}_{\zeta}\) be such that \( (t_{0},t_{1}) \in E(T^{(k)}_{\zeta}) \).  Set
    \((s,t_{1}) \in E(T^{(k+1)}_{\zeta}) \), or in other words, \( s \prec t_{1} \) in \( T_{\zeta}^{(k+1)} \).
  \item If \( j_{k+1} \) is neither maximal nor minimal element of \( R^{(k)}_{t_{0}} \), then set
    \( I^{(k+1)}_{t_{0}} = I^{(k)}_{t_{0}} \) and \( (s,t_{0}) \in E(T^{(k+1)}_{\zeta}) \).
  \end{itemize}
  It is straightforward to check that \( T_{\zeta}^{(k+1)} \) is a slim evaluation tree for
  \( \zeta \).

  Finally, we define \( T^{*}_{\zeta} = T_{\zeta}^{m} \).  Then \( T^{*}_{\zeta} \) is a slim
  evaluation tree for \( \zeta \) and, by construction, if \( j \) is such that \( \zeta(j) = e \),
  then \( I_{t_{0}}^{*} = \{j\} \) for some \( t_{0} \in T^{*}_{\zeta}\).

  Let \( \{i_{k}\}_{k=1}^{p} \) be the list of external letters of \( \zeta \).  Set
  \begin{displaymath}
    F^{*}_{t} =
    \begin{cases}
      R^{*}_{t} \cap \{i_{k}\}_{k=1}^{p} & \textrm{if \( R^{*}_{t} \cap \{i_{k}\}_{k=1}^{p} \ne \emptyset \)};\\
      I^{*}_{t} & \textrm{otherwise}.
    \end{cases}
  \end{displaymath}

  Note that \( F^{*}_{t} \) is symmetrization admissible for all \( t \).  Let
  \( \{t_{j}\}_{j=1}^{N} \) be the list of nodes of \( T_{\zeta}^{*} \).  For any
  \( j \in \seg{1}{N} \) pick some \( l_{j} \) such that \( l_{j} \in F_{t_{j}} \).  Set
  \( \xi_{0} = \zeta \) and construct inductively
  \[ \xi_{k+1} = \symmet{\alpha}{\xi_{k}}{l_{k+1}}{F_{t_{k+1}}}. \]
  Finally, set \( \xi = \xi_{N} \).  It follows from Lemma
  \ref{sec:metrics-amalgams-symmetrization-properties} that \( (\alpha,\xi) \) is a slim \( f
  \)-pair and is symmetric with respect to \( T_{\zeta}^{*} \) by construction.  Lemma
  \ref{sec:metrics-amalgams-symmet-rho-decreases} implies
  \[ \rho(\alpha,\zeta) \ge \rho(\alpha,\xi) \] as desired.
\end{proof}

If \( A = \{e\} \), that is we have a free product without amalgamation, then for any
\( f \in \amalgam \) there is exactly one reduced word \( \alpha \in \words \) such that
\( \hat{\alpha} = f \).  This observation together with Remark
\ref{sec:metrics-amalgams-finitely-many-symmetric} gives us the following
\begin{corollary}
  \label{sec:prop-graev-metr-computability-graev-for-free-products}
  If \( A = \{e\} \), then for any \( f \in \amalgam \)
  \[ N(f) = \min\{\rho(\alpha,\xi) : (\alpha,\xi) \textrm{ is a symmetric reduced \( f
    \)-pair}\}.  \]
\end{corollary}

We can now prove an analog of Proposition \ref{sec:graev-metric-groups-properties} for the Graev
metrics on the free products with amalgamation.

\begin{proposition}
  \label{sec:graev-metric-amalgam-properties}
  The Graev metric \( \dist \) has the following properties:
  \begin{enumerate}[(i)]
  \item\label{prop:graev-metric-amalgam-properties-item:extending-lipschitz} if \( (T,d_{T}) \) is a
    tsi group, \( \phi_{\lambda} : G_{\lambda} \to T \) are \( K \)-\lipschitz homomorphisms
    (\( K \) does not depend on \( \lambda \)) such that for all \( a \in A \) and all
    \( \lambda_{1}, \lambda_{2} \in \Lambda \)
    \[ \phi_{\lambda_{1}}(a) = \phi_{\lambda_{2}}(a), \]
    then there exist a unique \( K \)-\lipschitz homomorphism \( \phi : \amalgam \to T \) that
    extends \( \phi_{\lambda} \);
  \item\label{prop:graev-metric-amalgam-properties-item:induced-metric} let
    \( H_{\lambda} < G_{\lambda} \) be subgroups such that \( A < H_{\lambda} \) for all
    \( \lambda \) and think of \( \amalgamH \) as being a subgroup of \( \amalgam \).  Endow
    \( H_{\lambda} \) with the metric induced from \( G_{\lambda} \).  The Graev metric on
    \( \amalgamH \) is the same as the induced Graev metric from \( \amalgam \).  Moreover, if
    \( H_{\lambda} \) are closed subgroups, then \( \amalgamH \) is a closed subgroup \( \amalgam
    \);
  \item\label{prop:graev-metric-amalgam-properties-item:maximality} let \( \delta \) be any other
    tsi metric on the amalgam \( \amalgam \).  If \( \delta \) extends \( d \), then
    \( \delta(f_{1},f_{2}) \le \dist(f_{1},f_{2}) \) for all \( f_{1}, f_{2} \in \amalgam \), i.e.,
    \( \dist \) is maximal among all the tsi metrics that extend \( d \);
  \item\label{prop:graev-metric-amalgam-properties-item:density-character} if
    \( \Lambda' = \{\lambda \in \Lambda : G_{\lambda} \ne A\} \) and \( |\Lambda'| \ge 2 \), then
    \[ \chi(\amalgam) = \max \big\{ \aleph_{0}, \sup\{\chi(G_{\lambda}) : \lambda \in \Lambda\},
    |\Lambda'| \big\}.  \]
    In particular, if \( \Lambda \) is at most countable and \( G_{\lambda} \) are all separable,
    then the amalgam is also separable.
  \end{enumerate}
\end{proposition}

\begin{proof}
  \eqref{prop:graev-metric-amalgam-properties-item:extending-lipschitz} By the universal property
  for the free products with amalgamation there is a unique extension of the homomorphisms
  \( \phi_{\lambda} \) to a homomorphism \( \phi : \amalgam \to T \), it remains to check that
  \( \phi \) is \( K \)-\lipschitz.  Let \( (\alpha,\zeta) \) be a \multipliable \( f \)-pair of length
  \( n \).  Then
  \begin{displaymath}
    \begin{aligned}
      K\rho(\alpha,\zeta) = & \sum_{i=1}^{n} K d(\alpha(i),\zeta(i)) \ge \sum_{i=1}^{n}d_{T}\big(
      \phi(\alpha(i)),\phi(\zeta(i)) \big)
      \ge \\
      & d_{T}(\phi(\hat{\alpha}),\phi(\hat{\zeta})) = d_{T}(\phi(f),e).
    \end{aligned}
  \end{displaymath}
  And therefore
  \[ K \dist(f,e) = \inf\{ K \rho(\alpha,\zeta) : (\alpha,\zeta) \textrm{ is a \multipliable \( f
    \)-pair}\} \ge d_{T}(\phi(f),e). \] Hence \( \phi \) is \( K \)-\lipschitz.

  \eqref{prop:graev-metric-amalgam-properties-item:induced-metric} Let \( \dist_{H} \) be the Graev
  metric on \( \amalgamH \) and \( \dist \) be the Graev metric on \( \amalgam \).  From Proposition
  \ref{sec:metrics-amalgams-graev-computability} it follows that
  \( \dist_{H} = \dist |_{\amalgamH} \).

  For the moreover part suppose that \( H_{\lambda} \) are closed in \( G_{\lambda} \) for all \( \lambda \in \Lambda
  \).  Set \( H = \bigcup_{\lambda \in \Lambda} H_{\lambda} \).  Note that \( H \) is a closed subset of \( G \) by the
  definition of the metric on \( G \).  Suppose towards a contradiction that there exists \( f \in \amalgam \) such that
  \( f \not \in \amalgamH \), but \( f \in \overline{\amalgamH} \).  Let \( \alpha \in \words \) be a reduced form of
  \( f \), and let \( n = |\alpha| \).  Set
  \begin{align*}
    \label{eq:1}
    &\epsilon_{1} = \min \big\{ d(\alpha(i), A) : i \in \seg{1}{n} \big\}, \\
    &\epsilon_{2} = \min \big\{ d(\alpha(i), H) : i \in \seg{1}{n},\ \alpha(i) \not \in H \big \}.
  \end{align*}
  Note that \( \epsilon_{1} > 0 \) and \( \epsilon_{2} > 0 \).  Let \( i_{0} \in \seg{1}{n} \) be
  the largest such that \( \alpha(i_{0}) \not \in H \).  By Lemma
  \ref{sec:triv-words-amalg-reduced-forms} the numbers \( \epsilon_{i} \) and \( i_{0} \) are
  independent of the choice of the reduced form \( \alpha \).  Set
  \( \epsilon = \min\{\epsilon_{1},\epsilon_{2}\}. \) Let \( h \in \amalgamH \) be such that
  \( \dist(f,h) < \epsilon \).  By Lemma \ref{sec:metrics-amalgams-simple reduced-reduction} there
  is a simple reduced \( fh^{-1} \)-pair \( (\beta, \xi) \) such that
  \( \rho(\beta,\xi) < \epsilon \).  Let \( T_{\xi} \) be a slim evaluation tree for \( \xi \), and
  let \( t_{0} \in T_{\xi} \) be such that \( i_{0} \in R_{t_{0}} \).  It is easy to see that there
  is a word \( \alpha' \) such that \( \alpha' \) is a reduced form of \( f \),
  \( \alpha'(i) = \beta(i) \) for all \( i \in \seg{1}{i_{0}-1} \), and
  \( \alpha'(i_{0}) = \beta(i_{0}) \cdot h_{0} \) for some \( h_{0} \in H \).  Without loss of
  generality assume that \( \alpha' = \alpha \).  Note that \( \beta(i) \in H \) for all \( i > i_{0} \).

  We claim that \( i_{0} = m(R_{t_{0}}) \).  Suppose not.  Let \( j_{0} \in R_{t_{0}} \) be such that
  \( j_{0} < i_{0} \) and \( \seg{j_{0}+1}{i_{0}-1} \cap R_{t_{0}} = \emptyset \) (i.e., \( j_{0} \)
  is the predecessor of \( i_{0} \) in \( R_{t_{0}} \)). Let \( I = \seg{j_{0}+1}{i_{0}-1} \).
  Because \( T_{\xi} \) is slim, \( \hat{\xi}[I] = e \).  Since \( \beta \) is reduced and \(
  (\beta,\xi) \) is \multipliable, there is
  \( i_{1} \in I \) such that \( \xi(i_{1}) \in A \) (in fact, \( \xi(i_{1}) = e \)).  But then
  \[ \rho(\beta,\xi) \ge d(\beta(i_{1}),\xi(i_{1})) \ge d(\alpha(i_{1}),A) \ge \epsilon,  \]
  contradicting the assumption \( \rho(\beta,\xi) < \epsilon \).  The claim is proved.

  Therefore \( i_{0} = m(R_{t_{0}}) \).  Let \( \{j_{k}\}_{k=1}^{m} \) be the list of external
  letters of \( \xi \), and let \( F_{t_{0}} = R_{t_{0}} \cap \{j_{k}\}_{k=1}^{m} \).  We know that
  \( \xi(i_{0}) \not \in A \), since otherwise \( \rho(\beta,\xi) \ge \epsilon \).  Thus
  \( i_{0} \in F_{t_{0}} \).  Let \( \xi' = \symmet{\beta}{\xi}{i_{0}}{F_{t_{0}}} \).  By Lemma
  \ref{sec:metrics-amalgams-symmet-rho-decreases} \( \rho(\beta,\xi) \ge \rho(\beta,\xi') \).
  Since \( \beta(i) \in H \) for all \( i > i_{0} \), we get \( \xi'(i) \in H \) for all \( i \in
  R_{t_{0}} \setminus \{i_{0}\} \).  Let \( \lambda_{0} \) be such
  that \( \xi'(i) \in H_{\lambda_{0}} \) for all \( i \in R_{t_{0}} \setminus \{i_{0}\} \).  
  Since \( \hat{\xi}'[R_{t_{0}}] = e \), it follows
  that \( \xi'(i_{0}) \in H_{\lambda_{0}} \) as well.  Finally, we get
  \[ \rho(\beta,\xi) \ge \rho(\beta,\xi') \ge d(\beta(i_{0}),\xi'(i_{0})) \ge d(\alpha(i_{0}),H_{\lambda_{0}}) \ge
  \epsilon,  \] contradiction the choice of \( (\beta,\xi) \).  Therefore there is no \( f \in
  \overline{\amalgamH} \) such that \( f \not \in \amalgamH \).

  \eqref{prop:graev-metric-amalgam-properties-item:maximality} Let \( f \in \amalgam \) be given,
  let \( (\alpha,\zeta) \) be a \multipliable \( f \)-pair of length \( n \).  Since \( \delta \)
  extends \( d \), we get
  \[ \delta(f,e) \le \sum_{i=1}^{n} \delta(\alpha(i), \zeta(i)) = \sum_{i=1}^{n}
  d(\alpha(i),\zeta(i)). \]
  By taking the infimum over all such pairs \( (\alpha,\zeta) \) we get
  \( \delta(f,e) \le \dist(f,e) \).  By the left invariance
  \( \delta(f_{1},f_{2}) \le \dist(f_{1},f_{2}) \) for all \( f_{1}, f_{2} \in \amalgam \).

  \eqref{prop:graev-metric-amalgam-properties-item:density-character} If \( |\Lambda'| \ge 2 \), then
  \( \amalgam \) is an infinite metric space, therefore \( \chi(\amalgam) \ge \aleph_{0} \).  Since
  \( G_{\lambda} < \amalgam \), it follows that \( \chi(\amalgam) \ge \chi(G_{\lambda}) \).  We now
  show that \( \chi(\amalgam) \ge |\Lambda'| \).  It is enough to consider the case
  \( |\Lambda'| \ge \aleph_{0} \).  There is an \( \epsilon_{0} > 0 \) such that
  \[ \big| \big\{ \lambda \in \Lambda : \sup\{d(g,A) : g \in G_{\lambda}\} > \epsilon_{0} \big\} \big| = |\Lambda'|.  \]
  For any such \( \lambda \) choose a \( g_{\lambda} \in G_{\lambda} \) such that
  \( d(g_{\lambda}, A) > \epsilon_{0}\).  The family \( \{g_{\lambda}\}_{\lambda \in \Lambda} \) is
  \( 2 \epsilon_{0} \)-separated and hence \( \chi(\amalgam) \ge |\Lambda'| \).

  Finally, for the reverse inequality, let \( F_{\lambda} \subseteq G_{\lambda} \) be dense sets
  such that \( |F_{\lambda}| = \chi(G_{\lambda}) \) and \( F_{\lambda_{1}} \cap A = F_{\lambda_{2}} \cap A \) for all \(
  \lambda_{1}, \lambda_{2} \in \Lambda \).  The set
  \[ \Big\{ \hat{\alpha} : \alpha \in \word{\bigcup_{\lambda \in
      \Lambda}F_{\lambda}} \Big\} \] 
  is dense in
  \( \amalgam \) and
  \[ \Big| \word{\bigcup_{\lambda \in \Lambda} F_{\lambda}} \Big| = 
  \max \Big\{\aleph_{0}, \sup\{\chi(G_{\lambda}) : \lambda \in \Lambda\}, |\Lambda'| \Big\}.  \qedhere \]
\end{proof}

\subsection{Factors of Graev metrics.}
\label{sec:graev-metr-amalg-factors}

Note that one can naturally view \( G \) as a pointed metric space \( (G,e,d) \), and the identity
map \( G \to \amalgam \) is \( 1 \)-\lipschitz  (in fact, we have shown in Theorem
\ref{sec:metrics-amalgams-MAIN-Graev-metric-on-products} that it is an isometric embedding).  We can
construct the Graev metric on the free group \( (F(G),d_{F}) \), and by item
(\ref{prop:graev-metric-group-properties-item:extending-lipschitz}) of Proposition
\ref{sec:graev-metric-groups-properties} there is a \( 1 \)-\lipschitz homomorphism
\[ \phi : F(G) \to \amalgam \]
such that \( \phi(g) = g \) for all \( g \in G \).  Since \( G \) generates \( \amalgam \), the map
\( \phi \) is onto.  Let \( \nsbg = \mathrm{ker}(\phi) \) be the kernel of this homomorphism.  If
\( d_{0} \) is the factor metric on \( F(G) / \nsbg \) (see the remark after Proposition
\ref{sec:tsi-groups-factor-metric}), then \( (F(G) / \nsbg, d_{0}) \) is a tsi group and
\( F(G) / \nsbg \) is isomorphic to \( \amalgam \) as an abstract group.

\begin{proposition}
  \label{sec:metrics-amalgams-factor-isometry}
  In the above setting \( (F(G)/\nsbg, d_{0}) \) is \emph{isometrically} isomorphic to
  \( (\amalgam, \dist) \).
\end{proposition}

\begin{proof}
  We recall the definition of the factor metric: for \( f_{1}\nsbg, f_{2} \nsbg \in F(G)/\nsbg \)
  \[ d_{0}(f_{1}\nsbg, f_{2}\nsbg) = \inf\{d_{F}(f_{1}h_{1},f_{2}h_{2}) : h_{1}, h_{2}\in
  \nsbg\}.  \]
  Of course, by construction \( F(G)/\nsbg \) is isomorphic to \( \amalgam \) and we check that the
  natural isomorphism is an isometry.  

  Let \( f \in \amalgam \), and let \( w \in \words \) be a reduced form of \( f \).  We can
  naturally view \( w \) as a reduced form of the element in \( F(G) \), call it \( f' \).  
  It is enough to show
  that for any such \( f \) and \( f' \) we have
  \[ d_{0}(f'\nsbg, \nsbg) = \dist(f,e).  \] 
  Note that if \( h \in \nsbg \), then \( h^{\sharp} \in \nsbg\) (for the definition of \( h^{\sharp} \) see
  Subsection \ref{sec:free-groups-over}).  Therefore by Proposition \ref{sec:free-groups-over-1}
     \[ d_{0}(f'\nsbg, \nsbg) = \inf\{ d_{F}(f'h,e) : h \in \nsbg \} = \inf\{d_{F}(f'h^{\sharp},e) :
     h \in \nsbg \}.\]
  Let \( h \in \nsbg \) and \( \gamma \in \word{G} \) be the reduced form of \( h^{\sharp} \in F(G) \), we claim that
  \[ d_{F}(f'h^{\sharp},e) = \inf \Big\{ \rho \big( w \concat \gamma, (w \concat \gamma)^{\theta} \big) :
  \textrm{\( \theta \) is a match on \( \seg{1}{|w \concat \gamma|} \)} \Big\}. \]
  In general, \( w \concat \gamma \) may not be reduced, so let \( w = w_{0} \concat \alpha \),
  \( \gamma = \alpha^{-1} \concat \gamma_{0} \) be such that \( w_{0} \concat \gamma_{0} \) is reduced.  By Theorem
  \ref{sec:graev-metric-groups-graev-metric-computation}
  \[ d_{F}(f'h^{\sharp},e) = \inf \Big\{ \rho \big( w_{0} \concat \gamma_{0}, (w_{0} \concat \gamma_{0})^{\theta}
  \big) : \textrm{\( \theta \) is a match on \( \seg{1}{|w_{0} \concat \gamma_{0}|} \)} \Big\}. \]
  To see the claim it remains to note that for any match \( \theta \) on \( \seg{1}{|w_{0} \concat \gamma_{0}|} \) there
  is a canonical match \( \theta' \) on \( \seg{1}{|w \concat \gamma|} \) such that
  \[ \rho \big(w_{0} \concat \gamma_{0}, (w_{0} \concat \gamma_{0})^{\theta} \big) = \rho\big(w \concat \gamma, (w \concat
  \gamma)^{\theta'}\big). \] 
  The match \( \theta' \) can formally be defined by
  \begin{displaymath}
    \theta'(i) =
    \begin{cases}
      \theta(i)            & \textrm{if \( i \le |w_{0}| \) and \( \theta(i) \le |w_{0}| \)},\\
      \theta(i)+2|\alpha| & \textrm{if \( i \le |w_{0}| \) and \( \theta(i) >  |w_{0}| \)},\\
      2|w| - i + 1        & \textrm{if \( |w_{0}| < i \le |w_{0}|+2|\alpha| \)},\\
      \theta(i-2|\alpha|) & \textrm{if \( i  > |w_{0}| + 2|\alpha| \)  and \( \theta(i-2|\alpha|) \le |w_{0}| \)},\\
      \theta(i-2|\alpha|)+2|\alpha| & \textrm{if \( i  > |w_{0}| + 2|\alpha| \) and \( \theta(i-2|\alpha|) >  |w_{0}| \)}.\\
    \end{cases}
  \end{displaymath}
  Therefore
  \[ d_{F}(f'h^{\sharp},e) = \inf \Big\{ \rho \big( w \concat \gamma, (w \concat \gamma)^{\theta} \big) :
  \textrm{\( \theta \) is a match on \( \seg{1}{|w \concat \gamma|} \)} \Big\}. \]
  Since \( w, \gamma \in \word{G} \) and since \( \hat{\gamma} = e \), we get
  \( \dist(f,e) \le d_{0}(f'\nsbg,\nsbg). \) Since \( f \) was arbitrary and because of the
  left invariance of the metrics \( \dist \) and \( d_{0} \), we get \( \dist \le d_{0} \).

  For the reverse inequality note that \( d_{0} \) is a two-sided invariant metric on \( \amalgam \)
  and it extends the metric \( d \) on \( G \), therefore by item
  (\ref{prop:graev-metric-amalgam-properties-item:maximality}) of Proposition
  \ref{sec:graev-metric-amalgam-properties} we have \( d_{0} \le \dist \) and hence
  \( d_{0} = \dist \).
\end{proof}

\subsection{Graev metrics for products of Polish groups}
\label{sec:graev-metr-prod}

We would like to note that the construction of metrics on the free products with amalgamation works
well with respect to group completions.  Let us be more precise.  Suppose we start with tsi groups
\( (G_{\lambda}, d_{\lambda}) \) and a common closed subgroup \( A < G_{\lambda} \), assume
additionally that all the groups \( G_{\lambda} \) are complete as metrics spaces.  The group
\( (\amalgam, \dist) \), in general, is not complete, so let's take its group completion (for tsi
groups this is the same as the metric completion), which we denote by
\( (\overline{\amalgam}, \dist) \).  We have an analog of item
(\ref{prop:graev-metric-amalgam-properties-item:extending-lipschitz}) of Proposition
\ref{sec:graev-metric-amalgam-properties} for complete tsi groups.  But first we need a simple
lemma.

\begin{lemma}
  \label{sec:further-remarks-extension-to-complete-group}
  Let \( (H_{1}, d_{1}) \) and \( (H_{2}, d_{2}) \) be complete tsi groups, \( \Lambda < H_{1}\) be
  a dense subgroup and \( \phi : \Lambda \to H_{2} \) be a \( K \)-\lipschitz homomorphism.  Then
  \( \phi \) extends uniquely to a \( K \)-\lipschitz homomorphism
  \[ \psi : H_{1} \to H_{2}.  \]
\end{lemma}

\begin{proof}
  Let \( h \in H_{1} \) and let \( \{b_{n}\}_{n=1}^{\infty} \subseteq \Lambda \) be such that
  \( b_{n} \to h \).  Since \( \psi \) is \( K \)-\lipschitz, we have
  \[  d_{2}(\psi(b_{n}),\psi(b_{m})) \le K d_{1}(b_{n},b_{m}).  \]
  Hence \( \{\psi(b_{n})\}_{n=1}^{\infty} \) is a \( d_{2} \)-Cauchy sequence, and thus there is \(
  f \in H_{2}\) such that \( \psi(b_{n}) \to f \).  Set \( \psi(h) = f \).  This extends \( \psi \)
  to a map \( \psi : H_{1} \to H_{2} \) and it is easy to see that is extension is still \( K \)-\lipschitz.
\end{proof}

Combining the above result with item
\eqref{prop:graev-metric-amalgam-properties-item:extending-lipschitz} of Proposition
\ref{sec:graev-metric-amalgam-properties} we get

\begin{proposition}
  \label{sec:further-remarks-extension-lipschitz-comlete-groups}
  Let \( (T,d_{T}) \) be a complete tsi group, let \( \phi_{\lambda} : G_{\lambda} \to T \) be \( K
  \)-\lipschitz homomorphisms such that for all \( a \in A \) and all
  \( \lambda_{1}, \lambda_{2} \in \Lambda \)
  \[ \phi_{\lambda_{1}}(a) = \phi_{\lambda_{2}}(a).  \]
  There exist a unique \( K \)-\lipschitz homomorphism \( \phi : \overline{\amalgam} \to T \) such that
  \( \phi \) extends \( \phi_{\lambda} \) for all \( \lambda \).
\end{proposition}

This proposition together with item \eqref{prop:graev-metric-amalgam-properties-item:density-character} of Proposition
\ref{sec:graev-metric-amalgam-properties} shows that there are countable
coproducts in the category of tsi Polish metric groups and \( 1 \)-\lipschitz homomorphisms.

\subsection{Tsi groups with no Lie sums and Lie brackets}
\label{sec:tsi-groups-with-no-Lie}

In \cite{MR2541347} L. van den Dries and S. Gao gave an example of a group, which they denote by \( F \),
and a two-sided invariant metric \( d \) on \( F \) such that the completion \( (\overline{F}, d) \)
of this group has neither Lie sums nor Lie brackets.  More precisely, they constructed two
one-parameter subgroups
\[ A_{i} = \Big (f_{t}^{(i)} \Big)_{t \in \mathbb{R}} < \overline{F} \quad i =1,2, \]
such that neither Lie sum nor Lie bracket of \( A_{1} \) and \( A_{2} \) exist.  

Their group can be nicely explained in out setting.  It turns out that the group \( F \) that they
have constructed is isometrically isomorphic to the group \( \mathbb{Q} * \mathbb{Q} \) with the
Graev metric (and the metrics on the copies of the rationals are the usual absolute-value metrics).
The group completion of \( \mathbb{Q} * \mathbb{Q} \) is then the same as the group completion of
the group \( \mathbb{R} * \mathbb{R} \) with the Graev metric.  And moreover, \( A_{1} \) and
\( A_{2} \) are just the one-parameter subgroups given by the \( \mathbb{R} \) factors.

\section{Metrics on SIN groups}
\label{sec:free-prod-topol}

Recall that topological group is SIN if for every open neighborhood of the identity there is
a smaller open neighborhood \( V \subseteq G \) such that \( gVg^{-1} = V \) for all \( g \in G \).
SIN stands for Small Invariant Neighborhoods.  It is well-knows that a metrizable topological group
admits a compatible two-sided invariant metric if and only if it is a SIN group.

Suppose \( G_{\lambda} \) are metrizable topological groups that admit compatible two-sided
invariant metrics and \( A < G_{\lambda}\) is a common closed subgroup.  It is natural to ask whether one
can find compatible tsi metrics \( d_{\lambda} \) that agree on \( A \).  

\begin{question}
  \label{sec:further-remarks-metrics-agree-on-subgroup}
  Let \( G_{1} \) and \( G_{2} \) be metrizable SIN topological groups, and let \( A < G_{i} \) be a
  common closed subgroup.  Are there compatible tsi metrics \( d_{i} \) on \( G_{i} \) such that
  \[ d_{1}(a_{1},a_{2}) = d_{2}(a_{1},a_{2}) \] 
  for all \( a_{1}, a_{2} \in A \)?
\end{question}
We do not know the answer to this question.  Before discussing some partial results let us recall the
notion of a \birkhoff-Kakutani family of neighborhoods.

\begin{definition}
  \label{sec:further-remarks-birkhoff-kakutani-family}
  Let \( G \) be a topological group.  A family \( \{U_{i}\}_{i=0}^{\infty} \) of open
  neighborhoods of the identity \(e \in G \) is called \emph{\birkhoff-Kakutani} if the following
  conditions are met:
  \begin{enumerate}[(i)]
  \item \( U_{0} = G \);
  \item \( \bigcap_{i} U_{i} = e\);
  \item \( U_{i}^{-1} = U_{i} \);
  \item \( U_{i+1}^{3} \subseteq U_{i} \).
  \end{enumerate}
  If additionally
  \begin{enumerate}[(i)]
  \item[(v)] \( gU_{i}g^{-1} = U_{i} \) for all \( g \in G \),
  \end{enumerate}
  then the sequence is called \emph{conjugacy invariant}.
\end{definition}

It is well known (see, for example, \cite{MR2455198}) that a topological group \( G \) admits a
\birkhoff-Kakutani family if and only if it is metrizable.  Moreover, let
\( \{U_{i}\}_{i=0}^{\infty} \) be a \birkhoff-Kakutani family in a group \( G \), for
\( g_{1}, g_{2} \in G \) set
\[ \eta(g_{1},g_{2}) = \inf\{ 2^{-n} : g_{2}^{-1}g_{1} \in U_{n}\}, \]
\[ d(g_{1},g_{2}) = \inf\Big\{ \sum_{i=1}^{n-1} \eta(f_{i},f_{i+1}) : \{f_{i}\}_{i=1}^{n} \subseteq
G,\ f_{1} = g_{1}, f_{n} = g_{2}\Big\}.  \]
Then the function \( d \) is a compatible left invariant metric on \( G \) and for all \( g_{1}, g_{2} \in G \)
\begin{displaymath}
  \frac{1}{2} \eta(g_{1},g_{2}) \le d(g_{1},g_{2}) \le \eta(g_{1},g_{2}).  
\end{displaymath}
We call this metric \( d \) a \emph{\birkhoff-Kakutani metric} associated with the family \( \{U_{i}\} \).

A metrizable topological group admits a compatible tsi metric if and only if there is a conjugacy
invariant \birkhoff-Kakutani family, and moreover, if \( \{U_{i}\} \) is conjugacy invariant, then the metric \( d \)
constructed above is two-sided invariant.

\begin{proposition}
  \label{sec:further-remarks-common-bilipschitz-metric}
  Let \( G_{1} \) and \( G_{2} \) be metrizable SIN groups, let \( A < G_{i} \) be a common
  subgroup.  There are compatible tsi metrics \( d_{i} \) on \( G_{i} \) such that \( d_{1}|_{A} \)
  is bi-\lipschitz equivalent to \( d_{2}|_{A} \), i.e, there is \( K > 0 \) such that
  \[  \frac{1}{K} d_{1}(a_{1},a_{2}) \le d_{2}(a_{1},a_{2}) \le K d_{1}(a_{1},a_{2})  \]
  for all \( a_{1}, a_{2} \in A \).
\end{proposition}

\begin{proof}
  Since \( G_{1} \) and \( G_{2} \) are metrizable, we can fix two compatible metrics \( \mu_{1} \)
  and \( \mu_{2} \) on \( G_{1} \) and \( G_{2} \) respectively such that \( \mu_{i} \)-\(\diam{G_{i}} < 1 \).
  We construct conjugacy invariant \birkhoff-Kakutani families \( \{U_{i}^{(j)}\}_{i=0}^{\infty} \)
  for \( G_{j} \), \( j = 1,2 \), such that
  \begin{enumerate}[(i)]
  \item \( U_{2i+1}^{(1)} \cap A \subseteq U_{2i}^{(2)} \cap A \);
  \item \( U_{2i+2}^{(2)} \cap A \subseteq U_{2i+1}^{(1)} \cap A \).
  \end{enumerate}
  For the base of construction let \( U_{0}^{j} = G_{j} \).  Suppose we have constructed
  \( \{U_{i}^{(j)}\}_{i=1}^{N} \) and suppose \( N \) is even (if \( N \) is odd, switch the roles of
  \( G_{1} \) and \( G_{2} \)).  If \( V = U_{N}^{(2)} \cap A \), then \( V \) is an open
  neighborhood of the identity in \( A \) and therefore there is an open set \( U \subseteq G_{1} \) such
  that \( U \cap A = V \).  Let \( U_{N+1}^{(1)} \subseteq G_{1} \) be any open neighborhood of the
  identity such that \( (U_{N+1}^{(1)})^{-1} = U_{N+1}^{(1)} \),
  \( gU_{N+1}^{(1)}g^{-1} = U_{N+1}^{(1)} \) for all \( g \in G_{1}\), \(
  \mu_{1}\)-\(\diam{U^{(1)}_{N+1}} < 1/N \) and
  \[ (U_{N+1}^{(1)})^{3} \subseteq U \cap U_{N}^{(1)}.  \]
  Such a \( U_{N+1}^{(1)} \) exists because \( G_{1} \) is SIN.  Set
  \( U_{N+1}^{(2)} \) to be any open symmetric neighborhood of \( e \in G_{2} \) such that
  \( (U_{N+1}^{(2)})^{3} \subseteq U_{N}^{(2)} \).

  It is straightforward to check that such sequences \( \{U_{i}^{(j)}\}_{i=1}^{\infty} \) indeed
  satisfy all the requirements.  If \( d_{j} \) are the \birkhoff-Kakutani metrics that
  correspond to the families \( \{U_{i}^{(j)}\} \), then for all \( a_{1}, a_{2} \in A \)
  \[ \frac{1}{4}\eta_{1}(a_{1},a_{2}) \le \eta_{2}(a_{1}, a_{2}) \le 4 \eta_{1}(a_{1},a_{2}), \]
  whence
  \[ \frac{1}{8}d_{1}(a_{1},a_{2}) \le d_{2}(a_{1}, a_{2}) \le 8 d_{1}(a_{1},a_{2}),  \]
  and therefore \( d_{1}|_{A} \) and \( d_{2}|_{A} \) are bi-\lipschitz equivalent with a constant
  \( K = 8 \).
\end{proof}

\begin{remark}
  \label{sec:further-remarks-infinite-bilipschitz}
  It is, of course, straightforward to generalize the above construction to the case of finitely
  many groups \( G_{j} \), but we do not know if the result is true for infinitely many
  groups \( G_{j} \).
\end{remark}

\begin{remark}
  \label{sec:further-remarks-constant-multiply}
  Note that one can always multiply the metric \( d_{2} \) by a suitable constant (which is \( 8 \) in
  the above construction) to assure that \( d_{1}|_{A} \le d_{2}|_{A} \).  We use this observation later in
  Remark \ref{sec:further-remarks-normal-subgroup}.
\end{remark}

\begin{proposition}
  \label{sec:further-remarks-extension-left-invariant}
  Let \( G \) be a topological group, \( A < G \) be a closed subgroup of \( G \), \( N_{G} \) be a
  tsi norm on \( G \), \( N_{A} \) be a tsi norm on \( A \) and suppose that for all \( a \in A \)
  \[ N_{A}(a) \le N_{G}(a).  \] 
  There exists a compatible norm \( N \) on \( G \) such that
  \begin{enumerate}[(i)]
  \item\label{ext-left-inv:extension} \( N \) extends \( N_{A} \), that is \( N_{A}(a) = N(a) \) for
    all \( a \in A \);
  \item\label{ext-left-inv:domin} \( N \le N_{G} \).
  \end{enumerate}
  If, moreover, \( A \) is a normal subgroup of \( G \), then \( N \) is two-sided invariant.
\end{proposition}

\begin{proof}
  For \( g \in G \) set
  \[ N(g) = \inf\{N_{A}(a) + N_{G}(a^{-1}g) : a \in A\}.  \]
  We claim that \( N \) is a pseudo-norm on \( G \).
  \begin{itemize}
  \item \( N(e) = 0 \) is obvious.
  \item For any \( g \in G \) and any \( a \in A \) by the two-sided invariance of \( N_{G} \)
    \[ N_{A}(a) + N_{G}(a^{-1}g) = N_{A}(a^{-1}) + N_{G}(g^{-1}a) = N_{A}(a^{-1}) + N_{G}(ag^{-1}) \]
    and therefore \( N(g) = N(g^{-1}) \).
  \item If \( g_{1}, g_{2} \in G \), then
    \begin{displaymath}
      \begin{aligned}
        N(g_{1}g_{2}) =& \inf\{N_{A}(a) + N_{G}(a^{-1}g_{1}g_{2}) : a \in A\} = \\
        &\inf\{N_{A}(a_{1}a_{2}) + N_{G}(a_{2}^{-1}a_{1}^{-1}g_{1}g_{2}) : a_{1}, a_{2} \in A\} \le \\
        &\inf\{N_{A}(a_{1}) + N_{A}(a_{2}) + N_{G}(a_{1}^{-1}g_{1}) + N_{G}(g_{2}a_{2}^{-1}): a_{1}, a_{2} \in A\} =\\
        &\inf\{N_{A}(a_{1}) + N_{G}(a_{1}^{-1}g_{1}) : a_{1} \in A\} + \\
        &\inf\{N_{A}(a_{2}) + N_{G}(a_{2}^{-1}g_{2}) : a_{2} \in A\} = \\
        &N(g_{1}) + N(g_{2}).
      \end{aligned}
    \end{displaymath}
  \end{itemize}
  Next we show that \( N \) is a compatible pseudo-norm.  For a sequence
  \( \{g_{n}\}_{n=1}^{\infty} \subseteq G \) we have
  \begin{displaymath}
    \begin{aligned}
      N(g_{n}) \to 0 \iff & \exists \{a_{n}\}_{n=1}^{\infty} \subseteq A \quad N_{A}(a_{n}) +
      N_{G}(a_{n}^{-1}g_{n}) \to 0
      \iff \\
      & \exists \{a_{n}\}_{n=1}^{\infty} \subseteq A \quad a_{n} \to e \textrm{ and } a_{n}^{-1}g_{n} \to e \iff\\
      & g_{n} \to e.
    \end{aligned}
  \end{displaymath}
  In particular, \( N \) is a norm.

  \eqref{ext-left-inv:extension} Now we claim that \( N \) extends \( N_{A} \).  Let
  \( b \in A \).  Using \( N_{G} \ge N_{A} \) we get
  \begin{displaymath}
    \begin{aligned}
      N(b) = &\inf\{N_{A}(a) + N_{G}(a^{-1}b) : a \in A\} \ge\\
      &\inf\{N_{A}(a) + N_{A}(a^{-1}b) : a \in A\} \ge N_{A}(b).
    \end{aligned}
  \end{displaymath}
  On the other hand
  \[ N(b) \le N_{A}(b) + N_{G}(b^{-1}b) = N_{A}(b), \] and therefore \( N(b) = N_{A}(b) \).

  \eqref{ext-left-inv:domin} Finally, for any \( g \in G \) we have
  \begin{displaymath}
    \begin{aligned}
      N(g) = & \inf\{N_{A}(a) + N_{G}(a^{-1}g) : a \in A\} \le\\
      & \inf\{N_{G}(a) + N_{G}(a^{-1}g) : a \in A\} \le\\
      & N_{G}(e) + N_{G}(g) = N_{G}(g),
    \end{aligned}
  \end{displaymath}
  and therefore \( N \le N_{G} \).

  For the moreover part suppose that \( A \) is a normal subgroup.  If \( g_{1} \in G \), then
  \begin{displaymath}
    \begin{aligned}
      N(g_{1} g g_{1}^{-1}) = & \inf\{ N_{A}(a) + N_{G}(a^{-1}g_{1} g g_{1} ^{-1}) : a \in A\} = \\
      &\inf\{N_{A}(g_{1}^{-1}ag_{1}) + N_{G}(g_{1}^{-1}a^{-1}g_{1} g) : a \in A\} = N(g),
    \end{aligned}
  \end{displaymath}
  and so \( N \) is two-sided invariant.
\end{proof}

\begin{remark}
  \label{sec:further-remarks-normal-subgroup}
  Proposition \ref{sec:further-remarks-common-bilipschitz-metric} (with Remark
  \ref{sec:further-remarks-constant-multiply}) and Proposition
  \ref{sec:further-remarks-extension-left-invariant} together yield a positive answer to Question
  \ref{sec:further-remarks-metrics-agree-on-subgroup} when \( A \) is a normal subgroup of one of
  \( G_{j} \).
\end{remark}

It is natural to ask whether it is really necessarily to assume in Proposition
\ref{sec:further-remarks-extension-left-invariant} the existence of a norm \( N_{G} \) such that
\( N_{A} \le N_{G} \).  The following example shows that this assumption cannot be dropped.

\begin{example-nn}
  \label{sec:free-prod-topol-heisenberg}
  Let \( G \) be the discrete Heisenberg group
  \[ G = \left \{
    \begin{pmatrix}
      1 & a & b\\
      0 & 1 & c\\
      0 & 0 & 1
    \end{pmatrix} : a, b, c \in \mathbb{Z} \right \},
  \]
  and let \( A \) be the center of \( G \)
  \[ A = \left \{
    \begin{pmatrix}
      1 & 0 & b\\
      0 & 1 & 0\\
      0 & 0 & 1
    \end{pmatrix} : b \in \mathbb{Z} \right \}.
  \]
  The subgroup \( A \) is, of course, isomorphic to the group of integers \( \mathbb{Z} \).  Let
  \( d \) be a metric on \( A \) given by the absolute value: \( d(b_{1}, b_{2}) = |b_{1} - b_{2}| \).
  We claim that this tsi metric can not be extended to a tsi (in fact, even to a left invariant)
  metric on \( G \).  Indeed, suppose there is such an extension \( \dist \).  The group \( G \)
  is generated by the three matrices:
  \[ x =
  \begin{pmatrix}
    1 & 1 & 0\\
    0 & 1 & 0\\
    0 & 0 & 1
  \end{pmatrix}, \ y =
  \begin{pmatrix}
    1 & 0 & 0\\
    0 & 1 & 1\\
    0 & 0 & 1
  \end{pmatrix}, \ \textrm{and } z =
  \begin{pmatrix}
    1 & 0 & -1\\
    0 & 1 & 0\\
    0 & 0 & 1
  \end{pmatrix}. \]
  It is easy to check that \( z^{n^{2}} = [x^{n},y^{n}] = x^{n}y^{n}x^{-n}y^{-n}. \)  Therefore 
  \[ n^{2} = d(z^{(n^{2})},e) = \dist(z^{n^{2}},e) = \dist(x^{n}y^{n}x^{-n}y^{-n},e) \le 2n \big(\dist(x,e) + \dist(y,e)
  \big), \] 
  for all \( n \), which is absurd.

\end{example-nn}

\section{Induced metrics}
\label{sec:hnn-extensions}

In this section \( (G,d) \) denotes a tsi group, and \( A < G \) is a closed subgroup.  This section
is a preparation for the HNN construction, which is given in the next section.  Let
\( \langle t \rangle \) denote a copy of the free group on one element \( t \), i.e., a copy of the
integers, with the usual metric \( d(t^{m}, t^{n}) = |m-n| \).  The Graev metric on the free product
\( G * \langle t \rangle \) is denoted again by the letter \( d \).  Consider the subgroup of the
free product generated by \( G \) and \( tAt^{-1} \); it not hard to check that, in fact, as an
abstract group it is isomorphic to the free product \( G * tAt^{-1} \).  Thus we have two metrics on
the group \( G * tAt^{-1} \): one is just the metric \( d \), the other one is the Graev metric on
this free product; denote the latter by \( \dist \).  When are these two metrics the same?  It turns
out that they are the same if and only if the diameter of \( A \) is at most \( 1 \).  The proof of
this fact is the core of this section.

We can naturally view \( \word{G \cup tAt^{-1}} \) as a subset of
\( \word{G \cup \langle t \rangle} \) by treating a letter \( tat^{-1} \in tAt^{-1} \) as a word
\( t \concat a \concat t^{-1} \in \word{G \cup \langle t \rangle} \).  In what follows we identify
\( \word{G \cup tAt^{-1}} \) with a subset of \( \word{G \cup \langle t \rangle} \).

Let \( f \in G * tAt^{-1} \) be given and let \( \alpha \in \word{G \cup tAt^{-1}}\) be the reduced
form of \( f \).  Note that since we have a free product (no amalgamation), reduced form is unique.
The word \( \alpha \in \word{G \cup \langle t \rangle} \) can be written as
\[  \alpha = g_{1} \concat t \concat a_{1} \concat t^{-1} g_{2} \concat t \concat a_{2} \concat t^{-1} \concat \cdots
\concat t \concat a_{n} \concat t^{-1} \concat g_{n+1},   \]
where \( g_{i} \in G \), \( a_{i} \in A \), and also \( g_{1}\) or \( g_{n+1} \) may be absent. 

Lemma \ref{sec:metrics-amalgams-simple reduced-reduction} implies
\[ d(f,e) = \inf\{ \rho(\alpha,\zeta) : (\alpha,\zeta) \textrm{ is a \multipliable \( f \)-pair}\}, \]
and notice that the infimum is taken over all pairs with the same first coordinate \( \alpha \) --- the
reduced form of \( f \).  We can also impose some restrictions on \( \zeta \) and change the infimum to
a minimum, but we do not need this for a moment.

\textit{In the rest of the section \( \zeta, \xi, \delta \) denote words in the alphabet \( G \cup \langle t \rangle \)}.

\subsection{Hereditary words}
\begin{definition}
  \label{sec:hereditary-pair}
  A trivial word \( \zeta \in \word{G \cup \langle t \rangle} \) is called \emph{hereditary} if
  \( \zeta(i) \in \langle t \rangle \setminus \{e\} \) implies \( \zeta(i) = t^{\pm 1} \) for all \( i \in \seg{1}{n}
  \).  A \multipliable \( f \)-pair \( (\alpha,\zeta) \), where \( f \in G* tAt^{-1} \), is called \emph{hereditary} if
  \( \alpha \) is the reduced form of \( f \), \( \zeta \) is hereditary, and moreover,
  \[ \zeta(i) = t^{\pm 1} \implies \zeta(i) = \alpha(i).  \]
\end{definition}

\begin{lemma}
  \label{sec:reduction-to-hereditary-pair}
  Let \( f \in G * tAt^{-1} \), and let \( \alpha \in \word{G \cup \langle t \rangle } \) be the reduced form of \( f
  \).  If \( (\alpha,\zeta) \) is a \multipliable \( f \)-pair, then there exists a trivial word
  \( \xi \in \word{G \cup \langle t \rangle} \) such that \( (\alpha,\xi) \) is a hereditary \( f \)-pair and
  \( \rho(\alpha,\xi) \le \rho(\alpha,\zeta) \).
\end{lemma}

\begin{proof}
  Let \( T_{\zeta} \) be an evaluation tree for \( \zeta \).  Fix \( s \in T_{\zeta} \).  Suppose there exists
  \( j \in R_{s} \) such that \( \alpha(j) = t^{\pm 1} \) and neither \( \zeta(j) = \alpha(j) \) nor \( \zeta(j) = e \).
  Since \( \zeta(j) \ne e \) and because the pair \( (\alpha,\zeta) \) is \multipliable, it must be the case that
  \( \zeta(j) = t^{M} \) for some \( M \ne 0 \).  Let \( \{i_{k}\}_{k=1}^{m} \subseteq R_{s} \) be the complete list of
  external letters of \( \zeta \) in \( R_{s} \), note that \( j \in \{i_{k}\}_{k=1}^{m} \).  Since \( R_{s} \) is
  \( \zeta \)-\multipliable, we have \( \zeta(i_{k}) \cong t \) for all \( k \in \seg{1}{m} \).  Note that since we have a
  free product, any evaluation tree is, in fact, slim, and any \multipliable \( f \)-pair is, in fact, a simple \( f
  \)-pair.  So we can perform a symmetrization.  Set
  \[ \delta = \symmet{\alpha}{\zeta}{i_{1}}{\{i_{k}\}}.  \]
  By Lemma \ref{sec:metrics-amalgams-symmet-rho-decreases} \( \rho(\alpha,\delta) \le \rho(\alpha,\zeta) \) and also for
  all \( i \in R_{s} \) we have
  \[ (\alpha(i) = \delta(i)) \textrm{ or } (\delta(i) = e) \textrm{ or } (i = i_{1}).  \]
  Let \( \epsilon_{k} \in \{-1,+1\} \) be such that \( \alpha(i_{k}) = t^{\epsilon_{k}} \).  For all
  \( k \in \seg{2}{m} \)
  \[ \delta(i_{k}) = \alpha(i_{k}) = t^{\epsilon_{k}}. \]
  Let \( N \) be such that \( \delta(i_{1}) = t^{N} \).  Note that since \( \hat{\delta}[I_{s}] = e \),
  \[ N + \epsilon_{2} + \ldots + \epsilon_{m} = 0.\] We now construct a word \( \bar{\xi} \) as follows.

  \setcounter{case}{-1}
  \begin{case}
    If \( N = 0 \) or \( N = \epsilon_{1} \), then set \( \bar{\xi} = \delta \).
  \end{case}

  In cases below we assume \( N \not \in \{0,\epsilon_{1}\} \).

  \begin{case}
    Suppose \( \sign{N} = \sign{\epsilon_{1}} \).  Find different indices \( k_{1}, \ldots, k_{|N|-1} \) such that
    \( \sign{N} = -\sign{\epsilon_{k_{p}}} \) for all \( p \in \seg{1}{|N|-1} \).  Set
    \begin{displaymath}
      \bar{\xi}(i) =
      \begin{cases}
        \delta(i) & \textrm{if \( i \not \in \{i_{k_{p}}\}_{p=1}^{|N|-1} \) and \( i \ne i_{1} \)} ;\\
        \alpha(i_{1}) & \textrm{if \( i = i_{1} \)}; \\
        e & \textrm{if \( i \in \{i_{k_{p}}\}_{p=1}^{|N|-1} \)}.
      \end{cases}
    \end{displaymath}
  \end{case}

  \begin{case}
    Suppose \( \sign{N} = -\sign{\epsilon_{1}} \).  Find different indices \( k_{1}, \ldots, k_{|N|} \) such that
    \( \sign{N} = -\sign{\epsilon_{k_{p}}} \) for all \( p \in \seg{1}{|N|} \).  Set
    \begin{displaymath}
      \bar{\xi}(i) =
      \begin{cases}
        \delta(i) & \textrm{if \( i \not \in  \{i_{k_{p}}\}_{p=1}^{|N|} \) and \( i \ne i_{1} \)};\\
        e & \textrm{if \( i \in \{i_{k_{p}}\}_{p=1}^{|N|} \) or \( i = i_{1} \)}.
      \end{cases}
    \end{displaymath}
  \end{case}

  It is easy to check that \( \rho(\alpha,\delta) = \rho(\alpha,\bar{\xi}) \) and \( \hat{\bar{\xi}} = e \).  Moreover, for
  all \( i \in R_{s} \) either \( \bar{\xi}(i) = \alpha(i) \) or
  \( \bar{\xi}(i) = e \).

  Now apply the same procedure for all \( s \in T_{\zeta} \) and denote the result by \( \xi \).  The word \( \xi \) is
  as desired.
\end{proof}

To analyze the structure of hereditary words we introduce the following notion of a structure tree.

\begin{definition}
  \label{sec:structure-tree}
  Let \( \zeta \) be a hereditary word of length \( n \).  A tree \( T_{\zeta} \) together with a function that assigns
  to a node \( s \in T_{\zeta} \) an interval \( I_{s} \subseteq \seg{1}{n} \) is called a \emph{structure tree for
    \( \zeta \)} if for all \( s', s \in T_{\zeta} \) the following conditions are met:
  \begin{enumerate}[(i)]
  \item \( I_{\emptyset} = \seg{1}{n} \);
  \item \( \hat{\zeta}[I_{s}] = e \);
  \item if \( s \ne \emptyset \), then \( \zeta(m(I_{s})) = t^{\pm 1} \) and \( \zeta(M(I_{s})) = t^{\mp 1} \) (in
    particular \( \zeta(m(I_{s})) = \zeta(M(I_{s}))^{-1} \)).
  \end{enumerate}
  Set \( R_{s} = I_{s} \setminus \bigcup_{s' \prec s} I_{s'} \); then also
  \begin{enumerate}[(i)]
    \setcounter{enumi}{4}
  \item for all \( i \in R_{s} \) if \( i \not \in \{m(I_{s}), M(I_{s})\} \), then \( \zeta(i) \in G \) (in particular
    \( R_{s} \setminus\{m(I_{s}), M(I_{s})\}\) is \( \zeta \)-\multipliable);
  \item \( \zeta(i) \in G \) for all \( i \in R_{\emptyset} \) (in general \( R_{\emptyset} \) may be empty);
  \item\label{lem:structure-item:intervals-order-str} if \( H(s) \le H(s') \) and \( I_{s'} \cap I_{s} \ne \emptyset \),
    then \( s' \prec s \) or \( s'=s \);
  \item\label{lem:structure-item:strict-inclusion-str} if \( s' \prec s \) and \( s \ne \emptyset \), then
    \[ m(I_{s}) < m(I_{s'}) < M(I_{s'}) < M(I_{s}). \]
  \end{enumerate}
\end{definition}

\begin{lemma}
  \label{sec:hnn-extensions-symmetry-of-heriditary-pair}
  If \( \zeta \) is a hereditary word of length \( n \), then
  \[ |\{i \in \seg{1}{n} : \zeta(i) = t\}| = |\{i \in \seg{1}{n} : \zeta(i) = t^{-1}\}|.  \]
\end{lemma}

\begin{proof}
  Let \( \{i_{k}\}_{k=1}^{m} \) be the list of letters such that
  \begin{enumerate}[(i)]
  \item \( \zeta(i_{k}) = t^{\epsilon_{k}} \) for some \( \epsilon_{k} \in \{-1,1\} \);
  \item \( \zeta(i) = t^{\epsilon} \), \( \epsilon \in \{-1,1\} \), implies \( i = i_{k} \) for some \( k \).
  \end{enumerate}
  Since \( \hat{\zeta} = e \), we get
  \[ \epsilon_{1} + \ldots + \epsilon_{m} = 0, \] and therefore
  \[ |\{i \in \seg{1}{n} : \zeta(i) = t\}| = |\{i \in \seg{1}{n} : \zeta(i) = t^{-1}\}|. \qedhere \]
\end{proof}

\begin{lemma}
  \label{sec:hnn-extensions-induction-step-in-structure-tree}
  Let \( \zeta \) be a hereditary word of length \( n \).  If there is \( i \in \seg{1}{n} \) such that
  \( \zeta(i) = t \), then there is an interval \( I \subseteq \seg{1}{n} \) such that
  \begin{enumerate}[(i)]
  \item \( \zeta(m(I)) = t^{\pm 1} \) and \( \zeta(M(I)) = t^{\mp 1} \);
  \item \( \zeta(i) \in G \) for all \( i \in I \setminus \{m(I), M(I)\} \);
  \item \( \hat{\zeta}[I] = e \).
  \end{enumerate}
\end{lemma}

\begin{proof}
  Let \( I_{1}, \ldots, I_{m} \) be the list of intervals such that
  \begin{enumerate}[(i)]
  \item \( \zeta(m(I_{k})) = t^{\pm 1} \), \( \zeta(M(I_{k})) = t^{\mp 1} \);
  \item \( \zeta(i) \in G \) for all \( i \in I_{k} \setminus\{m(I_{k}), M(I_{k})\} \);
  \item \( M(I_{k}) \le m(I_{k+1}) \);
  \item if \( I \) is an interval that satisfies (i) and (ii) above, then \( I = I_{k} \) for some
    \( k \in \seg{1}{m} \).
  \end{enumerate}
  It follows from Lemma \ref{sec:hnn-extensions-symmetry-of-heriditary-pair} that the list of such intervals is
  nonempty.  Let \( J_{0}, \ldots, J_{m} \) be the complementary intervals:
  \[ J_{0} = \seg{1}{m(J_{1})-1}, \quad J_{m} = \seg{M(J_{m})+1}{n}, \]
  \[ J_{k} = \seg{M(I_{k})+1}{m(I_{k+1}) +1} \quad \textrm{for \( k \in \seg{2}{m-1} \)}. \]
  Some (and even all) of the intervals \( J_{k} \) may be empty.  If for some \( j_{1}, j_{2} \in J_{k} \) we have
  \( \zeta(j_{1}) = t^{\epsilon_{1}}\), \( \zeta(j_{2}) = t^{\epsilon_{2}} \), then \( \epsilon_{1} = \epsilon_{2} \),
  and moreover, \( \zeta(M(I_{k})) = \zeta(j_{1}) = \zeta(m(I_{k+1})) \).  It is now easy to see that
  \(\hat {\zeta}[I_{k}] \ne e \) for all \( k \in \seg{1}{m} \) implies \( \hat{\zeta} \ne e\), contradicting the
  assumption that \( \zeta \) is trivial.
\end{proof}

\begin{lemma}
  \label{sec:hnn-extensions-existence-of-structure-tree}
  If \( \zeta \) is a hereditary word of length \( n \), then there is a structure tree \( T_{\zeta} \) for \( \zeta \).
\end{lemma}

\begin{proof}
  We prove the lemma by induction on \( |\{i \in \seg{1}{n} : \zeta(i) = t\}| \).  For the base of induction suppose
  that \( \zeta(i) \ne t \) for all \( i \).  By the definition of a hereditary word and by Lemma
  \ref{sec:hnn-extensions-symmetry-of-heriditary-pair} we have \( \zeta(i) \in G \) for all \( i \in \seg{1}{n} \).  Set
  \( T_{\zeta} = \{\emptyset\} \) and \( I_{\emptyset} = \seg{1}{n} \).  It is easy to see that this gives a structure
  tree.

  Suppose now there is \( i \in \seg{1}{n} \) such that \( \zeta(i) = t \).  Apply Lemma
  \ref{sec:hnn-extensions-induction-step-in-structure-tree} and let \( I \) denote an interval granted by this lemma.
  Let \( m \) be the length of \( I \).  If \( m = n \), that is if \( I = \seg{1}{n} \), then set
  \( T_{\zeta} = \{\emptyset, s\} \) with \( s \prec \emptyset \) and \( I_{s} = I_{\emptyset} = \seg{1}{n} \).  One
  checks that this is a structure tree.  Assume now that \( m < n \) .  Define a word \( \delta \) of length \( n-m \)
  by
  \begin{displaymath}
    \delta(i) =
    \begin{cases}
      \zeta(i) & \textrm{if \( i < m(I) \)} \\
      \zeta(i+m) & \textrm{if \( i \ge m(I) \)}.
    \end{cases}
  \end{displaymath}
  The word \( \delta \) is a hereditary word and
  \[ |\{ i \in \seg{1}{|\delta|} : \delta(i) = t\}| < |\{i \in \seg{1}{n} : \zeta(i) = t\}|. \]
  Therefore, by induction hypothesis, there is a structure tree \( T_{\delta} \) and intervals \( J_{s} \),
  \( s \in T_{\delta} \), for the word \( \delta \).  Let \( s' \) be a symbol for a new node.  Set
  \( T_{\zeta} = T_{\delta} \cup \{s'\}\).  If \( m(I) = 1 \) or \( M(I) = n \), set
  \( (s', \emptyset) \in E(T_{\delta}) \).  Otherwise let \( s \in T_{\delta} \) be the minimal node such that
  \( m(J_{s}) < m(I) \le M(J_{s}) \) (\( s \) may still be the root \( \emptyset \)) and set
  \( (s',s) \in E(T_{\delta}) \).  Finally, define for \( s \in T_{\delta} \)
  \begin{displaymath}
    I_{s} =
    \begin{cases}
      J_{s} & \textrm{if \( M(J_{s}) < m(I) \)}; \\
      \seg{m(J_{s})}{M(J_{s}) + m} & \textrm{if \( m(J_{s}) < m(I) \le M(J_{s}) \)}; \\
      \seg{M(J_{s})+m}{M(J_{s})+m} & \textrm{if \( m(I) \le m(J_{s}) \)}.
    \end{cases}
  \end{displaymath}
  and set \( I_{s'} = I \).

  It is now straightforward to check that \( T_{\zeta} \) is a structure tree for \( \zeta \).
\end{proof}

\subsection{From hereditary to rigid words}
\label{sec:from-hered-rigid}

\textit{From now on \( A \) will denote a closed subgroup of \( G \) of diameter \( \diam{A} \le 1 \),} unless stated
otherwise.

\begin{lemma}
  \label{sec:hnn-extensions-a-cancellation-error}
  If \( (G,d) \) is a tsi group, then for all \( g_{1}, \ldots, g_{n-1} \in G \), for all
  \( a_{1}, \ldots, a_{n} \in A \) such that \( d(a_{i}, e) \le 1 \)
  \[ d(g_{1} \cdots g_{n-1}, a_{1}g_{1}a_{2} \cdots a_{n-1}g_{n-1}a_{n}) \le n \]
\end{lemma}

\begin{proof}
  By induction.  For \( n=2 \) we have
  \[ d(g_{1},a_{1}g_{1}a_{2}) \le d(g_{1},a_{1}g_{1}) + d(a_{1}g_{1},a_{1}g_{1}a_{2}) = d(e,a_{1}) + d(e,a_{2}) \le
  2.  \] For the step of induction
  \begin{displaymath}
    \begin{aligned}
      & d(g_{1} \cdots g_{n-1}, a_{1}g_{1}a_{2} \cdots a_{n-1}g_{n-1}a_{n}) \le\\
      & d(g_{1} \cdots g_{n-1}, g_{1} \cdots g_{n-1}a_{n}) + d(g_{1} \cdots g_{n-1}a_{n}, a_{1}g_{1}a_{2} \cdots
      a_{n-1}g_{n-1}a_{n}) = \\
      & d(e,a_{n}) + d(g_{1} \cdots g_{n-2}, a_{1}g_{1}a_{2} \cdots g_{n-2}a_{n-1}) \le 1 + (n-1)= n.
    \end{aligned}
  \end{displaymath}
  And the lemma follows.
\end{proof}

Let \( \beta \) be a word of the form
\[ \beta = g_{0} \concat t \concat a_{1} \concat t^{-1} \concat g_{1} \concat t \concat a_{2} \concat t^{-1} \concat
\cdots \concat g_{n-1} \concat t \concat a_{n} \concat t^{-1} \concat g_{n}, \]
where \( g_{i} \in G \) and \( a_{i} \in A \).

Define a word \( \delta \) by setting for \( i \in \seg{1}{|\beta|} \)
\begin{displaymath}
  \delta(i) =
  \begin{cases}
    e & \textrm{if \( i = 1 \mod 4 \)}; \\
    t & \textrm{if \( i = 2 \mod 4 \)}; \\
    e & \textrm{if \( i = 3 \mod 4 \)}; \\
    t^{-1} & \textrm{if \(i= 0 \mod 4 \)}.
  \end{cases}
\end{displaymath}
Or, equivalently,
\[ \delta = e \concat t\concat e \concat t^{-1} e \concat \cdots \concat e \concat t \concat e \concat t^{-1} \concat
e.  \]
If \( T_{\delta} = \{\emptyset, s_{1}, \ldots, s_{n}, s_{1}', \ldots, s_{n}'\} \) with \( s_{k} \prec \emptyset \),
\( s_{k}' \prec s_{k} \), \( I_{s_{k}} = \seg{4k-2}{4k} \), \( I_{{s_{k}'}} = \{4k-1\} \), then \( T_{\delta} \) is a
slim evaluation tree.  Set 
\[ \xi = \symmet{\beta}{\delta}{1}{\{4k+1\}_{k=0}^{n}} = \symmet{\beta}{\delta}{1}{R_{\emptyset}}. \]
\begin{lemma}
  \label{sec:hnn-extensions-even-subword-reduction-type-g}
  Let \( \beta, \xi \) be as above.  If \( \zeta \) is a trivial word of length \( |\beta| \), \( \zeta \) and
  \( \beta \) are \multipliable and \( \zeta(i) \in G \) for all \( i \), in other words if
  \[ \zeta = h_{0} \concat e \concat h_{1} \concat e \concat h_{2} \concat e \concat h_{3} \concat e \concat \cdots
  \concat h_{2n-2} \concat e \concat h_{2n-1} \concat e \concat h_{2n}, \]
  where \( h_{i} \in G \), then \( \rho(\beta,\xi) \le \rho(\beta,\zeta) \).
\end{lemma}

\begin{proof}
  By the two-sided invariance
  \[ \rho(\beta,\zeta) \ge d(g_{0}a_{1}g_{1}a_{2}\cdots g_{n-1}a_{n}g_{n},e) + 2n.  \] On the other hand
  \begin{displaymath}
    \begin{aligned}
      \rho(\beta,\xi) =& \sum_{i=1}^{n} d(a_{i},e) + d(g_{0}g_{1} \cdots g_{n},e) \le\\
      & n + d(g_{0}g_{1} \cdots g_{n}, e) \le \\
      & n + d(g_{0}g_{1} \cdots g_{n}, g_{0}a_{1}g_{1}a_{2} \cdots g_{n-1}a_{n}g_{n}) + d(g_{0}a_{1}g_{1}a_{2} \cdots
      g_{n-1}a_{n}g_{n},e) = \\
      & n + d(g_{1} \cdots g_{n-1}, a_{1}g_{1} \cdots g_{n-1}a_{n}) + d(g_{0}a_{1}g_{1}a_{2} \cdots g_{n-1}a_{n}g_{n},e)
      \le \\
      & \textrm{[by Lemma \ref{sec:hnn-extensions-a-cancellation-error}] } 2n + d(g_{0}a_{1}g_{1}a_{2} \cdots
      g_{n-1}a_{n}g_{n},e).
    \end{aligned}
  \end{displaymath}
  Hence \( \rho(\beta,\xi) \le \rho(\beta,\zeta). \)
\end{proof}

Suppose we have words
\begin{eqnarray*}
  \begin{aligned}
    \nu_{k} = g_{(k,1)} \concat \cdots \concat g_{(k,q_{k})}, \quad 
    \textrm{where \( g_{(k,j)} \in G \) and \(k \in \seg{0}{n}, \)}\\
    \mu_{k} = a_{(k,1)} \concat \cdots \concat a_{(k,p_{k})}, \quad \textrm{where \( a_{(k,j)} \in A \) and
      \(k \in \seg{1}{n}. \)}
  \end{aligned}
\end{eqnarray*}

And let \( \bar{\beta} \) be the word
\[ \bar{\beta} = \nu_{0} \concat t \concat \mu_{1} \concat t^{-1} \nu_{1} \concat \cdots \concat \nu_{n-1} \concat t
\concat \mu_{n} \concat t^{-1} \concat \nu_{n}.  \]
Let \( \{i_{k}\}_{k=1}^{n} \), \( \{i'_{k}\}_{k=1}^{n} \) be indices such that
\begin{enumerate}[(i)]
\item \( i_{k} < i_{k+1} \), \( i'_{k} < i'_{k+1} \);
\item \( \beta(i_{k}) = t \), \( \beta(i'_{k}) = t^{-1} \);
\item if \( \beta(i) = t \), then \( i = i_{k} \) for some \( k \in \seg{1}{n} \); if \( \beta(i) = t^{-1} \), then
  \( i = i'_{k} \) for some \( k \in \seg{1}{n} \).
\end{enumerate}
In other words
\[ i_{k} = \sum_{l=0}^{k-1} q_{l} + \sum_{l=1}^{k-1} p_{k} + 2(k-1) + 1, \quad i'_{k} = i_{k} + p_{k} + 1.\]
Define the word \( \delta \) of length \( |\bar{\beta}| \) by
\begin{displaymath}
  \delta(i) =
  \begin{cases}
    e & \textrm{if \( \bar{\beta}(i) \in G \)};\\
    \bar{\beta}(i) & \textrm{if \( \bar{\beta}(i) = t^{\pm 1} \)}.\\
  \end{cases}
\end{displaymath}
If \( T_{\delta} = \{\emptyset, s_{1}, \ldots, s_{n}, s_{1}', \ldots, s_{n}'\} \), \( s_{k} \prec \emptyset \),
\( s_{k}' \prec s_{k} \) and \( I_{s_{k}} = \seg{i_{k}}{i_{k}'} \), \( I_{s_{k}'} = \seg{i_{k}+1}{i_{k}'-1} \) 
(in other words \( I_{s_{k}} \) and \( I_{s_{k}'} \) are such that \( \bar{\beta}[I_{s_{i}}] = t \concat \mu_{i} \concat t^{-1} \), \( \bar{\beta}[I_{s_{i}}] = \mu_{i} \)), then
\( T_{\delta} \) is a slim evaluation tree.  
Let \( \{j_{k}\}_{k=1}^{m} \) be the enumeration of the set
\[ \seg{1}{|\bar{\beta}|} \setminus \bigcup_{k=1}^{n} \seg{i_{k}}{i'_{k}}.  \] 
Set inductively
\begin{eqnarray*}
  \begin{aligned}
    \xi_{0} &= \symmet{\bar{\beta}}{\delta}{j_{1}}{\{j_{k}\}} = \symmet{\bar{\beta}}{\delta}{j_{1}}{R_{\emptyset}},  \\
    \xi_{l}&= \symmet{\bar{\beta}}{\xi_{l-1}}{j^{(l)}_{1}}{\{j^{(l)}_{k}\}} =
    \symmet{\bar{\beta}}{\xi_{l-1}}{j^{(l)}_{1}}{R_{s_{l}'}},
  \end{aligned}
\end{eqnarray*}
where \( j^{(l)}_{k} = i_{l} + k \), \( l \in \seg{1}{n} \), \( k \in \seg{1}{p_{k}} \).  Finally set
\( \bar{\xi} = \xi_{n} \).

\begin{example-nn}
  \label{sec:from-hered-rigid-beta}
  For example, if
  \[ \bar{\beta} = g_{1} \concat g_{2} \concat t \concat a_{1} \concat a_{2} \concat a_{3} \concat t^{-1} \concat
  g_{3}, \] then
  \begin{displaymath}
    \begin{aligned}
      \delta &= e \concat e \concat t \concat e \concat e \concat e \concat t^{-1} \concat e,\\
      \xi_{0} &= x \concat g_{2}
      \concat t \concat e \concat e \concat e \concat t^{-1} \concat g_{3}, \quad x = g_{3}^{-1}g_{2}^{-1}, \\
      \xi_{1} &= x \concat g_{2} \concat t \concat y \concat a_{2} \concat a_{3} \concat t^{-1} \concat g_{3}, \quad y =
      a_{3}^{-1}a_{2}^{-1}.
    \end{aligned}
  \end{displaymath}
\end{example-nn}

\begin{lemma}
  \label{sec:hnn-extensions-even-subword-reduction-type-g-general}
  Let \( \bar{\beta}, \bar{\xi} \) be as above.  If \( \zeta \) is a trivial word of length \( |\bar{\beta}| \), \(
  \zeta \) and \( \bar{\beta} \) are \multipliable and \( \zeta(i) \in G \) for all \( i \), then
  \( \rho(\bar{\beta},\bar{\xi}) \le \rho(\bar{\beta},\zeta) \).
\end{lemma}

\begin{proof}
  Set
  \begin{displaymath}
    \begin{aligned}
      \beta &= \hat{\nu}_{0} \concat t \concat \hat{\mu}_{1} \concat t^{-1} \concat \ldots \concat \hat{\mu}_{n} \concat
      t^{-1} \concat \hat{\nu}_{n},\\
      \xi' &= \hat{\bar{\xi}}[1,i_{1} - 1] \concat t \concat \hat{\bar{\xi}}[i_{1} + 1, i_{1}' - 1] \concat t ^{-1}
      \concat \ldots \concat \hat{\bar{\xi}}[i_{n} + 1, i_{n}' - 1] \concat t^{-1} \concat \hat{\bar{\xi}}[i_{n}' +1,
      n],\\
      \zeta' &= \hat{\zeta}[1,i_{1} - 1] \concat t \concat \hat{\zeta}[i_{1} + 1, i_{1}' - 1] \concat t ^{-1} \concat
      \ldots \concat \hat{\zeta}[i_{n} + 1, i_{n}' - 1] \concat t^{-1} \concat \hat{\zeta}[i_{n}' +1, n].
    \end{aligned}
  \end{displaymath}
  If \( \xi \) is as in Lemma \ref{sec:hnn-extensions-even-subword-reduction-type-g}, then \( \xi' = \xi \) and
  \begin{displaymath}
    \rho(\bar{\beta}, \zeta) \ge \textrm{[by tsi]}\ \rho(\beta,\zeta') \ge \textrm{[by Lemma
      \ref{sec:hnn-extensions-even-subword-reduction-type-g}]}\ \rho(\beta,\xi) = \rho(\beta,\xi') =
    \rho(\bar{\beta},\bar{\xi}). \qedhere
  \end{displaymath}
\end{proof}

\begin{lemma}
  \label{sec:from-hered-subword-keep-alternating-beta}
  Let \( \bar{\beta} \) be a word of the form 
  \[ \bar{\beta} = \nu_{0} \concat t \concat \mu_{1} \concat t^{-1} \nu_{1} \concat \cdots \concat \nu_{n-1} \concat t
  \concat \mu_{n} \concat t^{-1} \concat \nu_{n},\]
  for some words \( \mu_{i} \in \word{A} \), \( \nu_{i} \in \word{G} \).  If
  \( j_{0}, j_{1} \in \seg{1}{|\bar{\beta}|} \) are such that \( j_{0} < j_{1} \),
  \( \bar{\beta}(j_{0}), \bar{\beta}(j_{1}) \in \{t, t^{-1}\} \) and
  \( \bar{\beta}(j_{0}) = \bar{\beta}(j_{1})^{-1} \), then
  \( \bar{\beta}\Big[\seg{1}{|\bar{\beta}|} \setminus \seg{j_{0}}{j_{1}}\Big] \) can be written as
  \[ \bar{\beta}\Big[\seg{1}{|\bar{\beta}|} \setminus \seg{j_{0}}{j_{1}}\Big] =
  \nu'_{0} \concat t \concat \mu'_{1} \concat t^{-1} \nu'_{1} \concat \cdots \concat \nu'_{m-1} \concat t
  \concat \mu'_{m} \concat t^{-1} \concat \nu'_{m},\]
  for \( \mu'_{i} \in \word{A} \), \( \nu'_{i} \in \word{G} \) and \( m \le n \). 
\end{lemma}

\begin{proof}
  Suppose for definiteness that \( \bar{\beta}(j_{0}) = t \) (the case \( \bar{\beta}(j_{0}) = t^{-1} \) is similar).
  For some \( k,l \) we can write
  \( \bar{\beta} = \bar{\beta}_{0} \concat \nu_{k} \concat t \concat \bar{\beta}_{1} \concat t^{-1} \concat \nu_{l}
  \concat \bar{\beta}_{2} \), where \( |\bar{\beta}_{0}| + |\nu_{k}| = j_{0} -1 \),
  \( |\bar{\beta}_{2}| + |\nu_{l}| = |\bar{\beta}|-j_{1} \) and \( \bar{\beta}_{0} \) is either empty or ends with
  \( t^{-1} \), \( \bar{\beta}_{2} \) is either empty or starts with \( t \).  Then
  \[ \bar{\beta}\Big[\seg{1}{|\bar{\beta}|} \setminus \seg{j_{0}}{j_{1}}\Big] = \bar{\beta}_{0} \concat \nu_{k} \concat
  \nu_{l} \concat \bar{\beta}_{2}. \qedhere\]
\end{proof}

\medskip

Let \( \gamma \) be a word of the form
\[ \gamma = a_{0} \concat t^{-1} \concat g_{0} \concat t \concat a_{1} \concat t^{-1} \concat g_{1} \concat t \concat
\cdots \concat a_{n-1} \concat t^{-1} \concat g_{n-1} \concat t \concat a_{n}, \]
where \( g_{i} \in G \) and \( a_{i} \in A \).  Let \( \zeta \) be a trivial word such that
\( \zeta \) and \( \gamma \) are \multipliable and \( \zeta(i) \in G \) for all \( i \).  In other words
\[ \zeta = h_{0} \concat e \concat h_{1} \concat e \concat h_{2} \concat e \concat h_{3} \concat e \concat \cdots
\concat h_{2n-2} \concat e \concat h_{2n-1} \concat e \concat h_{2n}, \]
where \( h_{i} \in G \).  Define a word \( \delta \) by
\begin{displaymath}
  \delta(i) =
  \begin{cases}
    a_{0} & \textrm{if \( i = 1 \)};\\
    e & \textrm{if \( i = 1 \mod 4 \) and \( 1 < i < 4n+1 \)}; \\
    t^{-1} & \textrm{if \( i = 2 \mod 4 \)}; \\
    e & \textrm{if \( i = 3 \mod 4 \)}; \\
    t & \textrm{if \(i= 0 \mod 4 \)};\\
    a_{0}^{-1} & \textrm{if \( i = 4n+1 \)}.
  \end{cases}
\end{displaymath}
Or, equivalently,
\[ \delta = a_{0} \concat t^{-1}\concat e \concat t \concat e \concat \cdots \concat e \concat t^{-1} \concat e \concat
t \concat a_{0}^{-1}.  \]
If \( T_{\delta} = \{\emptyset, u,s, s_{1}, \ldots, s_{n-1}, s_{1}', \ldots, s_{n-1}'\} \) with \( u \prec \emptyset \),
\( s \prec u \), \( s_{k} \prec s \), \( s_{k}' \prec s_{k} \), \( I_{u} = \seg{2}{n-1} \), \( I_{s} = \seg{3}{n-2} \),
\( I_{s_{k}} = \seg{4k}{4k+2} \), \( I_{{s_{i}'}} = \{4k+1\} \), then \( T_{\delta} \) is a slim evaluation tree.  Set
\[ \xi = \symmet{\gamma}{\delta}{3}{\{4k-1\}_{k=1}^{n}} = \symmet{\gamma}{\delta}{3}{R_{s}}. \]
\begin{example-nn}
  \label{sec:from-hered-rigid-gamma}
  For example, if
  \[ \gamma = a_{0} \concat t^{-1} \concat g_{0} \concat t \concat a_{1} \concat t^{-1} \concat g_{1} \concat t \concat
  a_{2}, \] then
  \begin{displaymath}
    \begin{aligned}
      &\delta = a_{0} \concat t^{-1} \concat e \concat t \concat e \concat t^{-1} \concat e \concat t \concat
      a_{0}^{-1},\\
      &\xi = a_{0} \concat t^{-1} \concat g_{1}^{-1} \concat t \concat e \concat t^{-1} \concat g_{1} \concat t \concat
      a_{0}^{-1}.
    \end{aligned}
  \end{displaymath}
\end{example-nn}

\begin{lemma}
  \label{sec:hnn-extensions-even-subsword-reduction-type-a}
  If \( \gamma, \zeta,\xi \) are as above, then \( \rho(\gamma,\xi) \le \rho(\gamma,\zeta) \).
\end{lemma}

\begin{proof}
  By the two-sided invariance
  \[ \rho(\gamma,\zeta) \ge d(a_{0}g_{0}a_{1}g_{1}\cdots a_{n-1}g_{n-1}a_{n},e) + 2n.  \] On the other hand
  \begin{displaymath}
    \begin{aligned}
      \rho(\gamma,\xi) =& d(a_{0}a_{n},e) + \sum_{i=1}^{n-1}d(a_{i},e) + d(g_{0}g_{1} \cdots g_{n}, e) \le \\
      &n + d(g_{0}g_{1}\ldots g_{n-1}, a_{0}^{-1}a_{n}^{-1}) + d(a_{0}^{-1}a_{n}^{-1},e) \le \\
      &n+1 + d(g_{0}g_{1} \cdots g_{n-1}, a_{0}^{-1}a_{n}^{-1}) \le \\
      &n+1 + d(a_{0}g_{0}g_{1} \cdots g_{n-2}g_{n-1}a_{n}, a_{0}g_{0}a_{1}g_{1} \cdots a_{n-1}g_{n-1}a_{n}) +\\
      & \qquad d(a_{0}g_{0}a_{1}g_{1}\cdots a_{n-1}g_{n-1}a_{n},e) = \\
      &n+1 + d(g_{1} \cdots g_{n-2}, a_{1}g_{1} \cdots g_{n-2}a_{n-1}) +\\
      & \qquad d(a_{0}g_{0}a_{1}g_{1}\cdots a_{n-1}g_{n-1}a_{n},e) \le \textrm{[by Lemma
        \ref{sec:hnn-extensions-a-cancellation-error}] }\\
      &n+1 + n-1 + d(a_{0}g_{0}a_{1}g_{1}\cdots a_{n-1}g_{n-1}a_{n},e) \le \rho(\gamma,\zeta).
    \end{aligned}
  \end{displaymath}
  And the lemma follows.
\end{proof}

Suppose we have words
\begin{eqnarray*}
  \begin{aligned}
    \mu_{k} = a_{(k,1)} \concat \cdots \concat a_{(k,p_{k})}, \quad \textrm{where \( a_{(k,j)} \in A \) and \(k \in
      \seg{0}{n}, \)} \\
    \nu_{k} = g_{(k,1)} \concat \cdots \concat g_{(k,q_{k})}, \quad \textrm{where \( g_{(k,j)} \in G \) and
      \(k \in \seg{1}{n}, \)}
  \end{aligned}
\end{eqnarray*}
and let \( \bar{\gamma} \) be the word
\[ \bar{\gamma} = \mu_{0} \concat t^{-1} \concat \nu_{0} \concat t \concat \mu_{1} \concat \cdots \concat \mu_{n-1}
\concat t^{-1} \concat \nu_{n-1} \concat t \concat \mu_{n}.  \]
Let \( \{i_{k}\}_{k=1}^{n} \), \( \{i'_{k}\}_{k=1}^{n} \) be indices such that
\begin{enumerate}[(i)]
\item \( i_{k} < i_{k+1} \), \( i'_{k} < i'_{k+1} \);
\item \( \gamma(i_{k}) = t^{-1} \), \( \gamma(i'_{k}) = t \);
\item if \( \gamma(i) = t^{-1} \), then \( i = i_{k} \) for some \( k \in \seg{1}{n} \); if \( \gamma(i) = t\), then
  \( i = i'_{k} \) for some \( k \in \seg{1}{n} \).
\end{enumerate}
Define the word \( \delta \) of length \( |\bar{\gamma}| \) by
\begin{displaymath}
  \delta(i) =
  \begin{cases}
    e & \textrm{if \( \bar{\gamma}(i) \in G \)};\\
    \bar{\gamma}(i) & \textrm{if \( \bar{\gamma}(i) = t^{\pm 1} \)}.\\
  \end{cases}
\end{displaymath}
If \( T_{\delta} = \{\emptyset, u, s, s_{1}, \ldots, s_{n-1}, s_{1}', \ldots, s_{n-1}'\} \), \( u \prec \emptyset \),
\( s \prec u \), \( s_{k} \prec s \), \( s_{k}' \prec s_{k} \) and \( I_{u} = \seg{i_{1}}{i_{n}'} \),
\( I_{s} = \seg{i_{1}+1}{i_{n}'-1} \), \( I_{s_{k}} = \seg{i_{k}'}{i_{k+1}} \),
\( I_{s_{k}'} = \seg{i_{k}'+1}{i_{k+1}'-1} \) (in other words \( I_{s_{k}} \) and \( I_{s_{k}'} \) are such that
\( \bar{\gamma}[I_{s_{i}}] = t \concat \mu_{i} \concat t^{-1} \), \( \bar{\gamma}[I_{s_{i}}] = \mu_{i} \)), then
\( T_{\delta} \) is a slim evaluation tree.

Let \( \{j_{k}\}_{k=1}^{m} \) be the enumeration of the set
\[ \bigcup_{k=1}^{n} \seg{i_{k} + 1}{i'_{k} - 1}.  \] Set inductively
\begin{displaymath}
  \begin{aligned}
    \xi_{0} &= \symmet{\bar{\gamma}}{\delta}{j_{1}}{\{j_{k}\}} =\symmet{\bar{\gamma}}{\delta}{j_{1}}{R_{s}},\\
    \xi_{l} &= \symmet{\bar{\gamma}}{\xi_{l-1}}{j^{(l)}_{1}}{\{j^{(l)}_{k}\}} =
    \symmet{\bar{\gamma}}{\xi_{l-1}}{j^{(l)}_{1}}{R_{s_{l}'}},
  \end{aligned}
\end{displaymath}
where \( j^{(l)}_{k} = i'_{l} + k \) and \( l \in \seg{1}{n-1} \), \( k \in \seg{1}{p_{k}} \).  Finally set
\[ \bar{\xi} = \symmet{\bar{\gamma}}{\xi_{n}}{1}{\seg{1}{i_{1}-1} \cup \seg{i'_{n}+1}{n}} = \symmet{\bar{\gamma}}{\xi_{n}}{1}{R_{\emptyset}}.\]
\begin{example-nn}
  \label{sec:from-hered-rigid-gamma-bar}
  For example, if
  \[ \bar{\gamma} = a_{1} \concat a_{2} \concat t^{-1} \concat g_{1} \concat t \concat a_{3} \concat a_{4} \concat
  t^{-1} \concat g_{2} \concat g_{3} \concat t \concat a_{5}, \] then
  \begin{displaymath}
    \begin{aligned}
      \delta &= e \concat e \concat t^{-1} \concat e \concat t \concat e \concat e \concat t^{-1} \concat e \concat e
      \concat t \concat e,\\
      \xi_{0} &= e \concat e \concat t^{-1}\concat x \concat t \concat e \concat e \concat t^{-1} \concat g_{2} \concat g_{3}
      \concat t \concat e, \quad x= g_{3}^{-1}g_{2}^{-1}\\
      \xi_{1} &= e \concat e \concat t^{-1} \concat x \concat t \concat a_{4}^{-1} \concat a_{4} \concat t^{-1} \concat
      g_{2} \concat g_{3} \concat t \concat e,\\
      \bar{\xi} &= y \concat a_{2} \concat t^{-1} \concat x \concat t \concat a_{4}^{-1} \concat a_{4} \concat t^{-1} \concat
      g_{2} \concat g_{3} \concat t \concat a_{5}, \quad y = a_{5}^{-1}a_{2}^{-1}.
    \end{aligned}
  \end{displaymath}
\end{example-nn}

\begin{lemma}
  \label{sec:hnn-extensions-even-subword-reduction-type-a-general}
  Let \( \bar{\gamma}, \bar{\xi} \) be as above.  If \( \zeta \) is a trivial word of length \( |\bar{\gamma}| \), \(
  \zeta \) and \( \bar{\gamma} \) are \multipliable and \( \zeta(i) \in G \) for all \( i \), then
  \( \rho(\bar{\gamma},\bar{\xi}) \le \rho(\bar{\gamma},\zeta) \).
\end{lemma}

\begin{proof}
  Proof is similar to the proof of Lemma \ref{sec:hnn-extensions-even-subword-reduction-type-g-general} using Lemma
  \ref{sec:hnn-extensions-even-subsword-reduction-type-a} instead of Lemma
  \ref{sec:hnn-extensions-even-subword-reduction-type-g}.
\end{proof}

\begin{lemma}
  \label{sec:from-hered-subword-keep-alternating}
  Let \( \bar{\gamma} \) be a word of the form 
  \[ \bar{\gamma} = \mu_{0} \concat t^{-1} \concat \nu_{0} \concat t \concat \mu_{1} \concat \cdots \concat \mu_{n-1}
  \concat t^{-1} \concat \nu_{n-1} \concat t \concat \mu_{n}, \]
  for some words \( \mu_{i} \in \word{A} \), \( \nu_{i} \in \word{G} \).  If
  \( j_{0}, j_{1} \in \seg{1}{|\bar{\gamma}|} \) are such that \( j_{0} < j_{1} \),
  \( \bar{\gamma}(j_{0}), \bar{\gamma}(j_{1}) \in \{t, t^{-1}\} \) and
  \( \bar{\gamma}(j_{0}) = \bar{\gamma}(j_{1})^{-1} \), then
  \( \bar{\gamma}\Big[\seg{1}{|\bar{\gamma}|} \setminus \seg{j_{0}}{j_{1}}\Big] \) can be written as
  \[ \bar{\gamma}\Big[\seg{1}{|\bar{\gamma}|} \setminus \seg{j_{0}}{j_{1}}\Big] =
 \mu'_{0} \concat t^{-1} \concat \nu'_{0} \concat t \concat \mu'_{1} \concat \cdots \concat \mu'_{m-1}
  \concat t^{-1} \concat \nu'_{m-1} \concat t \concat \mu'_{m}, \]
  for \( \mu_{i} \in \word{A} \), \( \nu_{i} \in \word{G} \) and \( m \le n \). 
\end{lemma}

\begin{proof}
  The proof is similar to the proof of Lemma \ref{sec:from-hered-subword-keep-alternating-beta}.
\end{proof}

\begin{definition}
  \label{sec:hnn-extensions-rigid-pair}
  Let \( (\alpha,\zeta) \) be a hereditary \( f \)-pair of length \( n \).  It is called \emph{rigid} if for all
  \( i \in \seg{1}{n} \)
  \[ \alpha(i) = t^{\pm 1} \implies \zeta(i) = \alpha(i).  \]
\end{definition}
Here is an example of a rigid pair:
\begin{displaymath}
  \begin{aligned}
    &\alpha = g_{0} \concat t \concat a_{1} \concat t^{-1} \concat g_{1} \concat t \concat a_{2} \concat t^{-1} \concat
    g_{2},\\
    &\zeta = g_{2}^{-1}g_{1}^{-1} \concat t \concat e \concat t^{-1} \concat g_{1} \concat t \concat e \concat t^{-1}
    \concat g_{2}.
  \end{aligned}
\end{displaymath}

\begin{lemma}
  \label{sec:induced-metrics-existence-of-a-good-pair}
  Let \( f \in G*tAt^{-1} \), and let \( \alpha \in \word{G \cup \langle t \rangle} \) be the reduced form of \( f \).
  If \( (\alpha,\zeta) \) is a hereditary \( f \)-pair, then there exists a rigid \( f \)-pair \( (\alpha,\xi) \) such
  that \( \rho(\alpha,\zeta) \ge \rho(\alpha,\xi) \).  Moreover, if for some \( i \) one has \( \alpha(i) = t \), then
  \( \xi(i+1) \in A \).
\end{lemma}

\begin{proof}
  Let \( (\alpha,\zeta) \) be hereditary and let \( T_{\zeta} \) be a structure tree for \( \zeta \).  Let
  \( s \in T_{\zeta} \) and set \( Q_{s} = R_{s} \setminus\{m(I_{s}),M(I_{s})\} \).  Let \( s_{1}, \ldots, s_{N} \in
  T_{\zeta} \) be such that \( R_{s} = I_{s} \setminus \bigcup_{i=1}^{N}I_{s_{i}} \).  Then using for each \( s_{i} \) Lemma
  \ref{sec:from-hered-subword-keep-alternating-beta} or Lemma \ref{sec:from-hered-subword-keep-alternating}
  (depending on whether
  \( \zeta(m(I_{s})) = t^{-1} \) or \( \zeta(m(I_{s})) = t \)), we get
  \begin{multline*}
    \alpha[Q_{s}] = \bar{\beta} = g_{(0,1)}\concat \cdots \concat g_{(0,q_{0})} \concat t \concat a_{(1,1)} \concat
    \cdots
    \concat a_{(1,p_{1})} \concat t^{-1} \concat \cdots \\
    \cdots \concat t \concat a_{(n,1)} \concat \cdots \concat a_{(n,p_{n})} \concat t^{-1} \concat g_{(n,1)} \cdots
    g_{(n,q_{n})},
  \end{multline*}
  or
  \begin{multline*}
    \alpha[Q_{s}] = \bar{\gamma} = a_{(0,1)}\concat \cdots \concat a_{(0,p_{0})} \concat t^{-1} \concat g_{(0,1)}
    \concat \cdots
    \concat g_{(0,q_{1})} \concat t \concat \cdots \\
    \cdots \concat t^{-1} \concat g_{(n-1,1)} \concat \cdots \concat g_{(n-1,q_{n})} \concat t \concat a_{(n,1)} \cdots
    a_{(n,p_{n})},
  \end{multline*}
  where \( a_{(i,j)} \in A \) and \( g_{(i,j)} \in G \).
  
  Let \( \bar{\xi}_{s} \) be as in Lemma \ref{sec:hnn-extensions-even-subword-reduction-type-g-general} or in Lemma
  \ref{sec:hnn-extensions-even-subword-reduction-type-a-general} depending on whether \( \alpha[Q_{s}] = \bar{\beta} \)
  or \( \alpha[Q_{s}] = \bar{\gamma} \) and set
  \[ \xi[Q_{s}] := \bar{\xi}_{s}, \quad \xi(m(I_{s})) = \alpha(m(I_{s})), \quad \xi(M(I_{s})) = \alpha(M(I_{s})) \quad
  \textrm{if \( s \ne \emptyset \)}, \]
  \[ \xi[R_{\emptyset}] := \bar{\xi}_{\emptyset}, \quad \textrm{if \( s = \emptyset \)}. \]
  Do this for all \( s \in T_{\zeta} \).  Then \( (\alpha,\xi) \) is a rigid \( f \)-pair and
  \[ \rho(\alpha,\zeta) \ge \rho(\alpha,\xi)\ \textrm{[by Lemma
    \ref{sec:hnn-extensions-even-subword-reduction-type-g-general} and Lemma
    \ref{sec:hnn-extensions-even-subword-reduction-type-a-general}]}.  \]
  The moreover part follows immediately from the construction of \( \xi \).
\end{proof}

\begin{theorem}
  \label{sec:induced-metrics-equality-of-induced-metrics}
  Let \( (G,d) \) be a tsi group, \( A < G \) be a closed subgroup, \underline{\emph{not}} necessarily of diameter at
  most one.  If \( d \) and \( \dist \) are as before (see the beginning of Section \ref{sec:hnn-extensions}), then
  \( d = \dist \) if and only if \( \diam{A} \le 1 \).
\end{theorem}

\begin{proof}
  First we show that the condition \( \diam{A} \le 1 \) is necessary.  Suppose \( \diam{A} > 1 \) and let \( a \in A \)
  be such that \( d(a,e) > 1 \).  Then
  \begin{displaymath}
    \begin{aligned}
      &\dist(ata^{-1}t^{-1},e) = d(a,e) + d(ta^{-1}t^{-1},e) = d(a,e) + d(a^{-1},e) = 2d(a,e) > 2,\\
      &d(ata^{-1}t^{-1},e) = d(ata^{-1}t^{-1},aea^{-1}e) \le\\
      & \qquad d(a,a) + d(t,e) + d(a^{-1},a^{-1}) + d(t^{-1},e) = 2.
    \end{aligned}
  \end{displaymath}
  And so \( \dist \ne d \).
 
  Suppose now \( \diam{A} \le 1 \).  Let \( f \in G*tAt^{-1} \) be given and let \( \alpha \) be the reduced form of
  \( f \).  If \( (\alpha,\zeta) \) is a \multipliable \( f \)-pair, then by Lemma \ref{sec:reduction-to-hereditary-pair}
  and Lemma \ref{sec:induced-metrics-existence-of-a-good-pair} there is a rigid \( f \)-pair \( (\alpha,\xi) \) such
  that \( \rho(\alpha,\xi) \le \rho(\alpha,\zeta) \) and \( \alpha(i)=t \) implies \( \xi(i+1) \in A \).  Hence we can
  view \( \xi \) as an element in \( \word{G \cup tAt^{-1}} \).  Since \( \zeta \) was arbitrary, it follows that
  \( \dist(f,e) \le d(f,e) \).  The inverse inequality \( d(f,e) \le \dist(f,e) \) follows from item
  \eqref{prop:graev-metric-amalgam-properties-item:maximality} of Proposition \ref{sec:graev-metric-amalgam-properties}.
  Thus \( \dist(f,e) = d(f,e) \), and, by the left invariance, \( \dist(f_{1},f_{2}) = d(f_{1},f_{2}) \) for all
  \( f_{1},f_{2} \in G*tAt^{-1} \).
\end{proof}

\begin{proposition}
  \label{sec:metr-hnn-extens-free-product-closed}
  Let \( (G,d) \) be a tsi group, \( A < G \) be a subgroup and \( d \) be the Graev metric on the free product
  \( G * \langle t \rangle \).  We can naturally view \( G * tAt^{-1} \) as a subgroup of \( G * \langle t \rangle\).
  If \( A \) is closed in \( G \), then \( G * tAt^{-1} \) is closed in \( G * \langle t \rangle \).
\end{proposition}

\begin{proof}
  The proof is similar in spirit to the proof of item \eqref{prop:graev-metric-amalgam-properties-item:induced-metric}
  of Proposition \ref{sec:graev-metric-amalgam-properties}, but requires some additional work.  Suppose the statement is
  false and there is \( f \in G * \langle t \rangle \) such that \( f \not \in G* tAt^{-1} \), but
  \( f \in \overline{G * tAt^{-1}} \).  Let \( \alpha \in \word{G \cup \langle t \rangle} \) be the reduced form of
  \( f \), \( n = |\alpha| \).  We show that this is impossible and \( f \in G * tAt^{-1} \).  The proof goes by
  induction on \( n \).

  \textbf{Base of induction}.  For the base of induction we consider cases \( n \in \{1,2\} \).  If \( n = 1 \), then
  either \( f \in G \) or \( f = t^{k} \) for some \( k \ne 0 \).  Since \( G < G * tAt^{-1} \), it must be the case
  that \( f = t^{k}\).  Let \( h \in G * tAt^{-1} \) be such that \( d(f,h) < 1 \), where \( d \) is the Graev
  metric on \( G * \langle t \rangle \).  Let \( \phi_{1} : G \to \mathbb{Z} \) be the trivial homomorphism:
  \( \phi_{1}(g) = 0 \) for all \( g \in G \); and let \( \phi_{2} : \langle t \rangle \to \mathbb{Z} \) be the natural
  isomorphism: \( \phi_{2}(t^{k}) = k \).  By item \eqref{prop:graev-metric-amalgam-properties-item:extending-lipschitz}
  of Proposition \ref{sec:graev-metric-amalgam-properties} \( \phi_{1} \) and \( \phi_{2} \) extend to a \( 1
  \)-\lipschitz homomorphism \( \phi : G* \langle t \rangle \to \mathbb{Z} \).  But \( d_{\mathbb{Z}}(\phi(f),\phi(h)) = |k| \ge 1
  \).  We get a contradiction with the assumption \( d(f,h) < 1 \).

  Note that for any \( h \in G * tAt^{-1} \)
  \begin{multline*}
    f \in \left ( \overline{G* tAt^{-1}} \right) \setminus G*tAt^{-1} \implies fh, hf \in \left( \overline{G* tAt^{-1}}
    \right) \setminus G*tAt^{-1}.
  \end{multline*}
  Using this observation the case \( n = 2 \) follows from the case \( n = 1 \).  Indeed, \( n = 2 \) implies
  \( \alpha = g \concat t^{k} \) or \( \alpha = t^{k} \concat g \) for some \( g \in G \), \( k \ne 0 \).  Multiplying
  \( f \) by \( g^{-1} \) from either left or right brings us to the case \( n = 1 \).

  \textbf{Step of induction}. Without loss of generality we may assume that \( \alpha(n) = t^{k} \) for some
  \( k \ne 0 \).  Indeed, if \( \alpha(n) = g \) for some \( g \in G \), then we can substitute \( fg^{-1} \) for
  \( f \).  Assume that \(\alpha = \alpha_{0} \concat t^{k_{1}} \concat g \concat t^{k_{2}} \), where
  \( k_{1}, k_{2} \ne 0 \) and \( g \in G \).  We claim that \( k_{1} = 1 \), \( k_{2} = -1 \), and \( g \in A \).  Set
  \begin{align*}
    &\epsilon_{1} = \min\{d(\alpha(i),e) : i \in \seg{1}{n} \},\\
    &\epsilon_{2} =
    \begin{cases}
      1 & \textrm{if \( \forall i\ \alpha(i) \in G \implies \alpha(i) \in A \)},\\
      \min\{d(\alpha(i),A) : \alpha(i) \in G \setminus A\} & \textrm{otherwise}.
    \end{cases}
  \end{align*}
  And let \( \epsilon = \min\{1, \epsilon_{1},\epsilon_{2}\} \).  Note that \( \epsilon > 0 \).

  Since \( f \in \overline{G * tAt^{-1}} \), there is \( h \in G * tAt^{-1} \) such that \( d(f,h) < \epsilon \).
  Therefore there is a reduced simple \( fh^{-1} \)-pair \( (\beta,\xi) \) such that \( \rho(\beta,\xi) < \epsilon\).
  Let \( \gamma \) be the reduced form of \( h^{-1 }\).  Suppose first that \( k_{2} \ne -1 \).  Assume for simplicity
  that \( \beta = \alpha \concat \gamma\) (in general the first letter of \( \gamma \) may get canceled; the proof for
  the general case is the same, it is just notationally simpler to assume that \( \beta = \alpha \concat \gamma \)).
  Let \( T_{\xi} \) be the slim evaluation tree for \( \xi \), and let \( s_{0} \in T_{\xi} \) be such that
  \( n \in R_{s_{0}} \).

  We claim that \( n = m(R_{s_{0}}) \).  If this is not the case, then there is \( i_{0} \in R_{s_{0}} \) such that
  \( i_{0} < n \) and \( \seg{i_{0}+1}{n-1} \cap R_{s_{0}} = \emptyset \).  Since \( \alpha \) is reduced,
  \( i_{0} < n-1\).  If \( I = \seg{i_{0}+1}{n-1} \), then \( \hat{\xi}[I] = e \) and so there is \( j_{0} \in I \) such
  that \( \xi(j_{0}) = e \) (since otherwise \( \xi[I] \) would be reduced).  Therefore
  \[\rho(\beta,\xi) \ge d(\beta(j_{0}),\xi(j_{0})) = d(\alpha(j_{0}), e) \ge \epsilon_{1} \ge \epsilon. \]
  Contradicting the choice of the pair \( (\beta,\xi) \).

  Thus \( n = m(R_{s_{0}}) \).  Let \( j_{1}, \ldots, j_{p} \) be such that
  \begin{enumerate}[(i)]
  \item \( j_{k} \in R_{s_{0}} \) for all \( k \in \seg{1}{p} \);
  \item \( j_{k} < j_{k+1} \);
  \item \( \xi(j_{k}) \ne e \);
  \item \( \xi(j) \ne e \) and \( j \in R_{s_{0}} \) implies \( j = j_{k} \) for some \( k \).
  \end{enumerate}
  In fact, we can always modify the tree to assure that \( \xi(j) \ne e \) for all \( j \in R_{s_{0}} \), but this is
  not used here. In this notation \( j_{1} = n \).  Since \( \rho(\beta,\xi) < 1 \), we get
  \( \beta(j_{k}) = \xi(j_{k}) = t^{\pm 1} \) for all \( k \in \seg{2}{p} \).  If \( I_{k} = \seg{j_{k}+1}{j_{k+1}-1} \)
  for \( k \in \seg{1}{p-1} \), then \( \hat{\xi}[I_{k}] = e \) for all \( k \), whence for any \( k \in \seg{1}{p-1} \)
  \[ |\{ i \in I_{k} : \xi(i) = t\}| = |\{ i \in I_{k} : \xi(i) = t^{-1} \} |.\] 
  We claim that \( \xi(j_{2}) = t \).  Suppose not.  Then \( \xi(j_{2}) = t^{-1} \) and we can write \( \gamma =
  \gamma_{0}  \concat t^{-1} \concat \gamma_{1} \), 
  \[ \beta = \alpha_{0} \concat t^{k_{1}} \concat g \concat t^{k_{2}} \concat \gamma_{0} \concat t^{-1} \concat
  \gamma_{1}, \]
  with \( |\alpha| + |\gamma_{0}| = j_{2} - 1 \).  Since \( \hat{\gamma}_{0} = e \) we must have
  \[ |\{ i \in \seg{1}{|\gamma_{0}|} : \gamma_{0}(i) = t\}| = |\{ i \in \seg{1}{|\gamma_{0}|} : \gamma_{0}(i) = t^{-1}
  \} |.\] 
  On the other hand 
  \[  \gamma_{0} = g_{0}' \concat t \concat a_{1}' \concat t^{-1} \concat \cdots \concat t \concat a'_{m},  \]
  (\( g_{0}' \) may be absent) and each \( t \) is paired with \( t^{-1} \) except for the last one.  Therefore
  \[ |\{ i \in \seg{1}{|\gamma_{0}|} : \gamma_{0}(i) = t\}| = |\{ i \in \seg{1}{|\gamma_{0}|} : \gamma_{0}(i) = t^{-1}
  \} | + 1.\] 
  Contradiction.  Therefore \( \xi(j_{2}) = t \).  Similarly,  it is now easy to see that
  \[ \xi(j_{2}) = t,\ \xi(j_{3}) = t^{-1},\ \xi(j_{4}) = t,\ldots,\ \xi(j_{p}) = t^{((-1)^{p})}.  \]
  Finally, since \( \hat{\xi}[R_{s_{0}}] = e \), we get \( \xi(j_{1}) = t^{-1}\) or \( \xi(j_{1}) = e\), depending on
  whether \( p \) is even or odd.  But since by assumption \( k_{2} \ne 0 \) we get \( k_{2} = -1 \).

  We have proved that \( k_{2} = -1 \).  The next step is to show that \( g \in A \).  We have two cases.
  
  \setcounter{case}{0}
  \begin{case}
    \label{sec:metr-hnn-extens-1}
    \( \gamma(1) \in G \).  In this case we have \( \beta = \alpha \concat \gamma \).  Let \( s_{1} \in T_{\xi} \) be
    such that \( n-1 \in R_{s_{1}} \).  Similarly to the previous step one shows that \( n-1 = m(R_{s_{1}}) \).  Let
    \( R_{s_{1}} = \{j_{k}\}_{k=1}^{p }\), where \( j_{k} < j_{k+1} \).  In particular, \( n-1 = j_{1} \).  Set
    \( I_{k} = \seg{j_{k}+1}{j_{k+1}-1} \).  From \( \hat{\xi}[I_{k}] = e \) it follows
    \[ |\{ i \in I_{k} : \xi(i) = t\}| = |\{ i \in I_{k} : \xi(i) = t^{-1} \} |.\]
    Therefore \( \xi(j_{k}) \in A \) for all \( k \in \seg{2}{p} \).  And so \( \xi(j_{1}) \in A \) as well.  Finally,
    if \( g \not \in A \), then
    \[\rho(\beta,\xi) \ge d(\beta(n-1),\xi(n-1)) \ge d(g,A) \ge \epsilon_{2} \ge \epsilon.  \]
    And again we have a contradiction with the choice of \( (\beta,\xi) \).
  \end{case}

  \begin{case}
    \label{sec:metr-hnn-extens-2}
    \( \gamma(1) = t \).  In this case \( \alpha = \alpha_{0} \concat t^{k_{1}} \concat g \concat t^{-1} \) and
    \( \gamma = t \concat a \concat t^{-1} \concat \gamma_{0} \), for some \( a \in A \) and a word \( \gamma_{0} \).
    If \( g \not \in A \) then \( \beta = \alpha_{0} \concat t^{k_{1}} \concat ga \concat t^{-1} \concat \gamma_{0} \).
    And we are essentially in Case 1.  Therefore by the proof of Case 1 we get \( ga \in A \), but then \( g \in A \).
  \end{case}

  Thus \( g \in A \).  The proof of \( k_{1} = 1 \) is similar to the proof of \( k_{2} = -1 \) given earlier, and we
  omit the details.

  We have shown that \( \alpha = \alpha_{0} \concat t \concat a \concat t^{-1} \).  If \( f' = f t a^{-1} t^{-1} \),
  then \( \alpha_{0} \) is the reduced form of \( f' \) and \( f' \in \overline{G * tAt^{-1}} \setminus G*tAt^{-1} \).
  We proceed by induction on the length of \( \alpha \).
\end{proof}

\section{HNN extensions of groups with tsi metrics}

We now turn to the HNN construction itself.  There are several ways to build an HNN extension.  We will follow the
original construction of G. Higman, B. H. Neumann and H. Neumann from \cite{MR0032641}, because their approach hides a
lot of complications into the amalgamation of groups, and we have already constructed Graev metrics on amalgams in the
previous sections.

Let us briefly remind what an HNN extension is.  Let \( G \) be an abstract group, \( A, B < G \) be isomorphic
subgroups and \( \phi : A \to B \) be an isomorphism between them.  An HNN extension of \( (G,\phi) \) is a pair
\( (H, t) \), where \( t \) is a new symbol and \( H = \langle G, t | tat^{-1} = \phi(a), a \in A \rangle \).  The
element \( t \) is called a \emph{stable letter} of the HNN extension.

\subsection{Metrics on HNN extensions}
\label{sec:metr-hnn-extens}

\begin{theorem}
  \label{sec:hnn-extens-class-existence-of-hnn-extension}
  Let \( (G,d) \) be a tsi group, \( \phi : A \to B \) be a \( d \)-isometric isomorphism between the closed subgroups
  \( A, B \).  Let \( H \) be the HNN extension of \( (G,\phi) \) in the abstract sense, and let \( t \) be the stable
  letter of the HNN extension.  If \( \diam{A} \le K \), then there is a tsi metric \( \dist \) on \( H \) such that
  \( \dist|_{G} = d \) and \( \dist(t,e) = K \).
\end{theorem}

\begin{proof}
  First assume that \( K = 1 \). Let \( \langle u \rangle \) and \( \langle v \rangle \) be two copies of the group
  \( \mathbb{Z} \) of the integers with the usual metric. Form the free products \( (G * \langle u \rangle, d_{u}) \)
  and \( (G * \langle v \rangle,d_{v}) \), where \( d_{u}, d_{v} \) are the Graev metrics.  Since
  \( \diam{A} = \diam{B} \le 1 \), by Theorem \ref{sec:induced-metrics-equality-of-induced-metrics} the Graev metric on
  \( G * uAu^{-1} \) is the restriction of \( d_{u} \) onto \( G* uAu^{-1} \), and, similarly, the Graev metric on
  \( G * vBv^{-1} \) is just the restriction of \( d_{v} \).  Let \( \psi : G * uAu^{-1} \to G * vBv^{-1} \) be an
  isomorphism that is uniquely defined by
  \[ \psi(g) = g, \quad \psi(uau^{-1}) = v\phi(a)v^{-1}, \quad a \in A,\ g \in G.  \]
  By Theorem \ref{sec:induced-metrics-equality-of-induced-metrics} \( \psi \) is an isometry. Also, by Proposition
  \ref{sec:metr-hnn-extens-free-product-closed} \( G*uAu^{-1} \) and \( G*vBv^{-1} \) are closed subgroups of
  \(G * \langle u \rangle \) and \( G * \langle v \rangle \) respectively.  Hence by the results of Section
  \ref{sec:metrics-amalgams} we can amalgamate \( G * \langle u \rangle \) and \( G * \langle v \rangle \) over
  \( G * uAu^{-1} = G*vBv^{-1} \).  Denote the result of this amalgamation by \( (\widetilde{H}, \dist) \).  Then
  \[ uau^{-1} = v\phi(a)v^{-1} \quad \textrm{for all \( a \in A \)}, \]
  and therefore \( v^{-1}uau^{-1}v = \phi(a) \).  If \( H = \langle G, v^{-1}u \rangle \), then \( (H,v^{-1}u) \) is an
  HNN extension of \( (G,\phi) \) and \( \dist|_{H_{\phi}} \) is a two-sided invariant metric on \( H \), which extends
  \( d \).

  This was done under the assumption that \( K=1 \).  The general case can be reduced to this one.  If \( d' = (1/K)d
  \), then \( d' \) is a tsi metric on \( G \), \( \phi \) is a \( d' \)-isometric isomorphism and \(
  d'\)-\(\diam{A}\le 1 \).  By the above construction there is a tsi metric \( \dist' \) on \( H \) such that
  \( \dist'|_{G} = d' \).  Now set \( \dist = K \dist' \).
\end{proof}

It is, of course, natural to ask if the condition of having a bounded diameter is crucial.  The answer to this question
is not known, but here is a necessary condition.

\begin{proposition}
  \label{sec:hnn-extens-class-necessary-hnn-condition}
  Let \( (G,d) \) be a tsi group, \( \phi : A \to B \) be a \( d \)-isometric isomorphism, and \( H \) be the HNN
  extension of \( (G,\phi) \) with the stable letter \( t \).  If \( d \) is extended to a tsi metric \( d' \) on
  \( H \), then
  \[ \sup\{d'(a,\phi(a)) : a \in A\} < \infty.  \]
\end{proposition}

\begin{proof}
  If \( K = d'(t,e) \), then for any \( a \in A \)
  \begin{displaymath}
    \begin{aligned}
      d'(a,\phi(a)) =& d'(a,tat^{-1}) = d'(a^{-1}tat^{-1},e) =\\
      &d'(a^{-1}tat^{-1}, a^{-1}eae) \le d'(t,e)+ d'(t^{-1},e)=2K.
    \end{aligned}
  \end{displaymath}
  Therefore \( \sup\{d'(a,\phi(a)) : a \in A\} \le 2K \).
\end{proof}

\begin{question}
  \label{sec:hnn-extens-class-sufficiency-of-the-necessary-condition-q}
  Is this condition also sufficient?  To be precise, suppose \( (G,d) \) is a tsi group, \( \phi : A \to B \) is a
  \( d \)-isometric isomorphism between closed subgroups \( A, B \), and suppose that
  \[ \sup \big\{d(a,\phi(a)) : a \in A \big\} < \infty. \]
  Does there exist a tsi metric \( \dist \) on the HNN extension \( H \) of \( (G,\phi) \) such that
  \( \dist|_{G} = d \)?
\end{question}

\subsection{Induced conjugation and HNN extension}
\label{sec:induc-conj-hnn-1}

Recall that a topological group \( G \) is called SIN if for every open \( U \subseteq G \) such that \( e \in U \)
there is an open subset \( V \subseteq U \) such that \( gVg^{-1} = V \) for all \( g \in G \).  A metrizable group
admits a compatible two-sided invariant metric if and only if it is SIN.

\begin{theorem}
  \label{sec:induc-conj-hnn-general-theorem}
  Let \( G \) be a SIN metrizable group.  Let \( \phi : A \to B \) be a topological isomorphism between two closed
  subgroups.  There exist a SIN metrizable group \( H \) and an element \( t \in H \) such that \( G < H \) is a
  topological subgroup and \( tat^{-1} = \phi(a) \) for all \( a \in A \) if and only if there is a compatible tsi
  metric \( d \) on \( G \) such that \( \phi \) becomes a \( d \)-isometric isomorphisms.
\end{theorem}

\begin{proof}
  Necessity of the condition is obvious: if \( d \) is a compatible tsi metric on \( H \), then \( \phi \) is
  \( d|_{G} \)-isometric.  We show sufficiency.  Let \( d \) be a compatible tsi metric on \( G \) such that \( \phi \)
  is a \( d \)-isometric isomorphism.  If \( d'(g,e) = \min\{d(g,e), 1\} \), then \( d' \) is also a compatible tsi
  metric on \( G \), \( \phi \) is a \( d' \)-isometric isomorphism, and \( d' \)-\( \diam{A} \le 1\) (because \( d'
  \)-\( \diam{G} \le 1\)).  Apply Theorem \ref{sec:hnn-extens-class-existence-of-hnn-extension} to get an extension of
  \( d' \) to a tsi metric on \( H \), where \( (H,t) \) is the HNN extension of \( (G,\phi) \).  Then \( (H,t) \)
  satisfies the conclusions of the theorem.
\end{proof}

\begin{corollary}
  \label{sec:induc-conj-hnn-extension-for-discrete-subgroups}
  Let \( G \) be a SIN metrizable group.  Let \( \phi : A \to B \) be a topological group isomorphism.  If \( A \) and
  \( B \) are discrete, then there is a topology on the HNN extension of \( (G,\phi) \) such that \( G \) is a closed
  subgroup of \( H \) and \( H \) is SIN and metrizable.
\end{corollary}

\begin{proof}
  Let \( d \) be a compatible tsi metric on \( G \).  Since \( A \) and \( B \) are discrete, there exists constant
  \( c > 0 \) such that
  \[ \inf\{d(a_{1},a_{2}) : a_{1},a_{2} \in A, a_{1} \ne a_{2}\} \ge c, \quad
  \inf\{d(b_{1},b_{2}) : b_{1},b_{2} \in B, b_{1} \ne b_{2}\} \ge c.\]
  If \( d'(g_{1},g_{2}) = \min\{d(g_{1},g_{2}), c\} \), then \( d' \) is a
  compatible tsi metric on \( G \) and \( \phi \) is a \( d' \)-isometric
  isomorphism.  Theorem \ref{sec:induc-conj-hnn-general-theorem} finishes the
  proof.
\end{proof}

\begin{corollary}
  \label{sec:induc-conj-hnn-inverse-abelian-group}
  Let \( (G,+) \) be an abelian metrizable group.  If \( \phi : G \to G \) is given by \( \phi(x) = -x \), then
  there is a SIN metrizable topology on the HNN extension \( H \) of \( (G,\phi) \) that extends the topology of \( G
  \).
\end{corollary}

\begin{proof}
  If \( d \) is a compatible tsi metric on \( G \) such that \(d\)-\( \diam{G} \le 1\), then \( \phi \) is a \( d
  \)-isometric isomorphism and we apply Theorem \ref{sec:induc-conj-hnn-general-theorem}.
\end{proof}

\begin{definition}
  \label{sec:induc-conj-hnn}
  Let \( G \) be a topological group.  Elements \( g_{1}, g_{2} \in G \) are said to be \emph{induced conjugated} if
  there exist a topological group \( H \) and an element \( t \in H \) such that \( G < H\) is a topological subgroup
  and \( tg_{1}t^{-1} = g_{2} \).
\end{definition}

\begin{example}
  \label{sec:induc-conj-hnn-circle-hnn}
  Let \( (\mathbb{T},+) \) be a circle viewed as a compact abelian group, and let \( g_{1}, g_{2} \in \mathbb{T} \).
  The elements \( g_{1} \) and \( g_{2} \) are induced conjugated if and only if one of the two conditions is satisfied:
  \begin{enumerate}[(i)]
  \item \( g_{1} \) and \( g_{2} \) are periodic elements of the same period;
  \item \( g_{1} = \pm g_{2} \).
  \end{enumerate}
\end{example}

\begin{proof}
  The sufficiency of any of these conditions follows from Corollary
  \ref{sec:induc-conj-hnn-extension-for-discrete-subgroups} and Corollary
  \ref{sec:induc-conj-hnn-inverse-abelian-group}.  We need to show the necessity.  If \( g_{1} \) and \( g_{2} \) are
  induced conjugated, then they have the same order.  If the order of \( g_{i} \) is finite, we are done.  Suppose the
  order is infinite.  The groups \( \langle g_{1} \rangle \) and \( \langle g_{2} \rangle\) are naturally isomorphic (as
  topological groups) via the map \( \phi(kg_{1}) = kg_{2}\).  This map extends to a continuous isomorphism
  \( \phi : \mathbb{T} \to \mathbb{T} \), because \( \mathbb{T} \) is compact and \( \langle g_{i} \rangle \) is dense
  in \( \mathbb{T} \).  But there are only two continuous isomorphisms of the circle: \( \phi = \id \) and
  \( \phi = -\id \).  Thus \( g_{1} = \pm g_{2} \).
\end{proof}

\begin{example}
  \label{sec:induc-conj-hnn-infinite-torus-shift}
  Let \( G = \mathbb{T}^{\mathbb{Z}} \) be a product of circles, and let
  \( S : \mathbb{T}^{\mathbb{Z}} \to \mathbb{T}^{\mathbb{Z}} \) be the shift map \( S(x)(n) = x(n+1) \) for all
  \( x \in \mathbb{T}^{\mathbb{Z}}\) and all \( n \in \mathbb{Z} \).  The group \( \mathbb{T}^{\mathbb{Z}} \) is
  monothetic and abelian.  If \( x = \{a_{n}\}_{n \in \mathbb{Z}} \), where \( a_{n} \)'s and \( 1 \) are linearly
  independent over \( \mathbb{Q} \), then \( \langle x \rangle \) is dense in \( \mathbb{T}^{\mathbb{Z}} \) (by
  Kronecker's theorem, see, for example, \cite[Theorem 443]{MR2445243}).  Since
  \( S \) is an automorphism, \( x \) and \( S(x) \) are topologically similar.  We claim that \( x \) and \( S(x) \)
  are not induced conjugated in any SIN metrizable group \( H \).
\end{example}

\begin{proof}
  Suppose \( H \) is a SIN metrizable group, \( G \) is a topological subgroup of \( H \) and \( t \in H \) is such that
  \( txt^{-1} = S(x) \).  If \( \phi_{t} : H \to H \) is given by \( \phi_{t}(y) = tyt^{-1} \), then
  \( \phi_{t}(mx) = S(mx) \) for all \( m \in \mathbb{Z} \) and hence, by continuity and density of
  \( \langle x \rangle \), \( \phi_{t}(y) = S(y) \) for all \( y \in \mathbb{T}^{\mathbb{Z}} \).  If \( d \) is a
  compatible tsi metric on \( H \), then \( \phi_{t} \) is a \( d \)-isometric isomorphism.  Therefore for
  \( x_{0} \in \mathbb{T}^{\mathbb{Z}} \),
  \begin{displaymath}
    x_{0}(n) =
    \begin{cases}
      1/2 & \textrm{if \( n = 0 \)}; \\
      0 & \textrm{otherwise},
    \end{cases}
  \end{displaymath}
  we get
  \[ d(\phi_{t}^{m}(x_{0}), e) = d(\phi_{t}^{m}(x_{0}), \phi_{t}^{m}(e)) = d(x_{0}, e) = \mathrm{const} > 0, \]
  but \( S^{m}(x_{0}) \to 0 \), when \( m \to \infty \).  This contradicts \( \phi_{t}(y) = S(y) \) for all
  \( y \in \mathbb{T}^{\mathbb{Z}} \).
\end{proof}

\bibliographystyle{amsplain} 
\bibliography{/home/kslutsky/gitrep/papers/references}{}

\end{document}